\documentclass[reqno]{amsart}
\usepackage[left=1in,right=1in,top=1in,bottom=1in]{geometry}
 \usepackage{microtype}
\usepackage{hyperref}
         
\usepackage{amsmath}             
\usepackage{amsfonts}             
\usepackage{amsthm}               
\usepackage{amsbsy}
\usepackage{amssymb}
\usepackage{amsthm}
\usepackage[nobysame]{amsrefs}

\usepackage{enumerate}

\newtheorem{thm}{Theorem}[section]
\newtheorem{lem}[thm]{Lemma}
\newtheorem*{lem*}{Lemma}
\newtheorem{prop}[thm]{Proposition}
\newtheorem{cor}[thm]{Corollary}

\newtheorem{ques}[thm]{Question}

\theoremstyle{definition}
\newtheorem{defn}[thm]{Definition}

\newtheorem{rem}[thm]{Remark}
\newtheorem{exam}[thm]{Example}

\newcommand{\bC}{{\mathbb{C}}}
\newcommand{\bN}{{\mathbb{N}}}
\newcommand{\bR}{{\mathbb{R}}}
\newcommand{\bZ}{{\mathbb{Z}}}

\newcommand{\A}{{\mathcal{A}}}
\newcommand{\B}{{\mathcal{B}}}
\newcommand{\C}{{\mathcal{C}}}

\newcommand{\F}{{\mathcal{F}}}
\renewcommand{\H}{{\mathcal{H}}}

\newcommand{\J}{{\mathcal{J}}}

\renewcommand{\L}{{\mathcal{L}}}
\newcommand{\M}{{\mathcal{M}}}

\renewcommand{\P}{{\mathcal{P}}}

\newcommand{\U}{{\mathcal{U}}}

\renewcommand{\phi}{\varphi}

\newcommand{\fM}{{\mathfrak{M}}}

\newcommand{\qand}{\quad\text{and}\quad}
\newcommand{\qqand}{\qquad\text{and}\qquad}

\newcommand{\alg}{\mathrm{alg}}

\newcommand{\Tr}{\mathrm{Tr}}
\renewcommand{\th}{\mathrm{th}}
\newcommand{\diag}{\mathrm{diag}}
\newcommand{\op}{\mathrm{op}}

\newcommand{\sa}{\mathrm{sa}}
\newcommand{\smbfe}{{X_1,\ldots, X_n \sqcup Y_1, \ldots, Y_m}}
\newcommand{\orb}{\mathrm{orb}}

\newcommand{\set}[1]{\left\{#1\right\}}
\newcommand{\ang}[1]{\left\langle#1\right\rangle}
\newcommand{\paren}[1]{\left(#1\right)}
\newcommand{\abs}[1]{\left|#1\right|}
\newcommand{\norm}[1]{\left\|#1\right\|}

\usepackage{tikz}
\usetikzlibrary{shapes,snakes,calc,arrows}
\usetikzlibrary{decorations.pathreplacing,shapes.geometric}
\usetikzlibrary{calc,positioning}
\tikzset{Box/.style={very thick, rounded corners}}
\tikzset{marked/.style={star, star point height = .75mm, star points =5, fill=black,minimum size=2mm, inner sep=0mm} }
\tikzset{verythickline/.style = {line width=7pt}}
\tikzset{thickline/.style = {line width=5pt}}
\tikzset{medthick/.style = {line width=3pt}}
\tikzset{med/.style = {line width=2pt}}
\tikzset{count/.style = {fill=white,circle,draw,thin, inner sep=2pt}}
\tikzset{rcount/.style = {fill=white,rectangle,draw,thin,inner sep=2pt, rounded corners}}
\tikzset{cpr/.style = {draw,fill=white,rectangle,thin, rounded corners}}

\definecolor{ggreen}{HTML}{7FDD99}

\begin{document}

\nocite{*}

\title[Bi-Free Entropy: Microstates]{Analogues of Entropy in Bi-Free Probability Theory: Microstates}

\author{Ian Charlesworth}
\address{Department of Mathematics, University of California, Berkeley, California, 94720, USA}
\email{ilc@math.berkeley.edu}

\author{Paul Skoufranis}
\address{Department of Mathematics and Statistics, York University, 4700 Keele Street, Toronto, Ontario, M3J 1P3, Canada}
\email{pskoufra@yorku.ca}

\subjclass[2010]{46L54, 46L53, 47B80, 94A17}
\date{\today}
\keywords{bi-free probability, entropy}
\thanks{The research of the second author was supported in part by NSERC (Canada) grant RGPIN-2017-05711.}

\begin{abstract}
In this paper, we extend the notion of microstate free entropy to the bi-free setting.  In particular, using the bi-free analogue of random matrices, microstate bi-free entropy is defined.  Properties essential to an entropy theory are developed, such as the behaviour of the entropy when transformations on the left variables or on the right variables are performed.  In addition, the microstate bi-free entropy is demonstrated to be additive over bi-free collections provided additional regularity assumptions are included and is computed for all bi-free central limit distributions.  Moreover, an orbital version of bi-free entropy is examined which provides a tighter upper bound for the subadditivity of microstate bi-free entropy and provides an alternate characterization of bi-freeness in certain settings.
\end{abstract}

\maketitle

\section{Introduction}
\label{sec:Intro}

In a series of revolutionary papers \cites{V1993, V1994, V1996, V1997, V1998-2, V1999}, Voiculescu developed free probability analogues of the notions of entropy and Fisher's information.  In particular \cite{V1994} introduced a microstate notion of free entropy. In this setting `microstates' refers to approximating the distribution of self-adjoint operators in a tracial von Neumann algebra using matrix algebras.  The notion of microstate free entropy led to many important results pertaining to free group factors, such as the absence of Cartan subalgebras \cite{V1996}, the absence of simple maximal abelian self-adjoint algebras \cite{G1997}, and the free group factors being prime \cite{G1998}.  Alternatively, an infinitesimal version of free entropy based on derivations developed in \cite{V1998-2} has also led to many developments.

Recently in \cite{V2014} Voiculescu extended the notion of free probability to simultaneously study the left and right actions of algebras on reduced free product spaces.  This so-called bi-free probability has attracted the attention of many researchers and has had numerous developments (e.g. \cites{BBGS2017, C2016, CNS2015-1, CNS2015-2, S2016-2, S2016-3, S2016-4}).  The interest surrounding bi-free probability stems from the possibility of extending the techniques of free probability to solve problems pertaining to pairs of von Neumann algebras, such as a von Neumann algebra and its commutant or the tensor product of two von Neumann algebras.

One important development in bi-free probability theory was a bi-free analogue of the connection between free probability and random matrix theory exhibited in \cites{S2016-2, S2016-3, S2016-4}.  As microstate free entropy was motivated by the connection between free probability and random matrix theory, in this paper we use the bi-free matrix models of \cites{S2016-2, S2016-3, S2016-4} to develop a notion of microstate bi-free entropy.  In our sister paper \cite{CS2017} a notion of non-microstate bi-free entropy is developed.

In addition to this introduction, this paper contains nine sections which are organized as follows.  In Section \ref{sec:Defn} we define our microstate version of bi-free entropy (Definition \ref{defn:micro-bi-free}).  This notion of entropy only applies in the tracially bi-partite setting: that is, when the left algebra commutes with the right algebra, and the state becomes tracial when restricted to the left algebra or the right algebra.  Although bi-free probability theory extends beyond the tracially bi-partite setting, many natural examples are tracially bi-partite such as pairs consisting of a type II$_1$ factor whose commutant is a type II$_1$ factor with the tracial states occurring via the same vector state from the $L_2$-space of some tracial von Neumann algebra.  Section \ref{sec:Defn} also demonstrates this notion of microstate bi-free entropy satisfies many of the natural properties of an entropy theory.

In Section \ref{sec:Trans} an analysis of how transformations affect microstate bi-free entropy is performed.  If the transformation modifies only the left variables or only the right variables, microstate bi-free entropy behaves identically to how microstate free entropy behaves. However, the behaviour of microstate bi-free entropy when a transformation mixes left and right variables is currently unknown.  This is unsurprising as such a mixing destroys the distinction of left and right variables, and so is not easy to view as a natural bi-free operation.

In Section \ref{sec:Addi} it is demonstrated, under the assumption of the existence of microstates of all orders and a limit condition, that the microstate bi-free entropy of bi-free collections is the sum of the bi-free entropies (Theorem \ref{thm:micro-bi-free-additive}).
Assuming the existence of microstates of all orders is currently a necessity for the analogous result for free entropy with the general case being at partially addressed in works such as \cite{D2016}.

In Section \ref{sec:Orbital} an orbital version of bi-free entropy is examined in a similar fashion to the orbital free entropy from \cite{U2014}.  In particular, two characterizations of orbital bi-free entropy are given and the base properties are demonstrated.  Furthermore, Theorem \ref{thm:orbital-subadditive-with-micro} provides a better bound for the difference between the joint microstate bi-free entropy and the sum of the individual microstate bi-free entropies.

In Section \ref{sec:Characterization} Theorem \ref{thm:orbital-bi-free-characterization} is demonstrated, which characterizes when pairs of algebras with finite-dimensional approximants are bi-free in terms of the orbital bi-free entropy.  In addition, it is shown in Corollary \ref{cor:additive-implies-bi-free}  that if collections of left and right operators have finite microstate bi-free entropy and the joint bi-free entropy is the sum of the individual bi-free entropies, then the collections are bi-free.

In Section \ref{sec:Calc} computations pertaining to microstate bi-free entropy are performed.  In particular, the value of the microstate bi-free entropy is computed for all finite bi-free central limit distributions.  This computation is non-trivial due to the same complications as in Section \ref{sec:Trans}.  It is worthy to note that the microstate bi-free entropy for bi-free central limit distributions has the same form as Gaussian distributions with respect to the Shannon entropy and the free central limit distributions with respect to free entropy.  Furthermore, the same value is obtained for non-microstate bi-free entropy in our sister paper \cite{CS2017}.

In Section \ref{sec:Entropy-Dimension} we develop the notion of microstate bi-free entropy dimension and show that for a bi-free central limit distribution pair that this dimension is equal to the dimension of the support of their joint distribution.  In Section \ref{sec:Gen} we discuss generalizing this microstate version of bi-free entropy to non-bi-partite systems and the resulting complications.  Finally, in Section \ref{sec:Ques}, several open questions are discussed, most of which might be possible to solve from a deeper understanding of the structure of free and/or bi-free microstates.

Note it is not the intent of this paper to reprove every single fact about microstate free entropy to the bi-free setting, but show most of the base and some interesting results carry-forward.

\section{Definition and Basic Properties}
\label{sec:Defn}

In \cite{V1991} Voiculescu observed a connection between random matrix theory and free probability.  Specifically it was demonstrated that the eigenvalue distribution of certain random matrices asymptotically tended to the free central limit distributions, and random matrices with independent entries tended in law to freely independent operators.  However other distributions can be approximated using the eigenvalues of matrices.  In an attempt to understand these approximations, Voiculescu introduced the notion of free entropy defined as follows.

\begin{defn}[\cite{V1994}]
\label{defn:micro-free}
Let $(\fM, \tau)$ be a tracial von Neumann algebra and let $X_1, \ldots, X_n \in \fM$ be self-adjoint operators.  Let $(\M_d, \tau_d)$ denote the tracial von Neumann algebra consisting of the $d \times d$ complex matrices with the normalized trace $\tau_d$.  We will use $\Tr_d$ to denote the unnormalized trace on $\M_d$ and $\M_d^{\sa}$ to denote the self-adjoint elements of $\M_d$.

For $M,d \in \bN$ and $R, \epsilon> 0$, let $\Gamma_R(X_1, \ldots, X_n; M, d, \epsilon)$ denote the set of all $n$-tuples $(A_1, \ldots, A_n) \in (\M_d^{\sa})^n$ such that $\left\|A_j\right\| \leq R$ for all $1 \leq j \leq n$ and
\[
\left|\tau(X_{i_1} \cdots X_{i_p}) - \tau_d(A_{i_1} \cdots A_{i_p})\right| < \epsilon
\]
for all $i_1, \ldots, i_p \in \{1,\ldots, n\}$ and $1 \leq p \leq M$.  Subsequently, if $\lambda_{d,n}$ denotes the Lebesgue measure on $(\M_d^{\sa})^n$ where $(\M_d^{\sa})^n$ is equipped with the Hilbert-Schmidt norm
\[
\left\|(A_1, \ldots, A_n)\right\|_{\mathrm{HS}} = \Tr_d(A_1^2 + \cdots + A_n^2),
\]
define
\begin{align*}
\chi_R(X_1, \ldots, X_n; M, d, \epsilon) &= \log\left( \lambda_{d,n}\left(\Gamma_R(X_1, \ldots, X_n; M, d, \epsilon)  \right)   \right) \\
\chi_R(X_1, \ldots, X_n; M, \epsilon) &= \limsup_{d \to \infty} \frac{1}{d^2}\chi_R(X_1, \ldots, X_n; M, d, \epsilon) + \frac{1}{2} n \log(d) \\
\chi_R(X_1, \ldots, X_n) &=  \inf\{\chi_R(X_1, \ldots, X_n; M, \epsilon) \, \mid \, M \in \bN, \epsilon > 0\}, \text{ and}\\
\chi(X_1, \ldots, X_n) &= \sup_{R > 0} \chi_R(X_1, \ldots, X_n).
\end{align*}
The quantity $\chi(X_1, \ldots, X_n) \in [-\infty, \infty)$ is called the \emph{free entropy of $X_1, \ldots, X_n$}.  The reason for the constants and various normalizations can be seen in \cite{V1994} or the computations in Section \ref{sec:Calc}.
\end{defn}

As even some bi-free central limit distributions fail to be tracial (see, e.g., \cite{C2016}*{Example 11}) we must replace microstates with a version which can approximate non-tracial distributions in order to deal with the bi-free setting.
Rather than allow arbitrary non-tracial states on the matrices, though, we seek to progress in a way that recognizes the distinction between left and right variables.
This leads us to the idea of microstates consisting of bounded linear maps on $\M_d$ given by left and right matrix multiplication operators; that is, for $A \in \M_d$, we define $L(A)$ and $R(A)$ to be the bounded linear maps on $\M_d$ defined by
\[
L(A)B = AB \qqand R(A)B = BA.
\]
We then equip the bounded linear maps on $\M_d$ with the state $\tau_d(\cdot I_d)$ which evaluates the linear maps when applied to the identity matrix and then computes the trace of the result.

Of course, these choices force some restrictions upon us.  In particular, as left matrix multiplication commutes with right matrix multiplication, we can only find microstates for so-called bi-partite families where all left variables commute with all right variables (in distribution). Furthermore, $\tau_d(\cdot I_d)$ is tracial when restricted to left multiplication operators or right multiplication operators, so we will only be able to produce microstates for distributions having this property.
We shall refer to systems satisfying the above as \emph{tracially bi-partite}, and give some indication of how to broaden this setting in Section \ref{sec:Gen}.

\begin{defn}\label{defn:micro-bi-free}
	Let $(\A, \varphi)$ be a C$^*$-non-commutative probability space and let $X_1, \ldots, X_n, Y_1, \ldots, Y_m$ be self-adjoint operators in $\A$.
	For $M, d \in \bN$ and $R, \epsilon > 0$, let $\Gamma_R(\smbfe; M, d, \epsilon)$ denote the set of all $(n+m)$-tuples $(A_1, \ldots, A_n, B_1, \ldots, B_m) \in \paren{\M_d^\sa}^{n+m}$ such that $\norm{A_i}, \norm{B_j} \leq R$ for all $1 \leq i \leq n$ and $1 \leq j \leq m$, such that
	\[
		\abs{\varphi(Z_{k_1}\cdots Z_{k_p}) - \tau_d(C_{k_1}\cdots C_{k_p}(I_d))} < \epsilon
	\]
	for all $1 \leq p \leq M$ and $k_1, \ldots, k_p \in \set{1, \ldots, n+m}$, where
	\[
		Z_{k} = \begin{cases}
			X_k & \text{if } k \in\{1\ldots, n\} \\
			Y_{k-n} & \text{if }k \in \{n+1,\ldots, n+m\}
		\end{cases}
		\qqand 
		C_{k} = \begin{cases}
			L\paren{A_k} & \text{if } k \in\{1\ldots, n\} \\
			R\paren{B_{k-n}} & \text{if }k \in \{n+1,\ldots, n+m\}
		\end{cases} \in B(\M_d).
	\]

	With $\lambda_{d, p}$ still standing for the Lebesgue measure on $\paren{\M_d^\sa}^p$ equipped with the Hilbert-Schmidt norm, we successively define
	\begin{align*}
		\chi_R(\smbfe; M, d, \epsilon) &= \log\left( \lambda_{d,n+m}\left(\Gamma_R(\smbfe; M, d, \epsilon)  \right)   \right) \\
		\chi_R(\smbfe; M, \epsilon) &= \limsup_{d \to \infty} \frac{1}{d^2}\chi_R(\smbfe; M, d, \epsilon) + \frac{1}{2} (n+m) \log(d) \\
		\chi_R(\smbfe) &=  \inf\{\chi_R(\smbfe; M, \epsilon) \, \mid \, M \in \bN, \epsilon > 0\}, \text{ and} \\
		\chi(\smbfe) &= \sup_{R > 0} \chi_R(\smbfe).
	\end{align*}
	The quantity $\chi(\smbfe)$ will be called the \emph{microstate bi-free entropy of $X_1, \ldots, X_n \sqcup Y_1, \ldots, Y_m$}.  We will see in Proposition \ref{prop:not-plus-infinity} that $\chi(\smbfe) \in [-\infty, \infty)$.
\end{defn}

\begin{rem}
\label{rem:only-a-certain-order-of-microstates-matters}
By analysing the joint distribution of $L(A_1), \ldots, L(A_n), R(B_1),\ldots, R(B_n)$ and the definition of $\chi(\smbfe)$, we see that $\chi(\smbfe) = -\infty$ unless we are in the tracially bi-partite setting.  We will make this the standing assumption until Section \ref{sec:Gen} of the paper.
This is a setting that many canonical examples fit into and thus is of great interest.
Note we will not assume that $\varphi$ is tracial on $\A$ nor faithful on $\A$ as these properties need not occur in most bi-free systems (see \cite{BBGS2017} and \cite{R2017} respectively).

Using the fact that the system is bi-partite, the definition of $\Gamma_R(\smbfe; M, d, \epsilon)$ may be simplified slightly, as it is enough to check that only certain moments are well-approximated: indeed, for $M,d \in \bN$ and $R, \epsilon> 0$ notice $\Gamma_R(\smbfe; M, d, \epsilon)$ is the set of all $(n+m)$-tuples $(A_1, \ldots, A_n, B_1, \ldots, B_m) \in (\M_d^{\sa})^{n+m}$ such that $\left\|A_i\right\|, \left\|B_j\right\| \leq R$ for all $1 \leq i \leq n$ and $1 \leq j \leq m$, and
\[
\left|\varphi(X_{i_1} \cdots X_{i_p}Y_{j_1} \cdots Y_{j_q}) - \tau_d(A_{i_1} \cdots A_{i_p} B_{j_q} \cdots B_{j_1})\right| < \epsilon
\]
for all $i_1, \ldots, i_p \in \{1,\ldots, n\}$ and $j_1, \ldots, j_q \in \{1,\ldots, m\}$ with $p+ q \leq M$.
\end{rem}

\begin{rem}
\label{rem:monobifree-is-free}
It is elementary to see based on the definition of microstate bi-free entropy that if $m = 0$ then
\[
\chi(\smbfe) = \chi(X_1, \ldots, X_n),
\]
whence the above notion of bi-free entropy is an extension of microstate free entropy.
Further, it can be readily verified that
\[
	(A_1, \ldots, A_n, B_1, \ldots, B_m) \in \Gamma_R(\smbfe; M, d, \epsilon)
\]
if and only if
\[
	(B_1^t, \ldots, B_m^t, A_1^t, \ldots, A_n^t) \in \Gamma_R(Y_1, \ldots, Y_m \sqcup X_1, \ldots, X_n; M, d, \epsilon).
\]
It follows that $\chi(\smbfe) = \chi(Y_1, \ldots, Y_m \sqcup X_1, \ldots, X_n)$ as transpose preserves Lebesgue measure, and in particular when $n = 0$ we have
\[
	\chi(\smbfe) = \chi(Y_1, \ldots, Y_m).
\]
\end{rem}

\begin{prop}
\label{prop:micro-subadditive}
If $0 \leq p \leq n$ and $0 \leq q \leq m$ then
\begin{align*}
\chi(\smbfe) &\leq  \chi(X_1, \ldots, X_p\sqcup Y_1,\ldots Y_q) + \chi(X_{p+1}, \ldots, X_n\sqcup Y_{q+1},\ldots Y_m).
\end{align*}
In particular,
\[
\chi(\smbfe) \leq \chi(X_1, \ldots, X_n) + \chi(Y_1, \ldots, Y_m).
\]
\end{prop}
\begin{proof}
First note that the inequality will be demonstrated provided we can show that
\begin{align*}
\chi_R &(\smbfe; M, d, \epsilon) \\
&\leq \chi_R(X_1, \ldots, X_p\sqcup Y_1,\ldots Y_q; M, d, \epsilon) + \chi_R(X_{p+1}, \ldots, X_n\sqcup Y_{q+1},\ldots Y_m; M, d, \epsilon)
\end{align*}
for all $M, d$, and $\epsilon$.  Since by definitions we have that
\begin{align*}
\Gamma_R &(\smbfe; M, d, \epsilon)  \\
& \subseteq \Gamma_R(X_1, \ldots, X_p\sqcup Y_1,\ldots Y_q; M, d, \epsilon) \times_{\ell r} \Gamma_R(X_{p+1}, \ldots, X_n\sqcup Y_{q+1},\ldots Y_m; M, d, \epsilon)
\end{align*}
where 
\[
(A_1, \ldots, A_p, B_1, \ldots, B_q) \times_{\ell r} (A_{p+1}, \ldots, A_n, B_{q+1},\ldots, B_m) = (A_1, \ldots, A_n, B_1, \ldots, B_m),
\]
clearly the above inequalities hold.
\end{proof}

These inequalities allow us to import upper bounds on entropy from the free case.  In particular, we learn that the bi-free entropy never takes the value $+\infty$.

\begin{prop}
	\label{prop:not-plus-infinity}
	Let $C^2 = \varphi(X_1^2+\cdots+X_n^2+Y_1^2+\cdots+Y_m^2)$.
	Then
\[\chi(X_1, \ldots, X_n \sqcup Y_1, \ldots, Y_m) \leq \frac{n+m}{2}\log\left(\frac{2\pi e}{n+m}C^2\right).\]
\end{prop}

\begin{proof}
	We recall that the analogous free statement was shown in \cite{V1994}*{Proposition 2.2}.
	Let 
	\[
	C_X^2 = \varphi(X_1^2 + \cdots + X_n^2) \qqand C_Y^2 = \varphi(Y_1^2+\cdots+Y_m^2).
	\]	
	Using the above, Proposition \ref{prop:micro-subadditive}, and the concavity of the logarithm, we obtain that
	\begin{align*}
		\chi(X_1, \ldots, X_n \sqcup Y_1, \ldots, Y_m)
		&\leq \chi(X_1, \ldots, X_n) + \chi(Y_1, \ldots, Y_m)\\
		&\leq \frac{1}{2} n \log\left(\frac{2\pi e}{n} C_X^2\right) + \frac{1}{2} m \log\left(\frac{2\pi e}{m} C_Y^2\right) \\
		&= \frac{n+m}{2}\left(\frac{n}{n+m} \log\left(\frac{2\pi e}{n} C_X^2\right) + \frac{m}{n+m} \log\left(\frac{2\pi e}{m} C_Y^2\right)\right) \\
		&\leq \frac{n+m}{2}  \log\left(\frac{n}{n+m}\frac{2\pi e}{n}C_X^2 + \frac{m}{n+m}\frac{2\pi e}{m}C_Y^2\right) \\
		&= \frac{n+m}{2}\log\left(\frac{2\pi e}{n+m}C^2\right). \qedhere
	\end{align*}
\end{proof}

There is a more interesting inequality relating the microstate bi-free entropy to the microstate free entropy.  In particular, the microstate bi-free entropy is bounded below by the microstate free entropy obtained by changing all of the right variables to left variables.

\begin{thm}
\label{thm:micro-converting-rights-to-lefts}
Let $\left(\{X_i\}^n_{i=1}, \{Y_j\}^m_{j=1}\right)$ be tracially bi-partite, self-adjoint operators in a C$^*$-non-commutative probability space $(\A, \varphi)$.  Suppose there exists another C$^*$-non-commutative probability space $(\A_0, \tau_0)$  and self-adjoint operators $X'_1, \ldots, X'_n, Y'_1, \ldots, Y'_m \in \A_0$ such that $\tau_0$ is tracial on $\A_0$ and
\[
\varphi(X_{i_1} \cdots X_{i_p} Y_{j_1} \cdots Y_{j_q}) = \tau_0(X'_{i_1} \cdots X'_{i_p} Y'_{j_q} \cdots Y'_{j_1})
\]
for all $p,q \in \bN \cup \{0\}$, $i_1, \ldots, i_p \in \{1,\ldots, n\}$, and $j_1, \ldots, j_q \in \{1, \ldots, m\}$. 
 Then
\[
\chi(X'_1, \ldots, X'_n, Y'_1, \ldots, Y'_m) \leq \chi(\smbfe).
\]
\end{thm}
\begin{proof}
	Using the characterization from the end of Remark \ref{rem:only-a-certain-order-of-microstates-matters}, we see that
\[
	\Gamma_R(X'_1, \ldots, X'_n, Y'_1, \ldots, Y'_m; M, d, \epsilon) \subseteq \Gamma_R(\smbfe; M, d, \epsilon),
\]
and hence
\[
	\chi(X'_1, \ldots, X'_n, Y'_1, \ldots, Y'_m) \leq \chi(\smbfe). \qedhere
\]
\end{proof}

This inequality, in essence, arises because the set of bi-free microstates is defined with fewer conditions than the set of free microstates.  In addition, as we need only specify certain moments for the ``one-sided'' family for a given pair of faces and as many of the moments can be chosen somewhat arbitrarily, Theorem \ref{thm:micro-converting-rights-to-lefts} provides many possible lower bounds.

\begin{exam}
\label{exam:changing-semis-to-free}
For an example application of Theorem \ref{thm:micro-converting-rights-to-lefts}, let $\F(\H)$ be the Fock space on a real Hilbert space $\H$, let $e_1, e_2\in \H$ be unit vectors, and let $S_1 = l(e_1) + l^*(e_1)$ and $D_2 = r(e_2) + r^*(e_2)$, where $l$ and $l^*$ are the left creation/annihilation operators respectively, and $r$ and $r^*$ are the right creation/annihilation operators respectively.  If $c = \langle e_1, e_2\rangle$ and if $S_2 = l(e_2) + l^*(e_2)$, then Theorem \ref{thm:micro-converting-rights-to-lefts} implies that
\[
\chi(S_1 \sqcup D_2) \geq \chi(S_1, S_2).
\]
Notice that if $c \in (-1, 1)$ then
\[
e_3 := \frac{1}{1-c^2}\left(e_2 - c e_1\right)
\]
is a unit vector orthogonal to $e_1$, and so if $S_3 = l(e_3) + l^*(e_3)$, then $S_1$ and $S_3$ are freely independent centred semicircular variables of variance one while
\[
\begin{bmatrix}
1 & 0 \\ -c & \sqrt{1-c^2}
\end{bmatrix} \begin{bmatrix}
S_1 \\ S_3
\end{bmatrix} = \begin{bmatrix}
S_1 \\ S_2
\end{bmatrix}.
\]
Therefore, by \cite{V1993}*{Proposition 3.5 and Proposition 5.4} (or the analogous Proposition \ref{prop:transforms} in this paper), we obtain that
\begin{align*}
\chi(S_1 \sqcup D_2) \geq \chi(S_1, S_2) &= \chi(S_1, S_3) + \log\left(\left| \det\left(\begin{bmatrix}
1 & 0 \\ -c & \sqrt{1-c^2}
\end{bmatrix}\right)  \right|  \right) \\
&= \chi(S_1) + \chi(S_3) + \log(\sqrt{1-c^2}) \\
&= 2\chi(S_1) + \frac12\log\paren{1-c^2}.
\end{align*}
It will be shown in Theorem \ref{thm:entropy-pair-of-semis} that this inequality is actually an equality.
\end{exam}

Like with free entropy, the upper bound on the norm of microstates $R$ can be controlled.

\begin{prop}
\label{prop:R-does-not-matter}
Let 
\[
	\rho = \max\paren{\{\left\|X_i\right\| \, \mid \, 1 \leq i \leq n\} \cup \{\left\|Y_j\right\| \, \mid \, 1 \leq j \leq m\}}.
\]
If $R_2 > R_1 > \rho$, then 
\[
\chi_{R_2}(\smbfe) = \chi_{R_1}(\smbfe).
\]
In particular, for all $R > \rho$, 
\[
\chi_R(\smbfe) = \chi(\smbfe).
\]
\end{prop}
\begin{proof}
As $\chi_R$ is an increasing function of $R$, it suffices to prove the first equality.  The proof of said equality will follow a similar proof to that of \cite{V1994}*{Proposition 2.4}.

Fix $R_2 > R_1 > R_0 > \rho$ and define $g : [-R_2, R_2] \to [-R_1, R_1]$ to be the function which is linear on $[-R_2, -R_0]$, $[-R_0, R_0]$, and $[R_0, R_2]$, and such that $g(-R_2) = -R_1$, $g(-R_0) = -R_0$, $g(R_0) = R_0$, and $g(R_2) = R_1$.  Furthermore, for $A_1,\ldots, A_n, B_1, \ldots, B_m \in \M_d^{\sa}$ with $\left\|A_i\right\| \leq R_2$ and $\left\|B_j\right\| \leq R_2$, let
\[
G(A_1, \ldots, A_n, B_1, \ldots, B_m) = (g(A_1), \ldots, g(A_n), g(B_1), \ldots, g(B_m)).
\]

Given $M \in \bN$ and $\epsilon > 0$, it is not difficult to see that there exists an $M_1 \geq M$ and a $0 < \epsilon_1 < \epsilon$ such that
\[
G(\Gamma_{R_2}(\smbfe; M_1, d, \epsilon_1)) \subseteq \Gamma_{R_1}(\smbfe; M, d, \epsilon)
\]
for all $d \in \bN$.  Indeed for any 
\[
(A_1,\ldots, A_n, B_1, \ldots, B_m) \in \Gamma_{R_2}(\smbfe; M_1, d, \epsilon_1)
\]
we obtain that
\[
|\tau_d(A_i^p)|, |\tau_d(B_j^p)| \leq \rho^p + \epsilon_1
\]
for all $1 \leq i \leq n$, $ 1 \leq j \leq m$, and $1 \leq p \leq M_1$.
Thus given $\delta > 0$, choosing $M_1$ large and $\epsilon_1$ small enough yields
\[
\tau_d(P_{[-R_2, -R_0] \cup [R_0, R_2]}(A_i)), \tau_d(P_{[-R_2, -R_0] \cup [R_0, R_2]}(B_j)) < \delta
\]
where $P_{[-R_2, -R_0] \cup [R_0, R_2]}$ is denoting the spectral projection onto $[-R_2, -R_0] \cup [R_0, R_2]$, and thus can be selected even smaller still to make
\[
\left\|g(A_i) - A_i\right\|_1, \left\|g(B_j) - B_j\right\|_1 < \delta
\]
for all $1 \leq i \leq n$ and $1 \leq j \leq m$ independent of $d$.  As $M$ and $R_2$ are fixed, by selecting $\delta$ sufficiently small we obtain that the trace of any word of length at most $M$ in $g(A_1), \ldots, g(A_n), g(B_1), \ldots, g(B_m)$ is within a function of $\delta$, $M$, and $R_2$ which tends to 0 as $\delta$ tends to 0 to the trace of the corresponding word in $A_1, \ldots, A_n, B_1, \ldots, B_m$.  Thus the claim follows.

To complete the proof, it will suffice to obtain a specific lower bound on the Jacobian of $G$ on  
\[
\Gamma_{R_2}(\smbfe ; M_1, d, \epsilon_1).
\]
Let $U(d)$ denote the set of unitary elements of $\M_d$ and consider the change of coordinates from $\M_d^\sa$ to $(U(d)/ \mathbb{T}) \times \{(c_1, \ldots, c_d) \in \bR^d \, \mid \, c_1 < \cdots < c_d\}$ (where $\mathbb{T}$ is the torus of diagonal unitaries) defined by $(U, D) \mapsto U^*DU$ where $D = \diag(c_1,\ldots, c_d)$.  This change of coordinates places the Lebesgue measure in the form
\[
K \left(\prod_{1 \leq i < j \leq d} (c_i - c_j)\right) \,d\gamma_{d, 0} d\lambda_d
\]
where $K$ is a  normalizing constant and $\gamma_{d, 0}$ is the Haar measure on $U(d)/\mathbb{T}$.  The absolute value of the Jacobian of the map $C \mapsto g(C)$ is easily seen to be
\[
g'(c_1) \cdots g'(c_d) \prod_{1 \leq i < j \leq d} \frac{g(c_i) - g(c_j)}{c_i - c_j} 
\]
when $C$ has eigenvalues $c_1, \ldots, c_d$ and $c_k \neq \pm R_0$ for all $k$.  

Let $\delta > 0$ be arbitrary.  If $M_1$ is large enough and $\epsilon_1$ is small enough, we obtain that
\[
\tau_d(P_{[-R_2, -R_0] \cup [R_0, R_2]}(C)) < \delta
\]
and thus we obtain 
\[
\left(\frac{R_1 - R_0}{R_2 - R_0}\right)^{d+d^2-(d(1-\delta))^2}
\]
as a lower bound for the Jacobian of $g$ on a coordinate projection of $\Gamma_{R_2}(\smbfe; M_1, d, \epsilon_1)$.  In particular a lower bound for the Jacobian of $G$ can be obtained by taking the above lower bound for the Jacobian of $g$ raised to the $(n+m)^\th$ power and thus
\begin{align*}
\chi_{R_1}(\smbfe; m, d,\epsilon)& \geq \chi_{R_2}(\smbfe; M_1, d, \epsilon_1) \\ & \qquad+ (n+m)(d + d^2(2\delta - \delta^2)) \log\left( \frac{R_1 - R_0}{R_2 - R_0}  \right).
\end{align*}
Hence it follows that
\begin{align*}
\chi_{R_1}(\smbfe; m, \epsilon) &\geq \chi_{R_2}(\smbfe; M_1,\epsilon_1)  \\
& \qquad+ (n+m)(2\delta - \delta^2) \log\left( \frac{R_1 - R_0}{R_2 - R_0}  \right).
\end{align*}
Therefore, as $\delta > 0$ was arbitrary, the result follows.
\end{proof}

\begin{rem}
\label{rem:R-does-not-matter}
The proof of Proposition \ref{prop:R-does-not-matter} can be extended further.  Indeed let $R_1, \ldots, R_n, R'_1, \ldots, R'_m > 0$ and 
\[
\Gamma_{R_1, \ldots, R_n, R'_1,\ldots, R'_m}( \smbfe ; M, d, \epsilon)
\]
be defined like $\Gamma_{R}( \smbfe ; M, d, \epsilon)$ where instead of $\left\|A_i\right\|, \left\|B_j\right\| \leq R$ for all $i, j$, we only require $\left\|A_i\right\| \leq R_i$ and $\left\|B_j\right\| \leq R'_j$ for all $i,j$. If we extend the notion of $\chi_R( \smbfe )$ to $\chi_{R_1, \ldots, R_n, R'_1,\ldots, R'_m}( \smbfe ) $,  then the same proof as Proposition \ref{prop:R-does-not-matter} can be used to show that if $R_i > \left\|X_i\right\|$ and $R'_j > \left\|Y_j\right\|$ for all $i,j$, then 
\[
\chi_{R_1, \ldots, R_n, R'_1,\ldots, R'_m}( \smbfe ) = \chi( \smbfe ).
\]

In fact, we note that \cite{BB2003} refined the techniques of \cite{V1994}*{Proposition 2.4} to demonstrate that if one lets $R = \infty$ in the start of Definition \ref{defn:micro-free}, then the same value of the microstate free entropy is obtained.  By repeating their results verbatim with the obvious modifications in our context identical to those used above in Proposition \ref{prop:R-does-not-matter}, we note that setting $R = \infty$ from the start of Definition \ref{defn:micro-bi-free} yields the same quantity for the microstate bi-free entropy.
\end{rem}

On the other hand, insisting on using microstates of bounded norm allows us the following proposition.
\begin{prop}
\label{prop:limit-of-distributions}
Let $\left(\{X_i\}^n_{i=1}, \{Y_j\}^m_{j=1}\right)$ and $\left(\left\{X_i^{(k)}\right\}^n_{i=1}, \left\{Y_j^{(k)}\right\}^m_{j=1}\right)$ for $k \in \bN$ be tracially bi-partite tuples in a C$^*$-non-commutative probability space $(\A, \varphi)$.  Suppose that $\left(\left\{X_i^{(k)}\right\}^n_{i=1}, \left\{Y_j^{(k)}\right\}^m_{j=1}\right)$ converges in distribution to $\left(\{X_i\}^n_{i=1}, \{Y_j\}^m_{j=1}\right)$; that is
\[
\lim_{k \to \infty} \varphi\left(X^{(k)}_{i_1} \cdots X^{(k)}_{i_p} Y^{(k)}_{j_1} \cdots Y^{(k)}_{j_q}\right) = \varphi(X_{i_1} \cdots X_{i_p} Y_{j_1} \cdots Y_{j_q})
\]
for all $i_1, \ldots i_p \in \{1,\ldots, n\}$, $j_1, \ldots, j_q \in \{1,\ldots, m\}$, and $p,q \in \bN$.  Then
\[
\limsup_{k \to \infty} \chi_R\left(X^{(k)}_1, \ldots, X^{(k)}_n \sqcup Y^{(k)}_1, \ldots, Y^{(k)}_m\right) \leq \chi_R(\smbfe ).
\]
Moreover, if $\sup_{k \in \bN} \left\|X_i^{(k)}\right\| < \infty$ for all $1 \leq i \leq n$ and $\sup_{k \in \bN} \left\|Y_k^{(k)}\right\| < \infty$ for all $1 \leq j \leq m$, then
\[
\limsup_{k \to \infty} \chi(X^{(k)}_1, \ldots, X^{(k)}_n \sqcup Y^{(k)}_1, \ldots, Y^{(k)}_m) \leq \chi(\smbfe ).
\]
\end{prop}

\begin{proof}
	Our convergence assumption tells us that all moments are converging to the correct values, and so for any $M \in \bN$ and $\epsilon > 0$ we have for large enough $k$ that 
	\[
		\Gamma_R\left(X^{(k)}_1, \ldots, X^{(k)}_n \sqcup Y^{(k)}_1, \ldots, Y^{(k)}_m; M, d, \epsilon\right) \subseteq \Gamma_R(\smbfe; M, d, 2\epsilon),
	\]
	since the sets involved see only finitely many moments.
	Hence for all sufficiently large $k$, we have
	\[
		\chi_R\left(X^{(k)}_1, \ldots, X^{(k)}_n \sqcup Y^{(k)}_1, \ldots, Y^{(k)}_m; M, d, \epsilon\right) \leq \chi_R(\smbfe; M, d, 2\epsilon)
	\]
	and passing through the appropriate limits and rescaling in $d$, then $M$, and then $\epsilon$ yields
	\[
		\limsup_{k \to \infty} \chi_R\left(X^{(k)}_1, \ldots, X^{(k)}_n \sqcup Y^{(k)}_1, \ldots, Y^{(k)}_m\right) \leq \chi_R(\smbfe)
	\]
	which is the first claimed inequality.
	The second inequality follows by applying Proposition~\ref{prop:R-does-not-matter}.
\end{proof}

\section{Transformations}
\label{sec:Trans}

One important property of the microstate free entropy is the ability to apply a non-commutative functional calculus to the self-adjoint operators and control the value of the free entropy.  In this section, we will develop an analogue of this result for our microstate bi-free entropy.  However, due to the distinction between the left and right operators, we will need to focus on transformations that modify only left variables or modify only right variables (although compositions of such transforms is allowed).  

To understand the difficulty in mixing left and right variables, consider the $n = m = 1$ case.  If
\[
(A, B) \in \Gamma_R(X \sqcup Y; M, d, \epsilon)
\]
and we wanted to consider the new pair $(X, Y + cX)$ for $c$ sufficiently small, it is incredibly unclear whether 
\[
(A, B+cA) \in \Gamma_R(X \sqcup Y+cX; M', d, \epsilon')
\]
as the assumptions on $(A, B)$ yield only information about $\tau_d(A^pB^q)$ for $1 \leq p+q \leq M$ whereas we require knowledge about $\tau_d(A^p(B+cA)^q)$.  The latter involves terms of the form $\tau_d(A^{i_1}B^{i_2}A^{i_3}\cdots B^{i_j})$ and direct information about these moments appears difficult to extract from knowledge of only $\tau_d(A^pB^q)$.

In order to develop our results, we recall some information from \cite{V1994}.  However, as the proofs are near identical, we refer the reader to \cite{V1994} on most occasions.

Let $x_1,\ldots, x_n$ be non-commuting indeterminates and let
\[
F(x_1, \ldots, x_n) = \sum^\infty_{k=1} \sum_{1 \leq i_1, \ldots, i_k \leq n} c_{i_1, \ldots, i_k} x_{i_1} \cdots x_{i_k}
\]
be a non-commuting power series with complex coefficients.  If $R_i \geq 0$ for all $1 \leq i \leq n$, it is said that $(R_1, \ldots, R_n)$ is a \emph{multiradius of convergence of $F$} if
\[
M(F; R_1, \ldots, R_n) := \sum^\infty_{k=1} \sum_{1 \leq i_1, \ldots, i_k \leq n} |c_{i_1, \ldots, i_k}| R_{i_1} \cdots R_{i_k} < \infty.
\]
If $X_1, \ldots, X_n$ are elements in a finite factor $(\fM, \tau)$ and $(\left\|X_1\right\|, \ldots, \left\|X_n\right\|)$ is a multiradius of convergence of $F$, then $F(X_1, \ldots, X_n)$ is well-defined with
\[
\left\|F(X_1,\ldots, X_n)\right\| \leq M(F; \left\|X_1\right\|, \ldots, \left\|X_n\right\|).
\]

If $(R_1, \ldots, R_n)$ is a multiradius of convergence of $F$, then the map taking $(X_1, \ldots, X_n)$ to $F(X_1, \ldots, X_n)$ is an analytic function on 
\[
\prod_{1 \leq i \leq n} \{X_i \in \fM \, \mid \, \left\|X_i\right\| \leq R_i\}
\]
with values in $\fM$.  If this map is denoted $F$, then $F$ is differentiable with derivative denoted by $DF$, and the positive Jacobian of $F$ at $(X_1, \ldots, X_n)$ can be defined by
\[
|\mathcal{J}|(F)(X_1, \ldots, X_n) = |\det|(DF(X_1,\ldots, X_n)),
\]
where $|\det|$ denotes the Fuglede-Kadison determinant.
Note that $DF(X_1, \ldots, X_n)$ lies in the algebra denoted in \cite{V1994} by $LR(\fM)$, which is the image in $B(\fM)$ of the projective tensor product $\fM \otimes_\pi \fM^{\op}$ under the contraction $a\otimes b \mapsto L_aR_b$ (where $L_a$ denotes left multiplication on $\fM$ by $a$ and $R_b$ denotes right multiplication on $\fM$ by $b$).

Finally, as our focus is on self-adjoint operators, we will focus on $F$ where $F^* =F$; that is $\overline{c_{i_1, \ldots, i_k}} = c_{i_k, \ldots, i_1}$ for all $k$ and $1 \leq i_1,\ldots, i_k \leq n$.

\begin{prop}
\label{prop:transforms}
Let $(\A, \varphi)$ be a C$^*$-non-commutative probability space and let $$\left(\{X_i\}^n_{i=1}, \{Y_j\}^m_{j=1}\right)$$ be a tracially bi-partite collection of self-adjoint operators such that $(\alg(X_1,\ldots, X_n), \varphi)$ sits inside a finite factor.  Let $F_1, \ldots, F_n, G_1, \ldots, G_n$ be non-commutative power series with complex coefficients such that $F_i^* = F_i$, $G_i^* = G_i$, $(\left\|X_1\right\| + \epsilon, \ldots, \left\|X_n\right\| + \epsilon)$ is a multiradius of convergence for the $F_i$'s for some $\epsilon > 0$, and
\[
(M(F_1; \left\|X_1\right\| + \epsilon, \ldots, \left\|X_n\right\| + \epsilon), \ldots, M(F_n; \left\|X_1\right\| + \epsilon, \ldots, \left\|X_n\right\| + \epsilon))
\]
is a multiradius of convergence for the $G_j$'s.  Assume further that
\[
G_i(F_1(x_1, \ldots, x_n), \ldots, F_n(x_1,\ldots, x_n)) = x_i
\]
for all $1 \leq i \leq n$.  Then
\begin{align*}
\chi & (F_1(X_1, \ldots, X_n), \ldots, F_n(X_1, \ldots, X_n) \sqcup Y_1,\ldots, Y_n) \\
& \geq \log\left(|\mathcal{J}|((F_1, \ldots, F_n))(X_1, \ldots, X_n) \right) + \chi(X_1,\ldots, X_n \sqcup Y_1, \ldots, Y_m).
\end{align*}
Moreover, if  $N_k = \left\|F_k(X_1, \ldots, X_n)\right\|$, then
\[
(M(G_1; N_1 + \epsilon, \ldots, N_n + \epsilon), \ldots, M(G_n; N_1 + \epsilon, \ldots, N_n + \epsilon))
\]
is a multiradius of convergence for the $F_i$'s, then
\begin{align*}
\chi & (F_1(X_1, \ldots, X_n), \ldots, F_n(X_1, \ldots, X_n) \sqcup Y_1,\ldots, Y_n) \\
&= \log\left(|\mathcal{J}|((F_1, \ldots, F_n))(X_1, \ldots, X_n) \right) + \chi(X_1,\ldots, X_n \sqcup Y_1, \ldots, Y_m).
\end{align*}
An analogous result holds for such functions applied to the $Y$'s instead of the $X$'s.
\end{prop}

\begin{proof}
First we invoke Remark \ref{rem:R-does-not-matter}.  Let $\left\|X_i\right\| < R_i < \left\|X_i\right\| + \epsilon$, let $\left\|Y_j\right\| < R'_j$, and
\[
M(F_i; R_1,\ldots, R_n) < \rho_i \leq M(F_i; \left\|X_1\right\| + \epsilon, \ldots, \left\|X_n\right\| + \epsilon).
\]
Given $M \in \bN$, and $\epsilon > 0$, there exist an $M_1 \geq M$ and an $0 < \epsilon_1 < \epsilon$ such that the map
\[
(A_1, \ldots, A_n, B_1, \ldots, B_m) \mapsto (F_1(A_1,\ldots, A_n), \ldots, F_n(A_1,\ldots, A_n), B_1, \ldots, B_m)
\]
maps $\Gamma_{R_1,\ldots, R_n, R'_1, \ldots, R'_m}(\smbfe ; M_1, d, \epsilon_1)$ into 
\[
\Gamma_{\rho_1,\ldots, \rho_n, R'_1, \ldots, R'_m} (F_1(X_1, \ldots, X_n), \ldots, F_n(X_1, \ldots, X_n) \sqcup Y_1,\ldots Y_m; M, d, \epsilon).
\]
The remainder of the proof is identical to the proof of \cite{V1993}*{Proposition 3.5} as it simply computes how the transformation (ours being a direct sum of the one used in \cite{V1993}*{Proposition 3.5} and the identity) modifies the microstates and thus the entropy.
\end{proof}

\begin{cor}
\label{cor:micro-linear-transformations}
\begin{enumerate}
\item If $a_1, \ldots, a_n, b_1,\ldots, b_m \in \bR$, then
\[
\chi(X_1 +a_1 I, \ldots, X_n + a_n I \sqcup Y_1 + b_1I , \ldots, Y_m + b_m I) = \chi( \smbfe ).
\]
\item If $A = [a_{i,j}] \in \M_n$ and $B = [b_{i,j}] \in \M_m$ are invertible, then
\[
\chi\left(\sum^n_{k=1} a_{1,k} X_k, \ldots, \sum^n_{k=1} a_{n,k} X_k \sqcup \sum^m_{k=1} b_{1,k} Y_k, \ldots, \sum^m_{k=1} b_{m,k} Y_k    \right) = \chi (\smbfe ) + \log(|\det(A \oplus B)|).
\]
\item If $X_1,\ldots, X_n$ are linearly dependent or $Y_1, \ldots, Y_m$ are linearly dependent, then
\[
 \chi (\smbfe ) = -\infty.
\]
\end{enumerate}
\end{cor}
\begin{proof}
Parts (1) and (2) follow from Proposition \ref{prop:transforms}.
In the case of part (3), if $X_1,\ldots, X_n$ are linearly dependent then there is an $A = [a_{i,j}] \in \M_n$ such that $0 < |\det(A)| < 1$ and 
\[
	\left(\sum^n_{k=1} a_{1,k} X_k, \ldots, \sum^n_{k=1} a_{n,k} X_k\right) = (X_1, \ldots, X_n).
\]
Applying part (2) along with the fact that 
\[
\chi (\smbfe ) < \infty
\]
by Proposition \ref{prop:not-plus-infinity} yields the result.
\end{proof}

\section{Additivity of Microstate Bi-Free Entropy}
\label{sec:Addi}

One important result for free entropy is additivity; that is, if $\{X_1, \ldots, X_p\}$ and $\{X_{p+1},\ldots, X_n\}$ are free then
\[
\chi(X_1, \ldots, X_n) = \chi(X_1, \ldots, X_p) + \chi(X_{p+1},\ldots, X_n)
\]
under certain regularity assumptions.  We desire to prove a bi-free analogue of this result.  Before we move to those results, we desire to analyze some limits with regards to the following concept.

\begin{defn}
A tracially bi-partite system $(\{X_i\}^n_{i=1}, \{Y_j\}^m_{j=1})$ in a C$^*$-non-commutative probability space $(\A, \varphi)$ is said to have \emph{finite-dimensional approximants} if for every $M \in \bN$, $\epsilon > 0$, and 
\[
R > \max\left\{  \max_{1 \leq i \leq n} \left\|X_i\right\|, \max_{1 \leq j \leq m} \left\|Y_j\right\| \right\},
\]
there exists an $D \in \bN$ such that $\Gamma_R(\smbfe ; M, d, \epsilon) \neq \emptyset$ for all $d \geq D$.

A single family of such variables $\set{X_i}_{i=1}^n$ is said to have \emph{finite-dimensional approximants} if $(\set{X_i}_{i=1}^n, \emptyset)$ does or, equivalently by  Remark \ref{rem:monobifree-is-free}, if $(\emptyset, \set{X_i}_{i=1}^n)$ does.
\end{defn}

\begin{rem}
\label{rem:fda}
By repeating the same ideas as in Theorem \ref{thm:micro-converting-rights-to-lefts}, the existence of microstates for tracially bi-partite systems can be often deduced from knowledge of free entropy.  Indeed suppose $(\{X_i\}^n_{i=1}, \{Y_j\}^m_{j=1})$ is a tracially bi-partite system in a C$^*$-non-commutative probability space $(\A, \varphi)$ and that there exists another C$^*$-non-commutative probability space $(\A_0, \tau_0)$  and self-adjoint operators $X'_1, \ldots, X'_n, Y'_1, \ldots, Y'_m \in \A_0$ such that $\tau_0$ is tracial on $\A_0$ and
\[
\varphi(X_{i_1} \cdots X_{i_p} Y_{j_1} \cdots Y_{j_q}) = \tau_0(X'_{i_1} \cdots X'_{i_p} Y'_{j_q} \cdots Y'_{j_1})
\]
for all $p,q \in \bN \cup \{0\}$ and $i_1, \ldots, i_p \in \{1,\ldots, n\}$ and $j_1, \ldots, j_q \in \{1, \ldots, m\}$. 
As in Theorem \ref{thm:micro-converting-rights-to-lefts},
\[
\Gamma_R(X'_1, \ldots, X'_n, Y'_1, \ldots, Y'_m; M, d, \epsilon) \subseteq \Gamma_R(\smbfe; M, d, \epsilon).
\]
Therefore if $(\{X'_i\}^n_{i=1}, \{Y'_j\}^m_{j=1})$ have finite-dimensional approximants, then so do  $(\{X_i\}^n_{i=1}, \{Y_j\}^m_{j=1})$ by Proposition \ref{prop:R-does-not-matter}.  In particular, if $\chi(X'_1, \ldots, X'_n, Y'_1, \ldots, Y'_m) > -\infty$, then $X_1, \ldots, X_n, Y_1,\ldots, Y_m$ has finite-dimensional approximants by \cite{V1998-1}*{Remark 3.2}. 

Furthermore, if 
\[
\Gamma_R(\smbfe ; M, d_0, \epsilon)  \neq \emptyset
\]
for some $d_0$, then there exists a $D$ such that
\[
\Gamma_R(\smbfe ; M, d, 2\epsilon) \neq \emptyset
\]
for all $d \geq D$.  Indeed this follows by taking $D$ to be a sufficiently large multiple of $d_0$ so that $\frac{d_0}{D}$ is sufficiently small thereby adding at most $\epsilon$ to the state estimates.  Hence, as in \cite{V1998-1}*{Remark 3.2}, it can easily be seen that if $\chi(X_1, \ldots, X_n \sqcup Y_1, \ldots Y_m) > -\infty$, then $(\{X_i\}^n_{i=1}, \{Y_j\}^m_{j=1})$ has finite-dimensional approximants. 
\end{rem}

In order to develop an additive result for microstate bi-free entropy, we will use the following notion.

\begin{defn}
	Let $(\A, \varphi)$ be a C$^*$-non-commutative probability space, let $\{\C_k\}_{k \in K}$ be a collection of finite subsets of $\A$, let $\A_k = \alg(\C_k)$, and let $\psi = \ast_{k\in K}\varphi|_{\A_k}$ be the unique state on $\ast_{k \in K} \A_k$ extending each $\varphi|_{\A_k}$ such that the $\A_k$ are free.
Given $M \in \bN$ and $\epsilon > 0$, it is said that $\{\C_k\}_{k \in K}$ are \emph{$(M, \epsilon)$-free in $(\A, \varphi)$} provided
\[
\left|\psi(Z_1 \cdots Z_p) - \varphi(Z_1 \cdots Z_p)\right| < \epsilon
\]
for all $Z_1, \ldots, Z_p \in \bigcup_{k \in K} \C_k$ and $1 \leq p \leq M$.
\end{defn}

Given $d \in \bN$, let $\U(d)$ denote the unitary group of $\M_d$ and let $\gamma_d$ denote the normalized Haar measure on $\U(d)$.  We recall the following result.

\begin{lem}[\cite{V1998-1}*{Corollary 2.13}]
\label{lem:voi-lem-about-lots-of-matrices-making-asymptotic-freeness}
Fix $R, \epsilon, \theta > 0$ and $M \in \bN$.
Then there exists an $N \in \bN$ such that for all $d \geq N$, $1 \leq p \leq M$, and sets $\C_1, \ldots, \C_p \subseteq \M_d$ of matrices bounded in norm by $R$, each containing no more than $M$ elements, we have
\begin{align*}
\mu_{d}^{\otimes p}\left(\left\{ (U_1, \ldots, U_p) \in\U(d)^p \, \left| \, {\text{the sets } U_1^*\C_1 U_1, \ldots, U_p^*\C_p U_p \text{ are }(M, \epsilon)\text{-free}} \right.\right\}   \right) > 1 - \theta.
\end{align*}
\end{lem}

The following is based on \cite{HP2006}*{Lemma 6.4.3}.  Note the following also shows why the reverse order is desirable on the right matrices in Definition \ref{defn:micro-bi-free}.
In essence, if we have two pairs of faces that are bi-free, ``most'' ways of choosing microstates for each individually produce good microstates for the family.

\begin{lem}
\label{lem:large-portion-good}
Let $(\{X_i\}^n_{i=1}, \{Y_j\}^m_{j=1})$ be a tracially bi-partite system.  Suppose that for some $0 \leq p \leq n$ and $0 \leq q \leq m$ that
\[
(\alg(X_1, \ldots, X_p), \alg(Y_1, \ldots, Y_q)) \qqand (\alg(X_{p+1}, \ldots, X_{n}), \alg(Y_{q+1}, \ldots, Y_m))
\]
are bi-free and that 
\[
	\paren{\{X_1, \ldots, X_p\}, \{ Y_1, \ldots, Y_q \}} \qqand \paren{\{X_{p+1}, \ldots, X_{n}\}, \{Y_{q+1}, \ldots, Y_m\}}
\]
have finite-dimensional approximants.  Then for every $M \in \bN$, $\epsilon > 0$, and 
\[
R > \max\left\{\max_{1 \leq i \leq n} \left\|X_i\right\|, \max_{1 \leq j \leq m} \left\|Y_j\right\|\right\}
\]
there exists an $\epsilon_1 > 0$ such that
\[
\lim_{d \to \infty} \frac{\lambda_{d, n+m}\left(\Psi_d(M, \epsilon_1) \cap \Theta_d(M, \epsilon) \right)    }{\lambda_{d, n+m}\left(\Psi_d(M, \epsilon_1)\right) } = 1
\]
where $\frac{0}{0} = 1$,
\begin{align*}
\Psi_d(M, \epsilon_1)&= \Gamma_R( X_1, \ldots, X_p \sqcup Y_1, \ldots, Y_q; M, d, \epsilon_1) \times_{\ell r} \Gamma_R(X_{p+1}, \ldots, X_{n} \sqcup Y_{q+1}, \ldots, Y_m; M, d, \epsilon_1), \\
\Theta_d(M, \epsilon) &=\Gamma_R( \smbfe ; M, d, \epsilon),
\end{align*}
and $\times_{\ell r}$ is as defined in the proof of Proposition \ref{prop:micro-subadditive}.
\end{lem}

\begin{proof}
Fix $M \in \bN$, $\epsilon > 0$, and $R$ as described.
We claim that there exists an $\epsilon_1 > 0$ such that if
\[
(A_1, \ldots, A_n, B_1, \ldots, B_m) \in \Psi_d(M, \epsilon_1)
\]
and if
\[
\{A_1,\ldots, A_p, B_1, \ldots, B_q\} \qand \{A_{p+1}, \ldots, A_n, B_{q+1},\ldots, B_m\} \qquad \text{are }(M, \epsilon_1)\text{-free},
\]
then
\[
(A_1, \ldots, A_n, B_1, \ldots, B_m) \in \Theta_d(M, \epsilon).
\]

To see this, first take operators $X_1', \ldots, X_n', Y_1', \ldots, Y_m'$, of norm bounded by some $R_1 \geq R$, in another non-commutative probability space $(\A', \varphi')$ such that for all $0 \leq p, q$, all $i_1, \ldots, i_p \in \set{1, \ldots, n}$, and all $j_1, \ldots, j_q \in \set{1, \ldots, m}$, we have
\[
	\kappa_{k+l}(X'_{i_1}, \ldots, X'_{i_p}, Y'_{j_q}, \ldots, Y'_{j_1}) = \kappa_\chi(X_{i_1}, \ldots, X_{i_p}, Y_{j_1}, \ldots, Y_{j_q}).
\]
and such that\[
	\alg(X_1', \ldots, X_p', Y'_1, \ldots, Y'_q) \qqand \alg(X_{p+1}', \ldots, X_n', Y'_{q+1}, \ldots, Y'_m)
\]
are free.  Note that these two conditions will be consistent as the first condition will automatically imply some mixed free cumulants will vanish precisely because the mixed bi-free cumulants vanish.
Moreover, as a consequence of the cumulant construction, we have
\[
	\varphi'(X'_{i_1} \cdots X'_{i_p} Y'_{j_q} \cdots Y'_{j_1})  = \varphi(X_{i_1} \cdots X_{i_p} Y_{j_1} \cdots Y_{j_q}).
\]

Hence, for suitably small $\epsilon_1 > 0$, if 
\[
	(A_1, \ldots, A_n, B_1, \ldots, B_m) \in \Psi_d(M, \epsilon_1)
\]
and if
\[
	\set{A_1,\ldots, A_p, B_1, \ldots, B_q} \qand \set{A_{p+1}, \ldots, A_n, B_{q+1},\ldots, B_m} \qquad \text{are }(M, \epsilon_1)\text{-free},
\]
then for all $0 \leq p,q$ with $p + q \leq M$, $i_1,\ldots i_p \in \{1,\ldots, n\}$, and $j_1, \ldots, j_q \in \{1,\ldots, m\}$ we have that
\[
\tau_d(A_{i_1} \cdots A_{i_p} B_{j_q} \cdots B_{j_1})
\]
is within a multiple of $\epsilon_1$ (involving $M$ and $R_1$) of
\[
	\varphi'(X'_{i_1} \cdots X'_{i_p} Y'_{j_q} \cdots Y'_{j_1}) = \varphi(X_{i_1} \cdots X_{i_p} Y_{j_1} \cdots Y_{j_q}),
\]
thereby completing the claim.

Given $\theta > 0$, by Lemma \ref{lem:voi-lem-about-lots-of-matrices-making-asymptotic-freeness} there exists an $N \in \bN$ such that
\[
\gamma_d\left(\left\{U \in \U(d) \, \left| \, \substack{ \{A_1, \ldots, A_p, B_1,\ldots, B_q\} \text{ and }\{U^*A_{p+1}U, \ldots, U^*A_nU, U^*B_{q+1}U,\ldots, U^* B_mU\}\\ \text{ are }(M, \epsilon_1)\text{-free}} \right. \right\}\right) \geq 1 - \theta
\]
for all $d \geq N$ and all $A_i, B_j \in \M_d^{\sa}$ with $\left\|A_i\right\|, \left\|B_j\right\| \leq R$ for $1\leq i \leq n $ and $1 \leq j \leq m$.  

By the assumption of finite-dimensional approximants, $\Psi_d(M, \epsilon_1)$ is non-empty for sufficiently large $d$.  Let $\nu_d$ denote the normalized restriction of $\lambda_{d, n+m}$ to $\Psi_d(M, \epsilon_1)$.  Since both $\Psi_d(M, \epsilon_1)$ and $\nu_d$ are invariant under the action of $\U(d)$ given by
\[
(A_1,\ldots, A_n, B_1,\ldots, B_n) \mapsto (A_1, \ldots, A_p, U^*A_{p+1}U, \ldots, U^*A_nU, B_1, \ldots, B_{q}, U^*B_{q+1} U, \ldots, U^*B_mU),
\]
we obtain that
\begin{align*}
&\frac{\lambda_{d, n+m}\left(\Psi_d(M, \epsilon_1) \cap \Theta_d(M, \epsilon) \right)    }{\lambda_{d, n+m}\left( \Psi_d(M, \epsilon_1)\right) } \\ 
&= \int_{\Psi_d(M, \epsilon_1)} \left( \int_{\U(d)} 1_{\Theta_d(M, \epsilon)}(A_1, \ldots, A_p, U^*A_{p+1}U, \ldots, U^*A_nU, B_1, \ldots, B_{q}, U^*B_{q+1} U, \ldots, U^*B_mU) \, d\gamma(U) \right) d \nu_d.
\end{align*}
By the choice of $\epsilon_1$, we obtain for sufficiently large $d$ that
\[
\int_{\U(d)} 1_{\Theta_d(M, \epsilon)}(A_1, \ldots, A_p, U^*A_{p+1}U, \ldots, U^*A_nU, B_1, \ldots, B_{q}, U^*B_{q+1} U, \ldots, U^*B_mU) \, d\gamma(U) > 1-\theta.
\]
Hence 
\[
\frac{\lambda_{d, n+m}\left(\Psi_d(M, \epsilon_1) \cap \Theta_d(M, \epsilon) \right)    }{\lambda_{d, n+m}\left(\Psi_d(M, \epsilon_1)\right) } \geq 1-\theta
\]
which completes the proof as $\theta$ was arbitrary.
\end{proof}

Unfortunately, at this point in trying to prove additivity of microstate bi-free entropy for bi-free collections, we reach a bit of an impasse.  Either we need to know that the $\limsup_{d \to \infty}$ in Definition \ref{defn:micro-bi-free} is actually a limit, or we need to replaces the $\limsup_{d \to \infty}$ with a limit along an ultrafilter.  Thus, for the following result, we use $\chi^\omega(X_1, \ldots, X_n \sqcup Y_1, \ldots, Y_m)$ to denote the same quantity as in Definition \ref{defn:micro-bi-free} where $\limsup_{d \to \infty}$ is replaced with $\limsup_{\omega \to \infty}$ for some non-principle ultrafilter $\omega$.

\begin{thm}
\label{thm:micro-bi-free-additive}
Let $(\{X_i\}^n_{i=1}, \{Y_j\}^m_{j=1})$ be a tracially bi-partite system.  Suppose that for some $0 \leq p \leq n$ and $0 \leq q \leq m$ that
\[
(\alg(X_1, \ldots, X_p), \alg(Y_1, \ldots, Y_q)) \qqand (\alg(X_{p+1}, \ldots, X_{n}), \alg(Y_{q+1}, \ldots, Y_m))
\]
are bi-free. If the $\limsup_{d \to \infty}$ in Definition \ref{defn:micro-bi-free} is actually a limit for $(\{X_i\}^p_{i=1}, \{Y_j\}^q_{j=1})$ and for $(\{X_i\}^n_{i=p+1}, \{Y_j\}^m_{j=q+1})$, then
\[
\chi( \smbfe) = \chi(X_1, \ldots, X_p \sqcup Y_1, \ldots, Y_q ) + \chi(X_{p+1}, \ldots, X_{n}\sqcup Y_{q+1}, \ldots, Y_m).
\]
Alternatively
\[
\chi^\omega( \smbfe) = \chi^\omega(X_1, \ldots, X_p \sqcup Y_1, \ldots, Y_q ) + \chi^\omega(X_{p+1}, \ldots, X_{n}\sqcup Y_{q+1}, \ldots, Y_m).
\]
\end{thm}
\begin{proof}
By Proposition \ref{prop:micro-subadditive}
\[
\chi( \smbfe) \leq \chi(X_1, \ldots, X_p \sqcup Y_1, \ldots, Y_q ) + \chi(X_{p+1}, \ldots, X_{n}\sqcup Y_{q+1}, \ldots, Y_m)
\]
so the result is immediate if either quantity on the right hand side is $-\infty$.
Thus we may assume these microstate bi-free entropies are finite (and thus have finite-dimensional approximants) and proceed with demonstrating the other inequality.

For any $M \in \bN$, $\epsilon > 0$, and 
\[
R > \max\left\{\max_{1 \leq i \leq n} \left\|X_i\right\|, \max_{1 \leq j \leq m} \left\|Y_j\right\|\right\},
\]
Lemma \ref{lem:large-portion-good} implies there exists an $\epsilon_1 > 0$ such that 
\[
\lim_{d \to \infty} \frac{\lambda_{d, n+m}\left(\Psi_d(M, \epsilon_1) \cap \Theta_d(M, \epsilon) \right)    }{\lambda_{d, n+m}\left( \Psi_d(M, \epsilon_1)\right) } = 1,
\]
where
\begin{align*}
\Psi_d(M, \epsilon_1)&= \Gamma_R( X_1, \ldots, X_p \sqcup Y_1, \ldots, Y_q; M, d, \epsilon_1) \times_{\ell r} \Gamma_R(X_{p+1}, \ldots, X_{n} \sqcup Y_{q+1}, \ldots, Y_m; M, d, \epsilon_1)  \\
\Theta_d(M, \epsilon) &=\Gamma_R( \smbfe ; M, d, \epsilon).
\end{align*}
Hence, assuming $\limsup_{d \to \infty}$ in Definition \ref{defn:micro-bi-free} can be replaced with $\lim_{d \to \infty}$, we obtain that
\begin{align*}
&\chi_R( \smbfe ; M, \epsilon) \\
&= \limsup_{d \to \infty} \frac{1}{d^2}\log(\lambda_{d, n+m}(\Theta_d(M, \epsilon))) + \frac{1}{2}(n+m) \log(d)\\
& \geq \limsup_{d \to \infty} \frac{1}{d^2}\log(\lambda_{d, n+m}(\Theta_d(M, \epsilon)\cap \Psi_d(M, \epsilon_1))) + \frac{1}{2}(n+m) \log(d) \\
&= \limsup_{d \to \infty} \frac{1}{d^2}\log(\lambda_{d, n+m}(\Psi_d(M, \epsilon_1))) + \frac{1}{2}(n+m) \log(d) \\
&= \lim_{d \to \infty} \frac{1}{d^2}\log(\lambda_{d, n+m}(\Gamma_R( X_1, \ldots, X_p \sqcup Y_1, \ldots, Y_q; M, d, \epsilon_1))) + \frac{1}{2}(p+q) \log(d) \\
& \qquad + \frac{1}{d^2}\log(\lambda_{d, n+m}(\Gamma_R(X_{p+1}, \ldots, X_{n} \sqcup Y_{q+1}, \ldots, Y_m; M, d, \epsilon_1))) + \frac{1}{2}(n+m-p-q) \log(d)\\
&= \chi_R(X_1, \ldots, X_p \sqcup Y_1, \ldots, Y_q; M, \epsilon_1) + \chi_R(X_{p+1}, \ldots, X_{n} \sqcup Y_{q+1}, \ldots, Y_m; M, \epsilon_1) \\
&\geq \chi_R(X_1, \ldots, X_p \sqcup Y_1, \ldots, Y_q) + \chi_R(X_{p+1}, \ldots, X_{n} \sqcup Y_{q+1}, \ldots, Y_m),
\end{align*}
where in the last inequality we have used the fact that $\chi_R(\cdot\sqcup\cdot; M, \epsilon)$ decreases as $M$ increases and as $\epsilon$ decreases.

The result for $\chi^\omega$ easily follows by similar arguments.
\end{proof}

\begin{cor}
\label{cor:additive bi-free entropy}
Let $(\{X_i\}^n_{i=1}, \{Y_j\}^m_{j=1})$ be a tracially bi-partite system. If
\[
\alg(\{X_1, \ldots, X_n\}) \qqand \alg( \{Y_1, \ldots, Y_m\})
\]
are classically independent  and if the $\limsup_{d \to \infty}$ in Definition \ref{defn:micro-bi-free} is actually a limit for $\{X_i\}^n_{i=1}$ and for $ \{Y_j\}^m_{j=1})$, then
\[
\chi(\smbfe ) = \chi(X_1, \ldots, X_n) + \chi(Y_1, \ldots, Y_m).
\]
Alternatively
\[
\chi^\omega(\smbfe ) = \chi^\omega(X_1, \ldots, X_n) + \chi^\omega(Y_1, \ldots, Y_m).
\]
\end{cor}
\begin{proof}
	We recall from \cite{V2014} that classical independence implies the bi-freeness of $(\alg(\set{X_1, \ldots, X_n}), \bC)$ from $(\bC, \alg(\set{Y_1, \ldots, Y_m}))$.
	Theorem~\ref{thm:micro-bi-free-additive} then allows us to equate the bi-free entropy of the whole system with the sum of the bi-free entropies of the left variables and of the right variables, which by Remark~\ref{rem:monobifree-is-free} is just their free entropies.
\end{proof}

\section{Orbital Bi-Free Entropy}
\label{sec:Orbital}

In this section, we will develop a bi-free analogue of the notion of orbital free entropy, which was introduced in \cite{HMU2009} and is deeply connected to microstate free entropy.
Using the joint orbital bi-free entropy, a more descriptive bound can be given related to the subadditivity of the microstate bi-free entropy.
The approach used here is based on that of \cite{U2014} which is a close thematic fit to this paper; although it may be interesting to consider an approach similar to that of \cite{BD2013}, we do not do so here.
In fact, most proofs in this section are adaptations of those from \cite{U2014}.

Throughout this section, let $(\A, \varphi)$ be a C$^*$-non-commutative probability space and let $\ell \in \bN$.
For each $1 \leq k \leq \ell$, let ${\bf X}_k = (X_{k,1}, X_{k,2}, \ldots, X_{k,n_k})$ and ${\bf Y}_k = (Y_{k,1}, Y_{k,2}, \ldots, Y_{k,m_k})$ denote, respectively, $n_k$- and $m_k$-tuples of self-adjoint operators from $\A$, where $n_k, m_k \geq 0$ with $n_k + m_k \geq 1$.
Furthermore, we will use ${\bf Z}_k$ to denote the system of variables ${\bf X}_k \sqcup {\bf Y}_k$ where ${\bf X}_k$ are viewed as the left variables and ${\bf Y}_k$ are viewed as the right variables.

Let $U(d)$ denote the unitary matrices from $\M_d$ and let $\gamma_d$ denote the Haar measure on $U(d)$.  Furthermore define
\[
\Phi_d : U(d)^\ell \times \left( \prod^\ell_{k=1} (M_d^{\sa})^{n_k}\right) \times \left( \prod^\ell_{k=1} (M_d^{\sa})^{m_k}\right) \to \left( \prod^\ell_{k=1} (M_d^{\sa})^{n_k}\right) \times \left( \prod^\ell_{k=1} (M_d^{\sa})^{m_k}\right)
\]
by
\[
\Phi_d((U_k)^\ell_{k=1}, ({\bf A}_k)^\ell_{k=1}, ({\bf B}_k)^\ell_{k=1}) = ((U_k^*{\bf A}_k U_k)^\ell_{k=1}, (U_k^*{\bf B}_k U_k)^\ell_{k=1})
\]
where  for ${\bf A}_k = (A_{k,1}, A_{k,2}, \ldots, A_{k,n_k}) \in (M_d^{\sa})^{n_k}$ and ${\bf B}_k = (B_{k,1}, B_{k,2}, \ldots, B_{k,m_k}) \in (M_d^{\sa})^{m_k}$, 
\begin{align*}
U_k^*{\bf A}_k U_k &= (U^*_kA_{k,1}U_k, U^*_kA_{k,2}U_k, \ldots, U^*_kA_{k,n_k}U_k) \text{ and}\\
U_k^*{\bf B}_k U_k &= (U^*_kB_{k,1}U_k, U^*_kB_{k,2}U_k, \ldots, U^*_kB_{k,m_k}U_k).
\end{align*}
Moreover, let $\P\left(\left( \prod^\ell_{k=1} (M_d^{\sa})^{n_k}\right) \times \left( \prod^\ell_{k=1} (M_d^{\sa})^{m_k}\right)\right)$ denote the set of all regular Borel probability measures on $\left( \prod^\ell_{k=1} (M_d^{\sa})^{n_k}\right) \times \left( \prod^\ell_{k=1} (M_d^{\sa})^{m_k}\right)$.

Using
\[
\Gamma_R({\bf X}_1, \ldots, {\bf X}_\ell \sqcup {\bf Y}_1, \ldots, {\bf Y}_\ell; M, d, \epsilon)
\]
to denote the bi-free microstates based on the self-adjoint variables contained in the left variables ${\bf X}_1, \ldots, {\bf X}_\ell$ in the order listed and the right variables ${\bf Y}_1, \ldots, {\bf Y}_\ell$ in the order listed, we may now define the object of study in this section.

\begin{defn}
\label{defn:orbital-bi-free}
With the above notation, for each $\mu \in \P\left( \left( \prod^\ell_{k=1} (M_d^{\sa})^{n_k}\right) \times \left( \prod^\ell_{k=1} (M_d^{\sa})^{m_k}\right)\right)$, $M, d \in \bN$, and $R, \epsilon > 0$, let
\[
\chi_{\orb, R}({\bf Z}_1, \ldots, {\bf Z}_\ell; M, d, \epsilon; \mu) = \log\left( ( \gamma_d^{\otimes \ell} \otimes \mu)\left(\Phi_d^{-1}(\Gamma_R({\bf X}_1, \ldots, {\bf X}_\ell \sqcup {\bf Y}_1, \ldots, {\bf Y}_\ell; M, d, \epsilon)   )\right)   \right)
\]
(with $\log(0) = -\infty$).  With this we define
\begin{align*}
\chi_{\orb, R}({\bf Z}_1, \ldots, {\bf Z}_\ell; M, d, \epsilon)&= \sup_{\mu \in \P\left(\left( \prod^\ell_{k=1} (M_d^{\sa})^{n_k}\right) \times \left( \prod^\ell_{k=1} (M_d^{\sa})^{m_k}\right)\right)} \chi_{\orb, R}({\bf Z}_1, \ldots, {\bf Z}_\ell; M, d, \epsilon; \mu), \\
\chi_{\orb, R}({\bf Z}_1, \ldots, {\bf Z}_\ell; M, \epsilon)&= \limsup_{d \to \infty} \frac1{d^2} \chi_{\orb, R}({\bf Z}_1, \ldots, {\bf Z}_\ell; M, d, \epsilon),\\ 
\chi_{\orb, R}({\bf Z}_1, \ldots, {\bf Z}_\ell) &= \inf\{\chi_{\orb, R}({\bf Z}_1, \ldots, {\bf Z}_\ell; M, \epsilon) \, \mid \, M \in \bN, \epsilon > 0), \text{ and} \\
\chi_{\orb}({\bf Z}_1, \ldots, {\bf Z}_\ell) &= \sup_{0 < R < \infty} \chi_{\orb, R}({\bf Z}_1, \ldots, {\bf Z}_\ell).
\end{align*}
The quantity $\chi_{\orb}({\bf Z}_1, \ldots, {\bf Z}_\ell) \in [-\infty, 0]$ will be called the orbital bi-free entropy of the collections ${\bf Z}_1, \ldots, {\bf Z}_\ell$.  Note that the fact that $\chi_{\orb}({\bf Z}_1, \ldots, {\bf Z}_\ell)\leq 0$ is clear by definition.
\end{defn}

\begin{rem}
Based on the definition of $\chi_{\orb}({\bf Z}_1, \ldots, {\bf Z}_\ell)$, we can see that the orbital bi-free entropy is a measure of how well conjugation by unitaries preserves the bi-free microstates.
We can see that the infimum over $\epsilon$ and $M$ occurs as $\epsilon$ tends to $0$ and $M$ tends to infinity.
Furthermore, it is not difficult to see that if $m_k = 0$ for all $k$, then $\chi_{\orb}({\bf Z}_1, \ldots, {\bf Z}_\ell)$ agrees with $\chi_{\orb}({\bf X}_1, \ldots, {\bf X}_\ell)$ as in \cite{U2014}*{Definition 2.1}.

If the variables in question are not tracially bipartite, we have $\Gamma_R({\bf X}_1, \ldots, {\bf X}_\ell \sqcup {\bf Y}_1, \ldots, {\bf Y}_\ell; M, d, \epsilon) = \emptyset$ for appropriately antagonistic parameters, and so $\chi_{\orb}({\bf Z}_1, \ldots, {\bf Z}_\ell) = -\infty$.
Thus we continue to make this standing assumption throughout this section. 
\end{rem}

Of course the supremum over the probability measures portion of Definition \ref{defn:orbital-bi-free} may seem difficult to work with for computations.
From the theoretical standpoint it is quite natural, as we will see.
However, as with \cite{U2014}, there are other ways to describe $\chi_{\orb}({\bf Z}_1, \ldots, {\bf Z}_\ell)$ without the need to take a supremum over probability measures.
To provide one such description, first we need to develop some additional notation and demonstrate a lemma that will be useful throughout the section.

Given $(({\bf A}_k)^\ell_{k=1}, ({\bf B}_k)^\ell_{k=1}) \in \left( \prod^\ell_{k=1} (M_d^{\sa})^{n_k}\right) \times \left( \prod^\ell_{k=1} (M_d^{\sa})^{m_k}\right)$, let 
\[
\Gamma_{\orb}({\bf Z}_1, \ldots, {\bf Z}_\ell: ({\bf A}_k)^\ell_{k=1}, ({\bf B}_k)^\ell_{k=1}; M, d, \epsilon)
\]
denote the set of all $(U_k)^\ell_{k=1} \in (U(d))^\ell$ such that
\[
\Phi_d((U_k)^\ell_{k=1}, ({\bf A}_k)^\ell_{k=1}, ({\bf B}_k)^\ell_{k=1}) \in \Gamma_\infty({\bf X}_1, \ldots, {\bf X}_\ell \sqcup {\bf Y}_1, \ldots, {\bf Y}_\ell; M, d, \epsilon).
\]
Note that for $\Gamma_{\orb}({\bf Z}_1, \ldots, {\bf Z}_\ell: ({\bf A}_k)^\ell_{k=1}, ({\bf B}_k)^\ell_{k=1}; M, d, \epsilon)$ to be non-empty, each $({\bf A}_k, {\bf B}_k)$ must be good microstates for ${\bf X}_k \sqcup {\bf Y}_k$; more precisely, if $(\M_d^{\sa})_R$ denotes all elements of $\M_d^{\sa}$ of operator norm at most $R$ and if for some $k$ we have that $({\bf A}_k, {\bf B}_k) \in ((\M_d^{\sa})_R)^{n_k+m_k} \setminus \Gamma_R({\bf X}_k \sqcup {\bf Y}_k; M, d, \epsilon)$, then
 \[
\Gamma_{\orb}({\bf Z}_1, \ldots, {\bf Z}_\ell: ({\bf A}_k)^\ell_{k=1}, ({\bf B}_k)^\ell_{k=1}; M, d, \epsilon) = \emptyset. 
\]

We are now prepared to prove the following technical lemma.

\begin{lem}
\label{lem:orbital-integral}
For every $R> 0$, the map from $ \left( \prod^\ell_{k=1} (M_d^{\sa})^{n_k}_R\right) \times \left( \prod^\ell_{k=1} (M_d^{\sa})_R^{m_k}\right)$ to $\bR$ defined by
\[
(({\bf A}_k)^\ell_{k=1}, ({\bf B}_k)^\ell_{k=1}) \mapsto \gamma^{\otimes \ell}_d\left( \Gamma_{\orb}({\bf Z}_1, \ldots, {\bf Z}_\ell: ({\bf A}_k)^\ell_{k=1}, ({\bf B}_k)^\ell_{k=1}; M, d, \epsilon)  \right)
\]
is Borel.  Furthermore, for every $\mu \in \P\left(\left( \prod^\ell_{k=1} (M_d^{\sa})^{n_k}\right) \times \left( \prod^\ell_{k=1} (M_d^{\sa})^{m_k}\right)   \right)$, 
\begin{align*}
( \gamma_d^{\otimes \ell} \otimes \mu)&\left(\Phi_d^{-1}(\Gamma_R({\bf X}_1, \ldots, {\bf X}_\ell \sqcup {\bf Y}_1, \ldots, {\bf Y}_\ell; M, d, \epsilon)   )\right)  \\
&= \int_{\left( \prod^\ell_{k=1} (M_d^{\sa})_R^{n_k}\right) \times \left( \prod^\ell_{k=1} (M_d^{\sa})^{m_k}_R\right)  } \gamma^{\otimes \ell}_d\left( \Gamma_{\orb}({\bf Z}_1, \ldots, {\bf Z}_\ell: ({\bf A}_k)^\ell_{k=1}, ({\bf B}_k)^\ell_{k=1}; M, d, \epsilon)  \right) \, d\mu \\
&= \int_{ \prod^\ell_{k=1}  \Gamma_R({\bf X}_k \sqcup {\bf Y}_k; M, d, \epsilon)     } \gamma^{\otimes \ell}_d\left( \Gamma_{\orb}({\bf Z}_1, \ldots, {\bf Z}_\ell: ({\bf A}_k)^\ell_{k=1}, ({\bf B}_k)^\ell_{k=1}; M, d, \epsilon)  \right) \, d\mu,
\end{align*}
with an implicit reordering of coordinates in the second integral.
\end{lem}
\begin{proof}
The result clearly follows from the fact that the sets and functions involved are Borel, by the above constructions, and by Fubini's Theorem.  
\end{proof}

Now we are able to demonstrate an alternate definition of the orbital bi-free entropy without the need to take a supremum over probability measures.

\begin{prop}
\label{prop:other-def-of-orbital}
For each $M, d \in \bN$, $\epsilon > 0$, and $R \in (0, \infty]$, let
\begin{align*}
&\tilde{\chi}_{\orb, R}({\bf Z}_1, \ldots, {\bf Z}_\ell; M, d, \epsilon) \\
&= \sup_{(({\bf A}_k)^\ell_{k=1}, ({\bf B}_k)^\ell_{k=1}) \in  \left( \prod^\ell_{k=1} (M_d^{\sa})^{n_k}_R\right) \times \left( \prod^\ell_{k=1} (M_d^{\sa})_R^{m_k}\right)} \log\left( \gamma^{\otimes \ell}_d\left( \Gamma_{\orb}({\bf Z}_1, \ldots, {\bf Z}_\ell: ({\bf A}_k)^\ell_{k=1}, ({\bf B}_k)^\ell_{k=1}; M, d, \epsilon)    \right)   \right) \\
&= \sup_{(({\bf A}_k)^\ell_{k=1}, ({\bf B}_k)^\ell_{k=1}) \in \prod^\ell_{k=1}  \Gamma_R({\bf X}_k \sqcup {\bf Y}_k; M, d, \epsilon)   } \log\left( \gamma^{\otimes \ell}_d\left( \Gamma_{\orb}({\bf Z}_1, \ldots, {\bf Z}_\ell: ({\bf A}_k)^\ell_{k=1}, ({\bf B}_k)^\ell_{k=1}; M, d, \epsilon)    \right)   \right),
\end{align*}
Note $\tilde{\chi}_{\orb, R}({\bf Z}_1, \ldots, {\bf Z}_\ell; M, d, \epsilon) \in [-\infty, 0]$.  Then
\[
\chi_{\orb, R}({\bf Z}_1, \ldots, {\bf Z}_\ell) = \inf_{M \in \bN, \epsilon > 0} \limsup_{d \to \infty} \frac{1}{d^2} \tilde{\chi}_{\orb, R}({\bf Z}_1, \ldots, {\bf Z}_\ell; M, d, \epsilon).
\]
\end{prop}
\begin{proof}
First, it is clear that the two definitions of $\tilde{\chi}_{\orb, R}({\bf Z}_1, \ldots, {\bf Z}_\ell; M, d, \epsilon)$ are equivalent by the comments before Lemma \ref{lem:orbital-integral}.  Furthermore, Lemma \ref{lem:orbital-integral} implies that
\[
\chi_{\orb, R}({\bf Z}_1, \ldots, {\bf Z}_\ell; M, d, \epsilon; \mu) \leq \tilde{\chi}_{\orb, R}({\bf Z}_1, \ldots, {\bf Z}_\ell; M, d, \epsilon)
\]
for any $\mu \in \P\left(\left( \prod^\ell_{k=1} (M_d^{\sa})^{n_k}\right) \times \left( \prod^\ell_{k=1} (M_d^{\sa})^{m_k}\right)   \right)$.  Hence we clearly have
\[
\chi_{\orb, R}({\bf Z}_1, \ldots, {\bf Z}_\ell) \leq \inf_{M \in \bN, \epsilon > 0} \limsup_{d \to \infty} \frac{1}{d^2} \tilde{\chi}_{\orb, R}({\bf Z}_1, \ldots, {\bf Z}_\ell; M, d, \epsilon).
\]

To prove the reverse inequality, consider $M$ and $\epsilon$ fixed.  If $\tilde{\chi}_{\orb, R}({\bf Z}_1, \ldots, {\bf Z}_\ell; M, d, \epsilon) = -\infty$ for all sufficiently large $d$ then 
\[
\limsup_{d \to \infty} \frac{1}{d^2} \tilde{\chi}_{\orb, R}({\bf Z}_1, \ldots, {\bf Z}_\ell; M, d, \epsilon) \leq \chi_{\orb, R}({\bf Z}_1, \ldots, {\bf Z}_\ell; M, \epsilon)
\]
trivially follows.  Otherwise there exists an increasing sequence $(d_l)_{l\geq1}$ such that $\tilde{\chi}_{\orb, R}({\bf Z}_1, \ldots, {\bf Z}_\ell; M, d, \epsilon) > -\infty$ and
\[
\limsup_{d \to \infty} \frac{1}{d^2} \tilde{\chi}_{\orb, R}({\bf Z}_1, \ldots, {\bf Z}_\ell; M, d, \epsilon) = \lim_{l \to \infty} \frac{1}{d_l^2} \tilde{\chi}_{\orb, R}({\bf Z}_1, \ldots, {\bf Z}_\ell; M, d_l, \epsilon).
\]
For each $l \in \bN$, we can choose $(({\bf A}_{k,l})^\ell_{k=1}, ({\bf B}_{k,l})^\ell_{k=1}) \in \left( \prod^\ell_{k=1} (M_{d_l}^{\sa})^{n_k}_R\right) \times \left( \prod^\ell_{k=1} (M_{d_l}^{\sa})_R^{m_k}\right)$ such that
\[
\tilde{\chi}_{\orb, R}({\bf Z}_1, \ldots, {\bf Z}_\ell; M, d_l, \epsilon)-1 \leq \log\left( \gamma^{\otimes \ell}_{d_l}\left( \Gamma_{\orb}({\bf Z}_1, \ldots, {\bf Z}_\ell: ({\bf A}_{k,l})^\ell_{k=1}, ({\bf B}_{k,l})^\ell_{k=1}; M, {d_l}, \epsilon)    \right)   \right).
\]
Therefore, if $\delta_l \in  \P\left(\left( \prod^\ell_{k=1} (M_{d_l}^{\sa})^{n_k}\right) \times \left( \prod^\ell_{k=1} (M_{d_l}^{\sa})^{m_k}\right)   \right)$ is the point-mass measure at $(({\bf A}_{k,l})^\ell_{k=1}, ({\bf B}_{k,l})^\ell_{k=1})$, then Lemma \ref{lem:orbital-integral} implies that
\begin{align*}
&\tilde{\chi}_{\orb, R}({\bf Z}_1, \ldots, {\bf Z}_\ell; M, d_l, \epsilon)-1 \\
&\leq \log\left( \gamma^{\otimes \ell}_{d_l}\left( \Gamma_{\orb}({\bf Z}_1, \ldots, {\bf Z}_\ell: ({\bf A}_{k,l})^\ell_{k=1}, ({\bf B}_{k,l})^\ell_{k=1}; M, {d_l}, \epsilon)    \right)   \right) \\
&= \log\left(\int_{\left( \prod^\ell_{k=1} (M_{d_l}^{\sa})_R^{n_k}\right) \times \left( \prod^\ell_{k=1} (M_{d_l}^{\sa})^{m_k}_R\right)  } \gamma^{\otimes \ell}_{d_l}\left( \Gamma_{\orb}({\bf Z}_1, \ldots, {\bf Z}_\ell: ({\bf A}_{k,l})^\ell_{k=1}, ({\bf B}_{k,l})^\ell_{k=1}; M, {d_l}, \epsilon)  \right) \, d\delta_l \right) \\
&= \log\left(   ( \gamma_{d_l}^{\otimes \ell} \otimes \delta_l)\left(\Phi_{d_l}^{-1}(\Gamma_R({\bf X}_1, \ldots, {\bf X}_\ell \sqcup {\bf Y}_1, \ldots, {\bf Y}_\ell; M, d_l, \epsilon)   )\right)   \right)\\
&= \chi_{\orb, R}({\bf Z}_1, \ldots, {\bf Z}_\ell; M, d_l, \epsilon; \delta_l) \\
& \leq  \chi_{\orb, R}({\bf Z}_1, \ldots, {\bf Z}_\ell; M, d_l, \epsilon).
\end{align*}
Hence
\begin{align*}
\limsup_{d \to \infty} \frac{1}{d^2} \tilde{\chi}_{\orb, R}({\bf Z}_1, \ldots, {\bf Z}_\ell; M, d, \epsilon) &= \lim_{l \to \infty} \frac{1}{d_l^2} \tilde{\chi}_{\orb, R}({\bf Z}_1, \ldots, {\bf Z}_\ell; M, d_l, \epsilon) \\
& \leq \limsup_{l \to \infty} \frac{1}{d^2_l}(\chi_{\orb, R}({\bf Z}_1, \ldots, {\bf Z}_\ell; M, d_l, \epsilon) + 1) \\
& \leq \chi_{\orb, R}({\bf Z}_1, \ldots, {\bf Z}_\ell; M, \epsilon).
\end{align*}
Thus the result follows.
\end{proof}

\begin{rem}
We note that \cite{U2014} has an alternate characterization of the orbital free entropy in the case that the von Neumann algebras generated by each collection of left operators is hyperfinite.
Unfortunately, the arguments for such a characterization break down in the case of orbital bi-free entropy due to the fact that we are no longer dealing with a tracial state, so Jung's lemma \cite{J2007}*{Lemma 2.9} no longer applies.
To find an analogue of this result, one would need an obscure condition of changing all right operators to left ones while still maintaining hyperfiniteness.

In addition, it was shown in \cite{U2014} that the orbital free entropy depends only on the von Neumann algebras that the individual collections of operators generate.
However, this proof also does not extend to the bi-free situation.
The reason for this is the need to approximate other operators in the von Neumann algebra using Kaplansky's density theorem, which would cause the potential mixing of left and right operators in the polynomials.
This prevents us from obtaining microstate approximations of these polynomials in a similar fashion to the comments at the start of Section \ref{sec:Trans}.
\end{rem}

However, many basic properties of the orbital free entropy extend to the bi-free setting.  To begin, we have the following lemma showing the independence of $R$.

\begin{lem}
\label{lem:orbital-R-doesnt-matter}
Let $\rho = \max(\{1\} \cup \{\left\|X_{i,k}\right\| \, \mid\, 1 \leq i \leq n_k, 1 \leq k \leq \ell\} \cup \{\left\|Y_{j,k}\right\| \, \mid\, 1 \leq j \leq m_k, 1 \leq k \leq \ell\})$.  For any $R > \rho$, we have
\[
\chi_{\orb} ({\bf Z}_1, \ldots, {\bf Z}_\ell) = \chi_{\orb, R}({\bf Z}_1, \ldots, {\bf Z}_\ell)
\]
including the case $R = \infty$.
\end{lem}

\begin{proof}
Fix $R > \rho$.  Clearly
\[
\chi_{\orb, \infty}({\bf Z}_1, \ldots, {\bf Z}_\ell) \geq \chi_{\orb, R}({\bf Z}_1, \ldots, {\bf Z}_\ell)
\]
by definitions so it suffices to prove the other inequality. To begin, define $f : \bR \to [-1,1]$ by
\[
f(z) = \begin{cases}
z & \text{if } z \in [-1,1]\\
-1 & \text{if }z < -1 \\
1 & \text{if }z > 1
\end{cases},
\]
and let $f_R : \bR \to [-R,R]$ by $f_R(z) = R f\left(\frac{z}{R}\right)$.  

Fix $M \in \bN$ and $\epsilon > 0$.
Let $K = \max\{(\rho^{2M}+1)^{\frac{1}{2M}}, R\} > 1$, let $0 < \epsilon_0 < \min(1,\frac{\epsilon}{2}))$, and choose an even $M' \in \bN$ such that $M' \geq 2M$ and
\[
R\left( \left(\frac{\rho}{R}\right)^{M'} + \frac{\epsilon_0}{R^{M'}}\right)^\frac{1}{M} < \frac{\epsilon}{2M K^{M-1}}.
\]

Let $({\bf A}_k, {\bf B}_k) \in \Gamma_\infty({\bf X}_k \sqcup {\bf Y}_k; M', d, \epsilon_0)$ be arbitrary.
Then $\tau_d(A_{k,i}^{M''}) < \varphi(X_{k, i}^{M''}) + \epsilon_0 \leq \rho^{M''} + \epsilon_0$ for all $1 \leq i \leq n_k$ and $M''$ even and not greater than $M'$.
Thus, if for $p \in [1, \infty)$, $\left\|\cdot\right\|_p$ denotes the $p$-norm on $\M_d$ with respect to $\tau_d$, we have that $\left\|A_{k,i}\right\|_p \leq \norm{A_{k,i}}_{M'} \leq K$ for all $p \leq M'$, and in particular, for all $p \leq 2M$.
Thus, if $a_1, a_2, \ldots, a_d$ are the eigenvalues of $A_{k,i}$ counting multiplicities, we have for $p < M$ that
\begin{align*}
\left\|A_{k,i} - f_R(A_{k,i})\right\|_p \leq
\left\|A_{k,i} - f_R(A_{k,i})\right\|_M
& \leq R \left( \frac{1}{d} \sum_{|a_q| > R} \left|\frac{a_q}{R}\right|^M\right)^\frac{1}{M} \\
& \leq R \left( \frac{1}{d} \sum_{|a_q| > R} \left|\frac{a_q}{R}\right|^{M'}\right)^\frac{1}{M} \\
& \leq R \left( \frac{\tau_d(A_{k,i}^{M'})}{R^{M'}}\right)^\frac{1}{M} \\
& \leq R\left( \left(\frac{\rho}{R}\right)^{M'} + \frac{\epsilon_0}{R^{M'}}\right)^\frac{1}{M} < \frac{\epsilon}{2 M K^{M-1}}.
\end{align*}
Moreover, clearly $\left\|f_R(A_{k,i})\right\|_p \leq R \leq K$ for every $p$, $k$, and $i$, and identical inequalities holds for the $B_{k,j}$'s.  

The above implies if $(U_k)^\ell_{k=1} \in \Gamma_{\orb}({\bf Z}_1, \ldots, {\bf Z}_\ell; ({\bf A}_k)^\ell_{k=1}, ({\bf B}_k)^\ell_{k=1}; M', d, \epsilon_0)$, then for all $p,q$ with $p+q \leq M$, for all $k_1, \ldots, k_p, l_1,\ldots, l_q \in \{1, \ldots, \ell\}$, for all valid indices $i_1, \ldots, i_p, j_1, \ldots, j_q$, we have
\begin{align*}
& \left| \tau_d\left( U_{k_1}^* f_R(A_{k_1, i_1}) U_{k_1} \cdots U_{k_p}^* f_R(A_{k_p, i_p})U_{k_p} U^*_{l_q} f_R(B_{l_q, j_l}) U_{l_q} \cdots U^*_{l_1} f_R(B_{l_1, j_1}) U_{i_1}\right) \right. \\
& \qquad \left. - \varphi\left(X_{k_1, i_1} \cdots X_{k_p, i_p} Y_{l_1, j_1} \cdots Y_{l_q, j_q}\right)   \right| \\
& \leq \left|  \tau_d\left( U_{k_1}^* f_R(A_{k_1, i_1}) U_{k_1} \cdots U_{k_p}^* f_R(A_{k_p, i_p})U_{k_p} U^*_{l_q} f_R(B_{l_q, j_l}) U_{l_q} \cdots U^*_{l_1} f_R(B_{l_1, j_1}) U_{i_1}\right) \right. \\
& \qquad \left. - \tau_d\left( U_{k_1}^* A_{k_1, i_1} U_{k_1} \cdots U_{k_p}^* A_{k_p, i_p}U_{k_p} U^*_{l_q} B_{l_q, j_l} U_{l_q} \cdots U^*_{l_1} B_{l_1, j_1} U_{i_1}\right) \right| \\
& \quad + \left|\tau_d\left( U_{k_1}^* A_{k_1, i_1} U_{k_1} \cdots U_{k_p}^* A_{k_p, i_p}U_{k_p} U^*_{l_q} B_{l_q, j_l} U_{l_q} \cdots U^*_{l_1} B_{l_1, j_1} U_{i_1}\right) - \varphi\left(X_{k_1, i_1} \cdots X_{k_p, i_p} Y_{l_1, j_1} \cdots Y_{l_q, j_q}\right)   \right|\\
& \leq \sum^p_{x=1} K^{p+q} \left\|U^*_{k_x}( f_R(A_{k_x, i_x}) - A_{k_x, i_x}) U_{k_x}  \right\|_{p+q} + \sum^q_{y=1} K^{p+q} \left\| U^*_{k_y}( f_R(B_{k_y, j_y}) - B_{k_y, j_y}) U_{k_y} \right\|_{p+q} + \epsilon_0 \\
&\leq M K^{M} \frac{\epsilon}{2M K^{M-1}}+ \frac{\epsilon}{2} = \epsilon.
\end{align*}
where the second inequality if shown by the generalized H\"{o}lder's inequality for matrices.  Hence 
\[
\Gamma_{\orb}({\bf Z}_1, \ldots, {\bf Z}_\ell; ({\bf A}_k)^\ell_{k=1}, ({\bf B}_k)^\ell_{k=1}; M', d, \epsilon_0) \subseteq \Gamma_{\orb}({\bf Z}_1, \ldots, {\bf Z}_\ell; ({\bf f_R(A_k)})^\ell_{k=1}, ({\bf f_R(B_k)})^\ell_{k=1}; M, d, \epsilon)
\]
thereby implying
\[
\tilde{\chi}_{\orb, \infty}({\bf Z}_1, \ldots, {\bf Z}_\ell; M', d, \epsilon_0) \leq \tilde{\chi}_{\orb, R}({\bf Z}_1, \ldots, {\bf Z}_\ell; M, d, \epsilon)
\]
as $\left\|f_R(A_{k,i})\right\| \leq R$ and $\left\|f_R(B_{k,j})\right\| \leq R$ for all $i,j,k$.  Hence Proposition \ref{prop:other-def-of-orbital} implies that
\[
\chi_{\orb, \infty}({\bf Z}_1, \ldots, {\bf Z}_\ell) \leq \chi_{\orb, R}({\bf Z}_1, \ldots, {\bf Z}_\ell). \qedhere
\]
\end{proof}

Some basic properties of orbital bi-free entropy are readily established.

\begin{prop}
\label{prop:orbital-base-properties}
The following hold:
\begin{enumerate}
\item $\chi_{\orb}({\bf Z}_1, \ldots, {\bf Z}_\ell) = -\infty$ if ${\bf X}_1, \ldots, {\bf X}_\ell \sqcup {\bf Y}_1, \ldots, {\bf Y}_\ell$ do not have finite-dimensional approximants.
\item $\chi_{\orb}({\bf Z}) = 0$ if ${\bf Z}$ has finite-dimensional approximants.  Otherwise $\chi_{\orb}({\bf Z}) = -\infty$.
\item $\chi_{\orb}({\bf Z}_1, \ldots, {\bf Z}_\ell) \leq \chi_{\orb}({\bf Z}_1, \ldots, {\bf Z}_q) + \chi_{\orb}({\bf Z}_{q+1}, \ldots, {\bf Z}_\ell)$ for all $1 \leq q < \ell$.
\end{enumerate}
\end{prop}
\begin{proof}
Note (1) follows as if ${\bf X}_1, \ldots, {\bf X}_\ell \sqcup {\bf Y}_1, \ldots, {\bf Y}_\ell$ do not have finite-dimensional approximants, then 
\[
\Gamma_R({\bf X}_1, \ldots, {\bf X}_\ell \sqcup {\bf Y}_1, \ldots, {\bf Y}_\ell; M, d, \epsilon) = \emptyset
\]
for sufficiently large $M$, sufficiently small $\epsilon$, and for all $d$ (see Remark \ref{rem:fda}).  Note (2) follows as the $U(d)$ section of $\Phi_d^{-1}(\Gamma_R({\bZ}; M, d, \epsilon))$ is $U(d)$ if $\Gamma_R({\bZ}; M, d, \epsilon) \neq \emptyset$ and is $\emptyset$ if $\Gamma_R({\bZ}; M, d, \epsilon) = \emptyset$.  
Finally, (3) holds as if an $\ell$-tuple of unitaries works in the first step of the definition of $\chi_{\orb}({\bf Z}_1, \ldots, {\bf Z}_\ell)$, then the first $q$ work in the first step of the definition for $\chi_{\orb}({\bf Z}_1, \ldots, {\bf Z}_q) $ and the remainder work in the first step of the definition for $\chi_{\orb}({\bf Z}_{q+1}, \ldots, {\bf Z}_\ell)$.
\end{proof}

Like with the microstate bi-free entropy, the orbital bi-free entropy also plays reasonably well with respect to distributional limits.

\begin{prop}
\label{prop:limsup-orbital}
If ${\bf X}_1^{(l)}, \ldots, {\bf X}_\ell^{(l)} \sqcup {\bf Y}_1^{(l)}, \ldots, {\bf Y}_\ell^{(l)}$ converges in distribution to ${\bf X}_1, \ldots, {\bf X}_\ell \sqcup {\bf Y}_1, \ldots, {\bf Y}_\ell$ as in the sense of Proposition \ref{prop:limit-of-distributions}, then
\[
\chi_{\orb, R}({\bf Z}_1, \ldots, {\bf Z}_\ell) \geq \limsup_{l \to \infty} \chi_{\orb, R}({\bf Z}^{(l)}_1, \ldots, {\bf Z}^{(l)}_\ell)
\]
for every $R > 0$ including $R = \infty$.  Therefore, if there exists a uniform operator norm bound of these operators over $l$, we have
\[
\chi_{\orb}({\bf Z}_1, \ldots, {\bf Z}_\ell) \geq \limsup_{l \to \infty} \chi_{\orb}({\bf Z}^{(l)}_1, \ldots, {\bf Z}^{(l)}_\ell)
\]
\end{prop}
\begin{proof}
As in the proof of Proposition \ref{prop:limit-of-distributions}, our convergence assumption tells us that all moments are converging to the correct values, and so for any $M \in \bN$ and $\epsilon > 0$ we have for large enough $l$ that 
	\[
		\Gamma_R\left({\bf X}^{(l)}_1, \ldots, {\bf X}^{(l)}_\ell \sqcup {\bf Y}^{(l)}_1, \ldots, {\bf Y}^{(l)}_\ell; M, d, \epsilon\right) \subseteq \Gamma_R( {\bf X}_1, \ldots, {\bf X}_\ell \sqcup {\bf Y}_1, \ldots, {\bf Y}_\ell ; M, d, 2\epsilon),
	\]
since the sets involved see only finitely many moments.
Hence for any $\mu \in\P\left(\left( \prod^\ell_{k=1} (M_d^{\sa})^{n_k}\right) \times \left( \prod^\ell_{k=1} (M_d^{\sa})^{m_k}\right)\right)$, we have that
\begin{align*}
\chi_{\orb,R}\left({\bf Z}^{(l)}_1, \ldots, {\bf Z}^{(l)}_\ell; M, d, \epsilon; \mu \right)& \leq \chi_{\orb,R}\left({\bf Z}_1, \ldots, {\bf Z}_\ell; M, d, 2\epsilon; \mu \right) \\
& \leq \chi_{\orb,R}\left({\bf Z}_1, \ldots, {\bf Z}_\ell ; M, d, 2\epsilon\right).
\end{align*}
By taking the appropriate sups, limsup, and infimums, we obtain 
\[
\chi_{\orb, R}({\bf Z}_1, \ldots, {\bf Z}_\ell) \geq \limsup_{l \to \infty} \chi_{\orb, R}({\bf Z}^{(l)}_1, \ldots, {\bf Z}^{(l)}_\ell)
\]
for every $R > 0$ including $R = \infty$. The remaining equation then follows from Lemma \ref{lem:orbital-R-doesnt-matter}.
\end{proof}

Of greater interest is how the orbital bi-free entropy behaves with respect to bi-free collections.  In particular, the following proof uses similar ideas as those used in Lemma \ref{lem:large-portion-good} and Theorem \ref{thm:micro-bi-free-additive}.

\begin{thm}
\label{thm:orbital-partial-additive-bi-free}
If ${\bf Z}_1$ and ${\bf X}_2, \ldots, {\bf X}_\ell \sqcup {\bf Y}_2, \ldots, {\bf Y}_\ell$ are as described above and are bi-free with respect to $\varphi$, then
\[
\chi_{\orb}({\bf Z}_1, \ldots, {\bf Z}_\ell) = \chi_{\orb}({\bf Z}_1) + \chi_{\orb}({\bf Z}_2, \ldots, {\bf Z}_\ell).
\]
\end{thm}
\begin{proof}
First, suppose that ${\bf Z}_1$ does not have finite-dimensional approximants.   Then ${\bf X}_1, \ldots, {\bf X}_\ell \sqcup {\bf Y}_1, \ldots, {\bf Y}_\ell$ also does not have finite-dimensional approximants, so the definition of the orbital bi-free entropy implies that
\[
\chi_{\orb}({\bf Z}_1, \ldots, {\bf Z}_\ell) = -\infty = \chi_{\orb}({\bf Z}_1) = \chi_{\orb}({\bf Z}_1) + \chi_{\orb}({\bf Z}_2, \ldots, {\bf Z}_\ell).
\]
Hence we may assume that ${\bf Z}_1$ has finite-dimensional approximants so $ \chi_{\orb}({\bf Z}_1) = 0$ by Proposition \ref{prop:orbital-base-properties}.

Next, suppose that $\chi_{\orb}({\bf Z}_2, \ldots, {\bf Z}_\ell) = - \infty$.  Since Proposition \ref{prop:orbital-base-properties} implies then that $\chi_{\orb}({\bf Z}_1, \ldots, {\bf Z}_\ell) \leq -\infty$, the equation still holds.  Hence we may assume that $\chi_{\orb}({\bf Z}_1, \ldots, {\bf Z}_\ell) > -\infty$.  Furthermore, as $\chi_{\orb}({\bf Z}_1, \ldots, {\bf Z}_\ell) \leq \chi_{\orb}({\bf Z}_1) + \chi_{\orb}({\bf Z}_2, \ldots, {\bf Z}_\ell)$ by Proposition \ref{prop:orbital-base-properties}, it suffices to prove the other inequality with $ \chi_{\orb}({\bf Z}_1) = 0$.

Fix $R   > \max(\{1\} \cup \{\left\|X_{i,k}\right\| \, \mid\, 1 \leq i \leq n_k, 1 \leq k \leq \ell\} \cup \{\left\|Y_{j,k}\right\| \, \mid\, 1 \leq j \leq m_k, 1 \leq k \leq \ell\})$ and fix $M \in \bN$ and $\epsilon > 0$.  By the same argument as at the start of the proof of Lemma \ref{lem:large-portion-good}, there exists an $\epsilon_1 > 0$ such that if
\begin{itemize}
\item $({\bf A}_1, {\bf B}_1) \in \Gamma_R({\bf X}_1 \sqcup {\bf Y}_1; M, d, \epsilon_1)$, 
\item $(({\bf A}_k)^\ell_{k=2}, ({\bf B}_k)^\ell_{k=2}) \in \Gamma_R({\bf X}_2, \ldots, {\bf X}_\ell \sqcup {\bf Y}_2, \ldots, {\bf Y}_\ell; M, d, \epsilon_1)$, and
\item $({\bf A}_1, {\bf B}_1)$ and $(({\bf A}_k)^\ell_{k=2}, ({\bf B}_k)^\ell_{k=2})$ are $(M, \epsilon_1)$-free, 
\end{itemize}
then $(({\bf A}_k)^\ell_{k=1}, ({\bf B}_k)^\ell_{k=1}) \in \Gamma_R({\bf X}_1, \ldots, {\bf X}_\ell \sqcup {\bf Y}_1, \ldots, {\bf Y}_\ell; M, d, \epsilon)$.

Since by Lemma \ref{lem:orbital-R-doesnt-matter} we have that $\chi_{\orb, R}({\bf Z}_2, \ldots, {\bf Z}_\ell) = \chi_{\orb}({\bf Z}_2, \ldots, {\bf Z}_\ell) > - \infty$, Proposition \ref{prop:other-def-of-orbital} implies there exists an increasing sequence $(d_l)_{l\geq 1}$ such that $\tilde{\chi}_{\orb, R}({\bf Z}_2, \ldots, {\bf Z}_\ell; M, d_l, \epsilon_1) > -\infty$ and
\[
\limsup_{d \to \infty} \frac{1}{d^2} \tilde{\chi}_{\orb, R}({\bf Z}_2, \ldots, {\bf Z}_\ell; M, d, \epsilon_1) = \lim_{l \to \infty} \frac{1}{d_l^2} \tilde{\chi}_{\orb, R}({\bf Z}_2, \ldots, {\bf Z}_\ell; M, d_l, \epsilon_1).
\]
For each $l \in \bN$, choose $(({\bf A}_{k,l})^\ell_{k=2}, ({\bf B}_{k,l})^\ell_{k=2}) \in \left( \prod^\ell_{k=2} (M_{d_l}^{\sa})^{n_k}_R\right) \times \left( \prod^\ell_{k=2} (M_{d_l}^{\sa})_R^{m_k}\right)$ such that
\[
-\infty < \tilde{\chi}_{\orb, R}({\bf Z}_2, \ldots, {\bf Z}_\ell; M, d_l, \epsilon_1)-1 \leq \log\left( \gamma^{\otimes \ell-1}_{d_l}\left( \Gamma_{\orb}({\bf Z}_2, \ldots, {\bf Z}_\ell: ({\bf A}_{k,l})^\ell_{k=2}, ({\bf B}_{k,l})^\ell_{k=2}; M, {d_l}, \epsilon_1)    \right)   \right).
\]
Note this implies $\gamma^{\otimes \ell-1}_{d_l}\left( \Gamma_{\orb}({\bf Z}_2, \ldots, {\bf Z}_\ell: ({\bf A}_{k,l})^\ell_{k=2}, ({\bf B}_{k,l})^\ell_{k=2}; M, {d_l}, \epsilon_1) \right) > 0$.
Furthermore, as ${\bf Z}_1$ has finite-dimensional approximants, for $l$ sufficiently large we may choose a fixed $({\bf A}_{1,l}, {\bf B}_{1,l}) \in \Gamma_R({\bf X}_1 \sqcup {\bf Y}_1; M, d_l, \epsilon_1)$.

To simplify notation, let
\begin{align*}
\Psi(M, d_l, \epsilon) &= \Gamma_{\orb}({\bf Z}_1, \ldots, {\bf Z}_\ell; ({\bf A}_{k,l})^\ell_{k=1}, ({\bf B}_{k,l})^\ell_{k=1}; M, d_l, \epsilon)\\
\Theta(M, d_l, \epsilon_1)&= \Gamma_{\orb}({\bf Z}_2, \ldots, {\bf Z}_\ell; ({\bf A}_{k,l})^\ell_{k=2}, ({\bf B}_{k,l})^\ell_{k=2}; M, d_l, \epsilon_1) \\
\Omega(M, d_l, \epsilon_1) &= \{(U_k)^\ell_{k=1} \in U(d_l)^{\ell} \, \mid \, (U_1^*  {\bf A}_{1,l} U_1, U_1^*{\bf B}_{1,l} U_1), ((U_k^*{\bf A}_{k,l})^\ell_{k=2}U_k, U_k^*({\bf B}_{k,l})^\ell_{k=2}U_k)\text{ are } (M, \epsilon_1)\text{-free}\},
\end{align*}
and let $\mu_{d_l}$ be the probability measure obtained by restricting and renormalizing $\gamma^{\otimes \ell-1}$ to $\Theta(M, d_l, \epsilon_1)$.

%
Notice that by the choice of $\epsilon_1$ and the conditions defining the sets in question, we have
\[
\paren{U(d_l) \times \Theta(M, d_l, \epsilon_1)} \cap \Omega(M, d_l, \epsilon_1) \subseteq \Psi(M, d_l, \epsilon).
\]

By Lemma \ref{lem:voi-lem-about-lots-of-matrices-making-asymptotic-freeness} (with $p = 1$) there exists a $D_0 \in \bN$ such that
\[
\gamma_{d_l}(\{U_1 \in U(d_l) \, \mid \, (U_k)^\ell_{k=1} \in \Omega(M, d_l, \epsilon_1)\}) > \frac{1}{2}
\]
for every $d_l \geq D_0$ and every $(U_k)^\ell_{k=2} \in U(d_l)^{\ell-1}$.  Hence for all $d_l \geq D_0$ we have that
\begin{align*}
\frac{\gamma^{\otimes \ell}_{d_l}(\Psi(M, d_l, \epsilon))}{\gamma^{\otimes \ell-1}_{d_l}(\Theta(M, d_l, \epsilon_1))} &\geq (\gamma_{d_l} \otimes \mu_{d_l})(\Psi(M, d_l, \epsilon)) \\
&\geq (\gamma_{d_l} \otimes \mu_{d_l})((U(d_l) \times \Theta(M, d_l, \epsilon_1)) \cap \Omega(M, d_\ell, \epsilon_1)) \\
& = \int_{\Theta(M, d_l, \epsilon_1)} \gamma_{d_l}(\{U_1 \in U(d_l) \, \mid \, (U_k)^\ell_{k=1} \in \Omega(M, d_\ell, \epsilon_1)\}) \, d\mu_{d_l}((U_k)^\ell_{k=2}) > \frac{1}{2}
\end{align*}
by Fubini's Theorem.  Hence
\begin{align*}
\tilde{\chi}_{\orb, R}({\bf Z}_2, \ldots, {\bf Z}_\ell; M, d_l, \epsilon_1) & \leq 1+ \log\left( \gamma^{\otimes \ell-1}_{d_l}\left(\Theta(M, d_l, \epsilon_1) \right)   \right) \\
&<  1+\log(2)+ \log\left( \gamma^{\otimes \ell}_{d_l}\left( \Psi(M, d_l, \epsilon)  \right)   \right)\\
& \leq 1 + \log(2) + \tilde{\chi}_{\orb, R}({\bf Z}_1, \ldots, {\bf Z}_\ell ; M, d_l, \epsilon)
\end{align*}
whenever $d_l \geq D_0$.  Thus as $\tilde{\chi}_{\orb, R}({\bf Z}_2, \ldots, {\bf Z}_\ell; M, d, \epsilon_1)$ from Proposition \ref{prop:other-def-of-orbital} decreases as $M$ increases and as $\epsilon_1$ decreases, we have
\begin{align*}
\chi_{\orb, R}({\bf Z}_2, \ldots, {\bf Z}_\ell) & \leq \limsup_{d \to \infty} \frac{1}{d^2} \tilde{\chi}_{\orb, R}({\bf Z}_2, \ldots, {\bf Z}_\ell; M, d, \epsilon_1)\\
&= \lim_{l \to \infty}\frac{1}{d_l^2} \tilde{\chi}_{\orb, R}({\bf Z}_2, \ldots, {\bf Z}_\ell; M, d_l, \epsilon_1)\\
&\leq \limsup_{l \to \infty} \frac{1}{d_l^2} \left( 1 + \log(2) + \tilde{\chi}_{\orb, R}({\bf Z}_1, \ldots, {\bf Z}_\ell ; M, d_l, \epsilon)\right) \\
&\leq \limsup_{d \to \infty} \frac{1}{d^2} \tilde{\chi}_{\orb, R}({\bf Z}_1, \ldots, {\bf Z}_\ell ; M, d, \epsilon).
\end{align*}
Hence Proposition \ref{prop:other-def-of-orbital} implies that $\chi_{\orb, R}({\bf Z}_2, \ldots, {\bf Z}_\ell) \leq \chi_{\orb, R}({\bf Z}_1, \ldots, {\bf Z}_\ell)$.
\end{proof}

We are finally able to compute the orbital bi-free entropy of certain collections.  In particular, in the following case the orbital bi-free entropy is maximized.

\begin{cor}
\label{cor:orbital-bi-free-zero}
If ${\bf Z}_1$, ${\bf Z}_2$, $\ldots$, ${\bf Z}_\ell$ are bi-free with respect to $\varphi$ and individually have finite-dimensional approximants, then $\chi_{\orb}({\bf Z}_1, \ldots, {\bf Z}_\ell) = 0$.
\end{cor}
\begin{proof}
This immediately follows from Theorem \ref{thm:orbital-partial-additive-bi-free} and part (2) of Proposition \ref{prop:orbital-base-properties}.
\end{proof}

To finish off this section, we note an improvement to the subadditivity result for microstate bi-free entropy.  In particular, the following gives us a smaller upper bound for the joint microstate bi-free entropy in terms of the individual microstate bi-free entropies.

\begin{thm}
\label{thm:orbital-subadditive-with-micro}
With the notation used throughout this section, we have that
\[
\chi({\bf X}_1, \ldots, {\bf X}_\ell \sqcup {\bf Y}_1, \ldots, {\bf Y}_\ell   ) \leq \chi_{\orb}({\bf Z}_1, \ldots, {\bf Z}_\ell) + \sum^\ell_{k=1} \chi({\bf Z}_k)
\]
\end{thm}
\begin{proof}
First, if $\chi({\bf Z}_k) = -\infty$ for some $k$, then the result follows from Proposition \ref{prop:micro-subadditive}.  Hence we may assume that $\chi({\bf Z}_k) > -\infty$ for all $k$ and thus ${\bf Z}_k$ has finite-dimensional approximants for all $k$ by Remark \ref{rem:fda}.  

Fix $R >  \max(\{1\} \cup \{\left\|X_{i,k}\right\| \, \mid\, 1 \leq i \leq n_k, 1 \leq k \leq \ell\} \cup \{\left\|Y_{j,k}\right\| \, \mid\, 1 \leq j \leq m_k, 1 \leq k \leq \ell\})$, $M \in \bN$, and $\epsilon  > 0$.  As ${\bf Z}_k$ has finite-dimensional approximants for all $k$, there exists an $D_0 \in \bN$ such that $\Gamma_R({\bf Z}_k; M, d, \epsilon) \neq \emptyset$ for all $d \geq D_0$ and $1 \leq k \leq \ell$.  

Define $\sigma : \prod^\ell_{k=1}\left( (\M_d^{\sa})^{n_k} \times(\M_d^{\sa})^{m_k}   \right) \to \left( \prod^\ell_{k=1} (M_d^{\sa})^{n_k}\right) \times \left( \prod^\ell_{k=1} (M_d^{\sa})^{m_k}\right)$ by
\[
\sigma\left( (({\bf A}_k, {\bf B}_k))^\ell_{k=1}   \right) = (({\bf A}_k)^\ell_{k=1}, ({\bf B}_k)^\ell_{k=1}).
\]
Since each $\Gamma_R({\bf Z}_k; M, d, \epsilon)$ is non-empty and open, we know the Lebesgue measure of $\Gamma_R({\bf Z}_k; M, d, \epsilon)$ is non-zero.  Therefore, as $\sigma$ preserves the Lebesgue measure on $\lambda_d^{\otimes n_1+ \cdots + n_\ell + m_1 + \cdots + m_\ell}$ under the natural isomorphism with the domain and codomain, we have that $\sigma\left(\prod^\ell_{k=1} \Gamma_R({\bf Z}_k; M, d, \epsilon) \right)$ has positive Lebesgue measure.  Let $\nu_R(M, d, \epsilon)$ denote the probability measure obtained by renormalizing $\lambda_d^{\otimes n_1+ \cdots + n_\ell + m_1 + \cdots + m_\ell}$ after restricting to $\sigma\left(\prod^\ell_{k=1} \Gamma_R({\bf Z}_k; M, d, \epsilon)\right)$ when $d \geq D_0$; that is
\[
	\nu_R(M, d, \epsilon) = \frac{1}{\prod^\ell_{k=1} \lambda_d^{\otimes n_k + m_k}(\Gamma_R({\bf Z}_k; M, d, \epsilon))} \left.\lambda_d^{\otimes n_1+ \cdots + n_\ell + m_1 + \cdots + m_\ell}\right|_{\sigma\left(\prod^\ell_{k=1} \Gamma_R({\bf Z}_k; M, d, \epsilon) \right)}.
\]

%

Since
\[
\Gamma_R({\bf X}_1, \ldots, {\bf X}_\ell \sqcup {\bf Y}_1, \ldots, {\bf Y}_\ell; M, d, \epsilon) \subseteq \sigma\left(\prod^\ell_{k=1} \Gamma_R({\bf Z}_k; M, d, \epsilon)\right),
\]
we have by Definition \ref{defn:orbital-bi-free} that for all  $d \geq D_0$
\begin{align*}
\chi_{\orb, R}({\bf Z}_1, \ldots, {\bf Z}_\ell; M, d, \epsilon) & \geq \chi_{\orb, R}({\bf Z}_1, \ldots, {\bf Z}_\ell; M, d, \epsilon; \nu_R(M, d, \epsilon)) \\
&= \log\left( \left(\gamma_d^{\otimes \ell} \otimes \lambda_d^{ \otimes n_1+ \cdots + n_\ell + m_1 + \cdots + m_\ell}\right) \left( \Phi_d^{-1}(\Gamma_R({\bf X}_1, \ldots, {\bf X}_\ell \sqcup {\bf Y}_1, \ldots, {\bf Y}_\ell; M, d, \epsilon) \right) \right)\\
& \qquad - \sum^\ell_{k=1} \log\left( \lambda_d^{\otimes n_k + m_k}\left(\Gamma_R({\bf Z}_k; M, d, \epsilon)   \right) \right).
\end{align*}
Thus
\begin{align*}
\log&\left( \left(\gamma_d^{\otimes \ell} \otimes \lambda_d^{ \otimes n_1+ \cdots + n_\ell + m_1 + \cdots + m_\ell}\right) \left( \Phi_d^{-1}(\Gamma_R({\bf X}_1, \ldots, {\bf X}_\ell \sqcup {\bf Y}_1, \ldots, {\bf Y}_\ell; M, d, \epsilon) \right) \right) \\
&\leq \chi_{\orb, R}({\bf Z}_1, \ldots, {\bf Z}_\ell; M, d, \epsilon) + \sum^\ell_{k=1} \log\left( \lambda_d^{\otimes n_k + m_k}\left(\Gamma_R({\bf Z}_k; M, d, \epsilon)   \right) \right)
\end{align*}
for sufficiently large $d$ for every $M \in \bN$ and $\epsilon > 0$.

For a fixed $(U_k)^\ell_{k=1} \in U(d)^\ell$, notice that the corresponding section of $\Phi_d^{-1}(\Gamma_R({\bf X}_1, \ldots, {\bf X}_\ell \sqcup {\bf Y}_1, \ldots, {\bf Y}_\ell; M, d, \epsilon))$, namely
\[
	\{ (   ({\bf A}_k)^\ell_{k=1}, ({\bf B}_k)^\ell_{k=1}) \, \mid \, \Phi_d((U_k)^\ell_{k=1}, ({\bf A}_k)^\ell_{k=1}, ({\bf B}_k)^\ell_{k=1}) \in \Gamma_R({\bf X}_1, \ldots, {\bf X}_\ell \sqcup {\bf Y}_1, \ldots, {\bf Y}_\ell; M, d, \epsilon)\},
\] 
is exactly
\[
	\Phi_d\paren{(U_k^*)_{k=1}^\ell, \Gamma_R({\bf X}_1, \ldots, {\bf X}_\ell \sqcup {\bf Y}_1, \ldots, {\bf Y}_\ell; M, d, \epsilon)}.
\]
Hence Fubini's theorem and the fact that Lebesgue measure is unitarily-invariant together imply that
\begin{align*}
&\left(\gamma_d^{\otimes \ell} \otimes \lambda_d^{ \otimes n_1+ \cdots + n_\ell + m_1 + \cdots + m_\ell}\right) \left( \Phi_d^{-1}(\Gamma_R({\bf X}_1, \ldots, {\bf X}_\ell \sqcup {\bf Y}_1, \ldots, {\bf Y}_\ell; M, d, \epsilon) \right)  \\
&= \int_{U(d)^{\ell}}  \lambda_d^{ \otimes n_1+ \cdots + n_\ell + m_1 + \cdots + m_\ell}
\paren{
	\Phi_d\paren{(U_k^*)_{k=1}^\ell, \Gamma_R({\bf X}_1, \ldots, {\bf X}_\ell \sqcup {\bf Y}_1, \ldots, {\bf Y}_\ell; M, d, \epsilon)}
}  d \gamma_d^{\otimes \ell} \\
&=\int_{U(d)^{\ell}}  \lambda_d^{ \otimes n_1+ \cdots + n_\ell + m_1 + \cdots + m_\ell}\left(\Gamma_R({\bf X}_1, \ldots, {\bf X}_\ell \sqcup {\bf Y}_1, \ldots, {\bf Y}_\ell; M, d, \epsilon)  \right)  d \gamma_d^{\otimes \ell} \\
&= \lambda_d^{ \otimes n_1+ \cdots + n_\ell + m_1 + \cdots + m_\ell}\left(\Gamma_R({\bf X}_1, \ldots, {\bf X}_\ell \sqcup {\bf Y}_1, \ldots, {\bf Y}_\ell; M, d, \epsilon)  \right).
\end{align*}
Hence
\begin{align*}
\log&\left( \lambda_d^{ \otimes n_1+ \cdots + n_\ell + m_1 + \cdots + m_\ell}\left(\Gamma_R({\bf X}_1, \ldots, {\bf X}_\ell \sqcup {\bf Y}_1, \ldots, {\bf Y}_\ell; M, d, \epsilon)  \right)  \right) \\
&\leq \chi_{\orb, R}({\bf Z}_1, \ldots, {\bf Z}_\ell; M, d, \epsilon) + \sum^\ell_{k=1} \log\left( \lambda_d^{\otimes n_k + m_k}\left(\Gamma_R({\bf Z}_k; M, d, \epsilon)   \right) \right)
\end{align*}
so
\begin{align*}
\frac{1}{d^2}\log&\left( \lambda_d^{ \otimes n_1+ \cdots + n_\ell + m_1 + \cdots + m_\ell}\left(\Gamma_R({\bf X}_1, \ldots, {\bf X}_\ell \sqcup {\bf Y}_1, \ldots, {\bf Y}_\ell; M, d, \epsilon)  \right)  \right) + \frac{1}{2}\left( \sum^\ell_{k=1} n_k + m_k\right) \log(d) \\
&\leq \frac{1}{d^2}\chi_{\orb, R}({\bf Z}_1, \ldots, {\bf Z}_\ell; M, d, \epsilon) + \sum^\ell_{k=1} \frac{1}{d^2}\log\left( \lambda_d^{\otimes n_k + m_k}\left(\Gamma_R({\bf Z}_k; M, d, \epsilon)   \right) \right) + \frac{1}{2} (n_k + m_k) \log(d).
\end{align*}
Now, taking the appropriate limits, the result follows.
\end{proof}

We note that inequality in Theorem \ref{thm:orbital-subadditive-with-micro} need not be an equality.  Indeed \cite{U2017} shows that the inequality can be strict in the free setting.

\section{A Characterization of Bi-Freeness}
\label{sec:Characterization}

The goal of this section is to develop another characterization of bi-freeness for specific tracially bi-partite systems.  To be specific, using the same notation as Section \ref{sec:Orbital}, the main goal of this section is to prove the following.

\begin{thm}
\label{thm:orbital-bi-free-characterization}
Let ${\bf Z}_1$, ${\bf Z}_2$, $\ldots$, ${\bf Z}_\ell$ be such that
\[
\left(  \bigcup^\ell_{k=1} \{X_{k,i}\}^{n_k}_{i=1},   \bigcup^\ell_{k=1} \{Y_{k,j}\}^{m_k}_{j=1}  \right)
\]
is a tracially bi-partite system.  Then ${\bf Z}_1$, ${\bf Z}_2$, $\ldots$, ${\bf Z}_\ell$ are bi-free and individually have finite-dimensional approximants if and only if $\chi_{\orb}({\bf Z}_1, \ldots, {\bf Z}_\ell) = 0$.
\end{thm}

Of course, the only if direction immediately follows from Corollary \ref{cor:orbital-bi-free-zero}.  Thus it remains to prove under these assumptions that $\chi_{\orb}({\bf Z}_1, \ldots, {\bf Z}_\ell) = 0$ implies bi-freeness (as it clearly implies finite-dimensional approximants).  Before we get to that, we immediately have the following by combining Theorem \ref{thm:orbital-bi-free-characterization} with the results of Section \ref{sec:Orbital}.

\begin{cor}
\label{cor:additive-implies-bi-free}
Let ${\bf Z}_1$, ${\bf Z}_2$, $\ldots$, ${\bf Z}_\ell$ be such that $\chi({\bf Z}_k) > -\infty$ for all $1 \leq k \leq \ell$.  Suppose further that
\[
\left(  \bigcup^\ell_{k=1} \{X_{k,i}\}^{n_k}_{i=1},   \bigcup^\ell_{k=1} \{Y_{k,j}\}^{m_k}_{j=1}  \right)
\]
is a tracially bi-partite system.   If
\[
\chi({\bf X}_1, \ldots, {\bf X}_\ell \sqcup {\bf Y}_1, \ldots, {\bf Y}_\ell) = \sum^\ell_{k=1} \chi({\bf Z}_k),
\]
then ${\bf Z}_1$, ${\bf Z}_2$, $\ldots$, ${\bf Z}_\ell$ are bi-free.
\end{cor}
\begin{proof}
As $\chi({\bf Z}_k) > -\infty$ for all $1 \leq k \leq \ell$, we know from Remark \ref{rem:fda} that ${\bf Z}_1$, ${\bf Z}_2$, $\ldots$, ${\bf Z}_\ell$ individually have finite-dimensional approximants.  Furthermore, the assumption of additivity of the microstate bi-free entropy implies that 
\[
\chi({\bf X}_1, \ldots, {\bf X}_\ell \sqcup {\bf Y}_1, \ldots, {\bf Y}_\ell) > -\infty.
\]

By Theorem \ref{thm:orbital-subadditive-with-micro} along with the assumption, we know that
\begin{align*}
\chi({\bf X}_1, \ldots, {\bf X}_\ell \sqcup {\bf Y}_1, \ldots, {\bf Y}_\ell) & \leq \chi_{\orb}({\bf Z}_1, \ldots, {\bf Z}_\ell) + \sum^\ell_{k=1} \chi({\bf Z}_k)\\
&= \chi_{\orb}({\bf Z}_1, \ldots, {\bf Z}_\ell) +\chi({\bf X}_1, \ldots, {\bf X}_\ell \sqcup {\bf Y}_1, \ldots, {\bf Y}_\ell).
\end{align*}
Thus $\chi_{\orb}({\bf Z}_1, \ldots, {\bf Z}_\ell) \geq 0$.  However, as $\chi_{\orb}({\bf Z}_1, \ldots, {\bf Z}_\ell) \leq 0$ by definition, we obtain that $\chi_{\orb}({\bf Z}_1, \ldots, {\bf Z}_\ell) = 0$.  Hence Theorem \ref{thm:orbital-bi-free-characterization} implies that ${\bf Z}_1$, ${\bf Z}_2$, $\ldots$, ${\bf Z}_\ell$ are bi-free.
\end{proof}

To begin the proof of Theorem \ref{thm:orbital-bi-free-characterization}, we first need an analogue of the free Wasserstein metric from \cite{BV2001} for the following objects.

\begin{defn}
A quadruple $(\A, \mathcal{L}, \mathcal{R}, \varphi)$ is said to be a \emph{left-right, tracially bi-partite, C$^*$-non-commutative probability space} if $(\A, \varphi)$ is a C$^*$-non-commutative probability space, $\mathcal{L}$ and $\mathcal{R}$ are unital C$^*$-subalgebras of $\A$ that commute with one another, and $\varphi$ is tracial when restricted to $\mathcal{L}$ and when restricted to $\mathcal{R}$.

By saying a tracially bi-partite system $\left(\{X_i\}^n_{i=1}, \{Y_j\}^m_{j=1}\right)$
is in a left-right, tracially bi-partite, C$^*$-non-commutative probability space $(\A, \mathcal{L}, \mathcal{R}, \varphi)$, we mean $\{X_i\}^n_{i=1} \subseteq \mathcal{L}$ and $\{Y_j\}^m_{j=1} \subseteq \mathcal{R}$.
Note any tracially bi-partite system can be realized in a left-right, tracially bi-partite, C$^*$-non-commutative probability space.
\end{defn}

\begin{defn}
\label{defn:Wasserstein}
Let $\left(\{X_{i,1}\}^n_{i=1}, \{Y_{j,1}\}^m_{j=1}\right)$ and $\left(\{X_{i,2}\}^n_{i=1}, \{Y_{j,2}\}^m_{j=1}\right)$ be tracially bi-partite systems in left-right, tracially bi-partite, C$^*$-non-commutative probability spaces $(\A_1, \mathcal{L}_1, \mathcal{R}_1, \varphi_1)$ and $(\A_2, \mathcal{L}_2, \mathcal{R}_2, \varphi_2)$ respectively.  We define
\[
W_2\left(\left(\{X_{i,1}\}^n_{i=1}, \{Y_{j,1}\}^m_{j=1}\right), \left(\{X_{i,2}\}^n_{i=1}, \{Y_{j,2}\}^m_{j=1}\right)    \right)
\]
to be infimum of
\[
\left( \sum^n_{i=1} \left\| X'_{i,1} - X'_{i,2}\right\|_2^2 + \sum^m_{j=1} \left\|Y'_{j,1} - Y'_{j,2}\right\|_2^2\right)^\frac{1}{2}
\]
over all tracially bi-partite systems $\left(\{X'_{i,1}\}^n_{i=1}, \{Y'_{j,1}\}^m_{j=1}\right)$ and $\left(\{X'_{i,2}\}^n_{i=1}, \{Y'_{j,2}\}^m_{j=1}\right)$ in a left-right, tracially bi-partite, C$^*$-non-commutative probability spaces $(\A, \mathcal{L}, \mathcal{R}, \varphi)$ such that $\left(\{X_{i,k}\}^n_{i=1}, \{Y_{j,k}\}^m_{j=1}\right)$ and $\left(\{X'_{i,k}\}^n_{i=1}, \{Y'_{j,k}\}^m_{j=1}\right)$ have the same $*$-distributions and individual operator norms for $k=1,2$, where $\left\| \, \cdot \, \right\|_2$ denotes the $2$-seminorm with respect to $\varphi$ (note we may only have a seminorm as we are not restricting ourselves to faithful states).
\end{defn}

\begin{rem}
\label{rem:Wasserstein-bi-free-product}
It is natural and necessary to ask whether one can find a left-right, tracially bi-partite, C$^*$-non-commutative probability spaces $(\A, \mathcal{L}, \mathcal{R}, \varphi)$ as described in Definition \ref{defn:Wasserstein} so that the infimum is over a non-empty set.  This is indeed the case by considering reduced free products.  If one takes the reduced free product Hilbert space $(\A_1, \varphi_1) \ast (\A_2, \varphi_2)$, we can let $\mathcal{L}_1$ and $\mathcal{L}_2$ act via the left regular representation on $\A_1$ and $\A_2$ respectively, and let $\mathcal{R}_1$ and $\mathcal{R}_2$ act via the right regular representation on $\A_1$ and $\A_2$ respectively.  These representations are $\varphi$-preserving $*$-homomorphism and thus preserve distributions and the operator norms.  Furthermore, the C$^*$-algebra $\mathcal{L}$ generated by the images of $\mathcal{L}_1$ and $\mathcal{L}_2$ commutes with the  C$^*$-algebra $\mathcal{R}$ generated by the images of $\mathcal{R}_1$ and $\mathcal{R}_2$.  Finally, the reduced free product state is tracial on $\mathcal{L}$ and is tracial on $\mathcal{R}$ by properties of the reduced free product (i.e. the free case).  Of course, this is one reason why the states in a left-right, tracially bi-partite, C$^*$-non-commutative probability space need not be faithful as the work of \cite{R2017} shows we would be greatly restricting the systems we can study in that the bi-free product of faithful states need not be faithful.
\end{rem}

Using Definition \ref{defn:Wasserstein}, we can consider a similar definition for `nice' states.

\begin{defn}
\label{defn:Wasserstein-states}
Let $\A$ be a C$^*$-algebra and let $\mathcal{L}$ and $\mathcal{R}$ be unital subalgebras of $\A$ that commute with one another.
Suppose that $\A$ is generated by $\mathcal{L}$ and $\mathcal{R}$, which in turn are generated by prescribed sets $\set{X_i}_{i=1}^n$ and $\set{Y_j}_{j=1}^m$ respectively.

Let $\mathcal{CS}(\A, \mathcal{L}, \mathcal{R})$ denote the set of all states (positive unital linear functionals of norm one) that are tracial when restricted to $\mathcal{L}$ and are tracial when restricted to $\mathcal{R}$.
We define
\[
W_2(\varphi_1, \varphi_2) = W_2\left(\left(\{X_{i,1}\}^n_{i=1}, \{Y_{j,1}\}^m_{j=1}\right), \left(\{X_{i,2}\}^n_{i=1}, \{Y_{j,2}\}^m_{j=1}\right)    \right)
\]
where $\left(\{X_{i,k}\}^n_{i=1}, \{Y_{j,k}\}^m_{j=1}\right)$ denote $\left(\{X_{i}\}^n_{i=1}, \{Y_{j}\}^m_{j=1}\right)$ in $(\A, \mathcal{L}, \mathcal{R}, \varphi_k)$ for $k=1,2$.
Note that $W_2$ depends on the choice of generating set, but we leave this implicit.
\end{defn}

Like the free Wasserstein metric from \cite{BV2001}, the function $W_2$ has some nice properties.

\begin{prop}
\label{prop:Wasserstein-0-equal}
The bi-free analogue of the Wasserstein metric is a semimetric on the collection of tracially bi-partite systems with equal numbers of left variables and equal numbers of right variables, and is a semimetric $\mathcal{CS}(\A, \mathcal{L}, \mathcal{R})$.  
\end{prop}
\begin{proof}
The reasons that
\[
W_2\left(\left(\{X_{i,1}\}^n_{i=1}, \{Y_{j,1}\}^m_{j=1}\right), \left(\{X_{i,2}\}^n_{i=1}, \{Y_{j,2}\}^m_{j=1}\right)    \right) = 0
\]
implies $\left(\{X_{i,1}\}^n_{i=1}, \{Y_{j,1}\}^m_{j=1}\right)$ and $\left(\{X_{i,2}\}^n_{i=1}, \{Y_{j,2}\}^m_{j=1}\right)$ have the same distribution and the reasons that $W_2(\varphi_1, \varphi_2) = 0$ implies $\varphi_1 = \varphi_2$ both follow from the facts that the operator norms of the representations in Definition \ref{defn:Wasserstein} are bounded, the left and right algebras commute with each other, all linear functionals considered are states, the traciality of the states on the individual left and right algebras, and the definition of $W_2$.
Indeed, for an example computation, with terms as in Definition \ref{defn:Wasserstein} (where all operator are self-adjoint),
notice that
\begin{align*}
\left|\varphi(X'_{1,1} X'_{2,1} Y'_{1,1}) - \varphi(X'_{1,2} X'_{2,1} Y'_{1,1})\right| &\leq \varphi(1) \varphi\left(Y'_{1,1}X'_{2,1} (X'_{1,1} - X'_{1,2})(X'_{1,1} - X'_{1,2})X'_{2,1} Y'_{1,1}\right)^\frac{1}{2}\\
&= \varphi\left(X'_{2,1} (X'_{1,1} - X'_{1,2})Y'_{1,1}Y'_{1,1}(X'_{1,1} - X'_{1,2})X'_{2,1}     \right)^\frac{1}{2}\\
& \leq \left\|Y'_{1,1}\right\| \varphi\left(X'_{2,1} (X'_{1,1} - X'_{1,2}) (X'_{1,1} - X'_{1,2})X'_{2,1}     \right)^\frac{1}{2}\\
& = \left\|Y'_{1,1}\right\| \varphi\left( (X'_{1,1} - X'_{1,2})X'_{2,1}X'_{2,1} (X'_{1,1} - X'_{1,2})     \right)^\frac{1}{2}\\
&\leq  \left\|Y'_{1,1}\right\|\left\|X'_{2,1}\right\| \left\|X'_{1,1} - X'_{1,2}\right\|_2
\end{align*}
where the first equality is left-right commutation, and the second equalityis traciality on the left.  Using telescoping sums along with the bounds on the operator norms, the fact that $W_2$ is 0 and thus we can find $\left(\{X'_{i,1}\}^n_{i=1}, \{Y'_{j,1}\}^m_{j=1}\right)$ and $\left(\{X'_{i,2}\}^n_{i=1}, \{Y'_{j,2}\}^m_{j=1}\right)$ as in Definition \ref{defn:Wasserstein} with arbitrarily small 2-seminorms, we can show the difference in the distribution of any monomial is as small as we desire and thus equal.  The remaining properties of a semimetric are trivial to verify.
\end{proof}

\begin{rem}
Unfortunately we do not know whether or not $W_2$ is a metric.  The problem with trying to repeat the proof of \cite{BV2001} is that there is no current bi-free product that enables one to amalgamate the left operators over one subalgebra and the right operators over another non-isomorphic subalgebra; that is, \cite{CNS2015-1} amalgamates over a copy of an algebra contained in both the left and right operators.  This creates a problem with trying to use the bi-free product construction from Remark \ref{rem:Wasserstein-bi-free-product} to take two different pairs and construct a left-right, tracially bi-partite, C$^*$-non-commutative probability space containing all three in a way that the both pairs are identified in the appropriate way.  In particular, positivity and lack of traciality become issues.

It would also be nice to generalize the above to non-bi-partite systems.  However, as we are dealing with seminorms, it does appear difficult to even get a semimetric considering the current proof of Proposition \ref{prop:Wasserstein-0-equal}.
\end{rem}

Fortunate for the discussions in this paper, Proposition \ref{prop:Wasserstein-0-equal} along with the following result are enough.

\begin{prop}
\label{prop:Wass-lower-semicontinuous}
Given sequences $(\varphi_{1,k})_{k\geq 1}$ and $(\varphi_{2,k})_{k\geq 1}$ in $\mathcal{CS}(\A, \mathcal{L}, \mathcal{R})$ that converge weak$^*$ to $\varphi_1$ and $\varphi_2$ in $\mathcal{CS}(\A, \mathcal{L}, \mathcal{R})$ respectively, we have
\[
\liminf_{k \to \infty} W_2(\varphi_{1,k}, \varphi_{2,k}) \geq W_2(\varphi_1, \varphi_2).
\]

Similarly, suppose $\left(\left(\{X_{i,1, k}\}^n_{i=1}, \{Y_{j,1, k}\}^m_{j=1}\right)\right)_{k\geq 1}$ and $\left(\left(\{X_{i,2,k}\}^n_{i=1}, \{Y_{j,2,k}\}^m_{j=1}\right)\right)_{k\geq 1}$ are tracially bi-partite systems in left-right, tracially bi-partite, C$^*$-non-commutative probability spaces $(\A_1, \mathcal{L}_1, \mathcal{R}_1, \varphi_1)$ and $(\A_2, \mathcal{L}_2, \mathcal{R}_2, \varphi_2)$ respectively that converge in distributions to $\left(\{X_{i,1}\}^n_{i=1}, \{Y_{j,1}\}^m_{j=1}\right)$ and $\left(\{X_{i,2}\}^n_{i=1}, \{Y_{j,2}\}^m_{j=1}\right)$ in $(\A_1, \mathcal{L}_1, \mathcal{R}_1, \varphi_1)$ and $(\A_2, \mathcal{L}_2, \mathcal{R}_2, \varphi_2)$ respectively and for which there is a uniform bound on all operator norms of all operators.  Then
\begin{align*}
\liminf_{k \to \infty} & W_2\left(\left(\{X_{i,1,k}\}^n_{i=1}, \{Y_{j,1,k}\}^m_{j=1}\right), \left(\{X_{i,2,k}\}^n_{i=1}, \{Y_{j,2,k}\}^m_{j=1}\right)    \right) \\
&\geq W_2\left(\left(\{X_{i,1}\}^n_{i=1}, \{Y_{j,1}\}^m_{j=1}\right), \left(\{X_{i,2}\}^n_{i=1}, \{Y_{j,2}\}^m_{j=1}\right)    \right).
\end{align*}
\end{prop}
\begin{proof}
The result trivially follows by considering Definitions \ref{defn:Wasserstein-states} and \ref{defn:Wasserstein}.
\end{proof}

With the above bi-free analogue of the Wasserstein metric, we can now begin the proof of Theorem \ref{thm:orbital-bi-free-characterization}.  However, many lemmata will be required along the way.

\begin{proof}[Proof of Theorem \ref{thm:orbital-bi-free-characterization}.]
To begin, we will suppose $\chi_{\orb}( {\bf Z}_1, \ldots, {\bf Z}_\ell) > -\infty$ and we will only the assumption that $\chi_{\orb}( {\bf Z}_1, \ldots, {\bf Z}_\ell) = 0$ at the end of the proof.  This will enable us to develop a Talagrand-like inequality for the orbital bi-free entropy.

Let $R > \max(\{\left\|X_{i,k}\right\| \, \mid\, 1 \leq i \leq n_k, 1 \leq k \leq \ell\} \cup \{\left\|Y_{j,k}\right\| \, \mid\, 1 \leq j \leq m_k, 1 \leq k \leq \ell\})$.  By Proposition \ref{prop:other-def-of-orbital} we can chose an increasing sequence $(d_l)_{l\geq 1}$ of natural numbers such that
\[
\tilde{\chi}_{\orb, R}\left({\bf Z}_1, \ldots, {\bf Z}_\ell; l, d_l, \frac{1}{l}\right) > -\infty
\]
for all $l \in \bN$ and
\[
\chi_{\orb}( {\bf Z}_1, \ldots, {\bf Z}_\ell) = \lim_{l \to \infty} \frac{1}{d_l^2} \tilde{\chi}_{\orb, R}\left({\bf Z}_1, \ldots, {\bf Z}_\ell; l, d_l, \frac{1}{l}\right).
\]
For each $l \in \bN$, choose $(({\bf A}_{k,l})^\ell_{k=1}, ({\bf B}_{k,l})^\ell_{k=1}) \in \left( \prod^\ell_{k=1} (M_{d_l}^{\sa})^{n_k}_R\right) \times \left( \prod^\ell_{k=1} (M_{d_l}^{\sa})_R^{m_k}\right)$ such that
\[
-\infty < \tilde{\chi}_{\orb, R}\left({\bf Z}_1, \ldots, {\bf Z}_\ell; l, d_l, \frac{1}{l}\right)-1 \leq \log\left( \gamma^{\otimes \ell}_{d_l}\left( \Gamma_{\orb}\left({\bf Z}_1, \ldots, {\bf Z}_\ell: ({\bf A}_{k,l})^\ell_{k=1}, ({\bf B}_{k,l})^\ell_{k=1}; l, {d_l}, \frac{1}{l}\right)    \right)   \right).
\]
Note this implies $\gamma^{\otimes \ell}_{d_l}\left( \Gamma_{\orb}\left({\bf Z}_1, \ldots, {\bf Z}_\ell: ({\bf A}_{k,l})^\ell_{k=1}, ({\bf B}_{k,l})^\ell_{k=1}; l, {d_l}, \frac{1}{l}\right)    \right)  > 0$.

Let $SU(d)$ denote the special unitary group of $\M_d$, let $\mathbb{T}_d$ denote the set of unitaries which are scalar multiples of $I_d$, and let $\gamma_{d, s}$ denote the Haar measure on $SU(d)$.
We want to work with $SU(d)$ instead of $U(d)$ here for technical reasons.
Indeed this is possible as we note that if
\[
	(U_k)_{k=1}^\ell \in \Gamma_{\orb}\left({\bf Z}_1, \ldots, {\bf Z}_\ell: ({\bf A}_{k,l})^\ell_{k=1}, ({\bf B}_{k,l})^\ell_{k=1}; l, {d_l}, \frac{1}{l}\right) 
\]
then
\[
	(V_k U_k)_{k=1}^\ell \in \Gamma_{\orb}\left({\bf Z}_1, \ldots, {\bf Z}_\ell: ({\bf A}_{k,l})^\ell_{k=1}, ({\bf B}_{k,l})^\ell_{k=1}; l, {d_l}, \frac{1}{l}\right) 
\]
for all $(V_k)_{k=1}^\ell \in \mathbb{T}_d^\ell$.
Hence it immediately follows that  if
\[
\Gamma_l = SU(d) \cap  \Gamma_{\orb}\left({\bf Z}_1, \ldots, {\bf Z}_\ell: ({\bf A}_{k,l})^\ell_{k=1}, ({\bf B}_{k,l})^\ell_{k=1}; l, {d_l}, \frac{1}{l}\right)
\]
then
\[
\gamma^{\otimes \ell}_{d_l}\left( \Gamma_{\orb}\left({\bf Z}_1, \ldots, {\bf Z}_\ell: ({\bf A}_{k,l})^\ell_{k=1}, ({\bf B}_{k,l})^\ell_{k=1}; l, {d_l}, \frac{1}{l}\right)    \right) = \gamma_{d_l, s}^{\otimes \ell}(\Gamma_l)
\]
(so $\tilde{\chi}_{\orb, R}\left({\bf Z}_1, \ldots, {\bf Z}_\ell; l, d_l, \frac{1}{l}\right)\leq 1 + \log\left( \gamma_{d_l, s}^{\otimes \ell}(\Gamma_l)\right)$).

Let $C[-R, R]$ denote the C$^*$-algebra of continuous functions on $[-R,R]$ and let 
\[
\B_R = \left(\ast^\ell_{k=1} C[-R,R]^{\ast n_k} \right) \otimes_{\max} \left(\ast^\ell_{k=1} C[-R,R]^{\ast m_k} \right)
\]
where $\ast$ denotes the universal free product of C$^*$-algebras.
Thus, by properties of the universal free product C$^*$-algebra and the maximal tensor product, there exists a homomorphism $\pi : \B_R \to \A$ such that $\pi(x_{k,i}) = X_{k,i}$ and $\pi(y_{k,j}) = Y_{k,j}$ where $x_{k,i}$ is the identify function on $C[-R,R]$ in the $k^\th$ term of $\ast^\ell_{k=1} C[-R,R]^{\ast n_k} \subseteq \B_R$ and the $i^\th$ term of
$C[-R,R]^{\ast n_k}$, and $y_{k,j}$ is the identity function on $C[-R,R]$ in the $k^{\th}$ term of $\ast^\ell_{k=1} C[-R,R]^{\ast m_k}$ and the $j^\th$ term of $C[-R,R]^{\ast m_k}$.
Consequently, if $\varphi_{{\bf Z}_1, \ldots, {\bf Z}_\ell} = \varphi \circ \pi$, $\mathcal{L} = \left(\ast^\ell_{k=1} C[-R,R]^{\ast n_k} \right) \otimes 1$, and $\mathcal{R} = 1 \otimes \left(\ast^\ell_{k=1} C[-R,R]^{\ast m_k} \right)$, then $\mathcal{L}$ and $\mathcal{R}$ are C$^*$-subalgebras of $\B_R$ that commute with each other and $\varphi_{{\bf Z}_1, \ldots, {\bf Z}_\ell} \in \mathcal{CS}(\B_R, \mathcal{L}, \mathcal{R})$.
By similar arguments, by viewing bi-free copies of ${\bf Z}_1, \ldots, {\bf Z}_\ell$ acting on a reduced free product space, there exists a $\varphi^{\text{bi-free}}_{{\bf Z}_1, \ldots, {\bf Z}_\ell} \in \mathcal{CS}(\B_R, \mathcal{L}, \mathcal{R})$ corresponding to the bi-free distribution of ${\bf Z}_1, \ldots, {\bf Z}_\ell$.
Hence, to complete the proof, it suffices to show that $\varphi_{{\bf Z}_1, \ldots, {\bf Z}_\ell} = \varphi^{\text{bi-free}}_{{\bf Z}_1, \ldots, {\bf Z}_\ell}$.
Equivalently, by Proposition \ref{prop:Wasserstein-0-equal}, it suffices to show that
\[
W_2(\varphi_{{\bf Z}_1, \ldots, {\bf Z}_\ell}, \varphi^{\text{bi-free}}_{{\bf Z}_1, \ldots, {\bf Z}_\ell}) = 0.
\]


For a fixed $l$, for each probability measure $\mu$ on $SU(d_l)^{\ell}$ we will define $\hat{\mu} \in \mathcal{CS}(\B_R, \mathcal{L}, \mathcal{R})$ as follows.
For each $(U_k)^\ell_{k=1} \in SU(d)^{\ell}$ note there exists a $^*$-homomorphism $\pi_{(U_k)^\ell_{k=1}}$ from $\B_R$ to $\B(\M_{d_l})$ that sends $x_{k,i}$ to left multiplication by $U^*_k A_{k,i,l} U_k$ and sends $y_{k,j}$ to right multiplication by $U^*_k B_{k,j,l} U_k$.   We then desire to define
\[
\hat{\mu}(Z) = \int_{SU(d_l)^{\otimes \ell}} \tau_{d_\ell}( \pi_{(U_k)^\ell_{k=1}}(Z)I_{d_l}) \, d\mu.
\]
The fact that $\hat{\mu} \in \mathcal{CS}(\B_R, \mathcal{L}, \mathcal{R})$ follows as $\pi_{(U_k)_{k=1}^\ell}$ is a representation and $\tau_{d_\ell}$ is a trace.

With the above in hand, we need two technical lemmata on the weak$^*$-convergence of certain elements of $\mathcal{CS}(\B_R, \mathcal{L}, \mathcal{R})$.

\begin{lem}
\label{lem:weak-limit-1}
Let $\mu_l = \frac{1}{\gamma_{d_l, s}^{\otimes \ell}(\Gamma_l)} \left.\gamma_{d_l, s}^{\otimes \ell}\right|_{\Gamma_l}$.  Then the weak$^*$ limit of  $(\widehat{\mu_l})_{l\geq 1}$ is $\varphi_{{\bf Z}_1, \ldots, {\bf Z}_\ell}$.
\end{lem}
\begin{proof}
Notice for any $z = x_{k_1, i_1} \cdots x_{k_p, i_p} \otimes y_{l_1, j_1} \cdots y_{l_q, j_q} \in \B_R$ with $p + q \leq l$ that
\begin{align*}
 \widehat{\mu_l}(z) - \varphi_{{\bf Z}_1, \ldots, {\bf Z}_\ell}(z) &=  \frac{1}{\gamma_{d_l, s}^{\otimes \ell}(\Gamma_l)} \int_{\Gamma_l} \tau_{d_l}\left(U_{k_1}^*A_{k_1, i_1,l} U_{k_1} \cdots U_{k_p}^*A_{k_p, i_p,l} U_{k_p}U_{l_q}^*B_{l_q, i_q,l} U_{l_q} \cdots U_{l_1}^*B_{l_1, i_1,l} U_{l_1}\right) d(\gamma_{d_l, s}^{\otimes \ell})\\
 & \qquad - \varphi(X_{k_1, i_1} \cdots X_{k_p, i_p}Y_{l_1, j_1} \cdots Y_{l_q, j_q})  
\end{align*}
which is at most $\frac{1}{l}$ in absolute value by the definition of $\Gamma_l$.  Thus $ \widehat{\mu_l}(z)$ tends to $\varphi_{{\bf Z}_1, \ldots, {\bf Z}_\ell}(z)$ as $l$ tends to infinity for any $z \in \B_R$ thereby completing the proof.
\end{proof}

\begin{lem}
\label{lem:weak-limit-2}
The weak$^*$ limit of  $(\widehat{\gamma_{d_l, s}^{\otimes \ell}})_{l\geq 1}$ is $\varphi^{\text{bi-free}}_{{\bf Z}_1, \ldots, {\bf Z}_\ell}$.
\end{lem}
\begin{proof}
For each $M \in \bN$ and $\epsilon, \theta > 0$, Lemma \ref{lem:voi-lem-about-lots-of-matrices-making-asymptotic-freeness} implies if 
\[
\Omega(M, d_l, \epsilon) = \{(U_k)^\ell_{k=1} \in U(d_l)^{\ell} \, \mid \, (U_1^*  {\bf A}_{1,l} U_1, U_1^*{\bf B}_{1,l} U_1), \ldots, (U_\ell^*  {\bf A}_{\ell,l} U_1, U_1^*{\bf B}_{\ell,l} U_1)
\text{ are } (M, \epsilon)\text{-free}\}
\]
then $\gamma^{\otimes \ell}(\Omega(M, d_\ell, \epsilon)) > 1 - \theta$ for $d_l$ sufficiently large.

Let $\tau_{d_l}^{\ast \ast \ell}$ denote the state on $\B_R$ obtained as follows: take the reduced free product of $\ell$-copies of $\B(\M_{d_l})$ with respect to $z \mapsto \tau_{d_l}(z 1_d)$, and constructing the $^*$-homomorphism $\pi$
on $\B_R$ that sends $x_{k,i}$ to the left regular representation on the $k^{\th}$ copy of
$\B(\M_{d_l})$ acting by left multiplication by $A_{k,i,l}$ and sends $y_{k,j}$ to the right regular representation on the $k^{\th}$ copy of $\B(\M_{d_l})$ acting by right multiplication by $B_{k,j,l}$.
Then $\tau_{d_l}^{\ast \ast \ell}$ is the vacuum state on the reduced free product composed with $\pi$.
That is, $\tau_{d_l}^{\ast \ast \ell}$ is the distribution so that $\{({\bf A}_{k,l}, {\bf B}_{k,l})\}^{\ell}_{k=1}$ are bi-free with respect to the left-right matrix multiplication actions of $({\bf A}_{k,l}, {\bf B}_{k,l})$ on $\M_{d_l}$.
Note this distribution does not change if $\{({\bf A}_{k,l}, {\bf B}_{k,l})\}^{\ell}_{k=1}$ is replaced with $\{(U^*_k{\bf A}_{k,l}U_k, U^*_k{\bf B}_{k,l}U_k)\}^{\ell}_{k=1}$.

Clearly we have that $(\tau_{d_l}^{\ast \ast \ell})_{l\geq 1}$ converges weak$^*$ to $\varphi^{\text{bi-free}}_{{\bf Z}_1, \ldots, {\bf Z}_\ell}$ as 
\[
\gamma^{\otimes \ell}_{d_l}\left( \Gamma_{\orb}\left({\bf Z}_1, \ldots, {\bf Z}_\ell: ({\bf A}_{k,l})^\ell_{k=1}, ({\bf B}_{k,l})^\ell_{k=1}; l, {d_l}, \frac{1}{l}\right)    \right)  > 0.
\]
Therefore, since for all $z = x_{k_1, i_1} \cdots x_{k_p, i_p} \otimes y_{l_1, j_1} \cdots y_{l_q, j_q} \in \B_R$ with $p + q \leq M$ we have
\begin{align*}
\widehat{\gamma_{d_l, s}^{\otimes \ell}}(z)
&= \int_{SU(d_l)^{\otimes \ell}} \tau_{d_l}\left(U_{k_1}^*A_{k_1, i_1,l} U_{k_1} \cdots U_{k_p}^*A_{k_p, i_p,l} U_{k_p}U_{l_q}^*B_{l_q, i_q,l} U_{l_q} \cdots U_{l_1}^*B_{l_1, i_1,l} U_{l_1}\right) d\gamma_{d_l, s}^{\otimes \ell} \\
&= \int_{U(d_l)^{\otimes \ell}} \tau_{d_l}\left(U_{k_1}^*A_{k_1, i_1,l} U_{k_1} \cdots U_{k_p}^*A_{k_p, i_p,l} U_{k_p}U_{l_q}^*B_{l_q, i_q,l} U_{l_q} \cdots U_{l_1}^*B_{l_1, i_1,l} U_{l_1}\right) d\gamma_{d_l}^{\otimes \ell},
\end{align*}
we have that for sufficiently large $l$ that
\begin{align*}
&\left| \widehat{\gamma_{d_l, s}^{\otimes \ell}}(z) - \tau_{d_l}^{\ast \ast \ell}(z) \right| \\
&\leq \int_{\Omega(M, d_l, \epsilon)} \left| \tau_{d_l}\left(U_{k_1}^*A_{k_1, i_1,l} U_{k_1} \cdots U_{k_p}^*A_{k_p, i_p,l} U_{k_p}U_{l_q}^*B_{l_q, i_q,l} U_{l_q} \cdots U_{l_1}^*B_{l_1, i_1,l} U_{l_1}\right)  -   \tau_{d_l}^{\ast \ast \ell}(z) \right| \,  d\gamma_{d_l}^{\otimes \ell} \\
&\qquad + \int_{U(d_l)^{\ell} \setminus \Omega(M, d_l, \epsilon)} \left| \tau_{d_l}\left(U_{k_1}^*A_{k_1, i_1,l} U_{k_1} \cdots U_{k_p}^*A_{k_p, i_p,l} U_{k_p}U_{l_q}^*B_{l_q, i_q,l} U_{l_q} \cdots U_{l_1}^*B_{l_1, i_1,l} U_{l_1}\right)  -   \tau_{d_l}^{\ast \ast \ell}(z) \right| \,  d\gamma_{d_l}^{\otimes \ell}\\
& \leq \epsilon + 2(R+1)^M \theta
\end{align*}
where the first inequality follows from $(M, \epsilon)$-freeness (which gives the correct approximation of $\tau_{d_l}^{\ast \ast \ell}(z)$ by the same arguments at the beginning of Lemma \ref{lem:large-portion-good}) and the second inequality follows from operator norm estimates and our bound on $\gamma^{\otimes \ell}(\Omega(M, d_\ell, \epsilon))$.
As $\epsilon$ and $\theta$ can be made sufficiently small for any such $M$, we have that $(\widehat{\gamma_{d_l, s}^{\otimes \ell}})_{l \geq 1}$ and $(\tau_{d_l}^{\ast \ast \ell})_{l\geq 1}$ have the same weak$^*$-limit thereby completing the lemma.
\end{proof}

Now we need to know that the operation of taking a probability measure on $SU(d_l)^{\otimes \ell}$ and producing an element of  $\mathcal{CS}(\B_R, \mathcal{L}, \mathcal{R})$ is well-behaved.

\begin{lem}
\label{lem:Wasserstein-inequalities}
For any probability measures $\mu_1$ and $\mu_2$ on $SU(d_l)^{\ell}$, we have that
\[
W_2(\widehat{\mu_1}, \widehat{\mu_2}) \leq \frac{2R \sqrt{n+m}}{\sqrt{d_l}} W_{2, \left\| \, \cdot \, \right\|_{\text{HS}}}(\mu_1, \mu_2) \leq \frac{2R \sqrt{n+m}}{\sqrt{d_l}} W_{2, \left\| \, \cdot \, \right\|_{\text{geod}}}(\mu_1, \mu_2) 
\]
where $n = \max_{1 \leq k \leq \ell} n_k$, $m = \max_{1 \leq k \leq \ell} m_k$, and $W_{2, \left\| \, \cdot \, \right\|_{\text{HS}}}$ and $W_{2, \left\| \, \cdot \, \right\|_{\text{geod}}}$ are the 2-Wasserstein distances for measures with respect to the Hilbert-Schmidt norm $\left\| \, \cdot \, \right\|_{\text{HS}}$ and the geodesic distance, respectively.
\end{lem}
\begin{proof}
The proof goes along the same lines as \cite{HMU2009}*{Lemma 3.4}.  First, let $\Pi(\mu_1, \mu_2)$ denote the set of all probability measures on $SU(d_l)^{\ell} \times SU(d_l)^{\ell}$ whose left- and right- marginal measures are $\mu_1$ and $\mu_2$ respectively.  For each $\mu \in \Pi(\mu_1, \mu_2)$ we associate a state $\widehat{\mu}$ on
\[
\left(\ast^\ell_{k=1} C[-R,R]^{\ast n_k} \right) \ast \left(\ast^\ell_{k=1} C[-R,R]^{\ast n_k} \right) \otimes_{\max} \left(\ast^\ell_{k=1} C[-R,R]^{\ast m_k} \right) \ast \left(\ast^\ell_{k=1} C[-R,R]^{\ast m_k} \right)
\]
as described above (i.e. for each $((U_k)^\ell_{k=1}, (V_k)^\ell_{k=1}) \in SU(d_\ell)^\ell\times  SU(d_\ell)^\ell$, for $1 \leq k \leq \ell$ we send $x_{k,i}$ to left multiplication by $U^*_k A_{k,i,l} U_k$ and $y_{k,j}$ to right multiplication by $U^*_k B_{k,j,l} U_k$, and for $\ell + 1 \leq k \leq 2\ell$ we send $x_{k,i}$ to left multiplication by $V^*_k A_{k,i,l} V_k$ and sends $y_{k,j}$ to right multiplication by $V^*_k B_{k,j,l} V_k$).  By the definition of $W_2$ this immediately implies
\[
W_2(\widehat{\mu_1}, \widehat{\mu_2}) \leq \sqrt{\int_{SU(d_l)^{\ell}}\int_{SU(d_l)^{\ell}} \sum_{k=1}^\ell \sum^{n_k}_{i=1} \left\|U_k^* A_{k,i,l} U_k - V_k^* A_{k,i,l} V_k\right\|^2_{\text{HS}}   + \sum^{m_k}_{j=1}  \left\|U_k^* B_{k,j,l} U_k - V_k^* B_{k,j,l} V_k\right\|^2_{\text{HS}}  \, d\mu }
\]
for any $\mu \in \Pi(\mu_1, \mu_2)$ where the first integration is with respect to $(U_k)^\ell_{k=1}$ and the second is with respect to $(V_k)^\ell_{k=1}$.  Thus as
\[
\left\|U_k^* A_{k,i,l} U_k - V_k^* A_{k,i,l} V_k\right\|^2_{\text{HS}} \leq 4R^2 \left\|U_k -  V_k\right\|^2_{\text{HS}}
\]
with a similar inequality for the $B$-terms, we obtain the first inequality by the definition of the Wasserstein distances for measures with respect to the Hilbert-Schmidt norm.  

Finally, the second inequality is trivial because the geodesic distance majorizes the Hilbert-Schmidt norm distance.
\end{proof}

Using the above, under the assumption $ \chi_{\orb}( {\bf Z}_1, \ldots, {\bf Z}_\ell) > -\infty$ instead of $\chi_{\orb}( {\bf Z}_1, \ldots, {\bf Z}_\ell) = 0$, we obtain the following Talagrand-like inequality for the orbital bi-free entropy.

\begin{prop}
\label{prop:Talagrand}
Under the above notation and assumptions, 
\[
W_2(\varphi_{{\bf Z}_1, \ldots, {\bf Z}_\ell}, \varphi^{\text{bi-free}}_{{\bf Z}_1, \ldots, {\bf Z}_\ell}) \leq 4R \sqrt{n+m} \sqrt{- \chi_{\orb}( {\bf Z}_1, \ldots, {\bf Z}_\ell) }
\]
where $n = \max_{1 \leq k \leq \ell} n_k$ and $m = \max_{1 \leq k \leq \ell} m_k$.
\end{prop}
\begin{proof}
The proof is near identical to \cite{HMU2009}*{Proposition 3.5}.
Indeed since the Ricci curvature of $SU(d_l)^{\ell}$ (with respect to the inner product induced by the real part of the unnormalized trace) is known to be constant and equal to $\frac{d_l}{2}$, the transportation cost inequality
\[
W_{2, \text{geod}}\left(\mu_l, \gamma^{\otimes \ell}_{d_l, s}\right) \leq \sqrt{\frac{4}{d_l} S\left(\mu_l, \gamma^{\otimes \ell}_{d_l, s}\right)}
\]
holds by \cite{OV2000}, where $\mu_l$ is as in Lemma~\ref{lem:weak-limit-1}, $S\left(\mu_l, \gamma^{\otimes \ell}_{d_l, s}\right)$ denotes the relative entropy of $\mu_l$ with respect to $\gamma^{\otimes \ell}_{d_l, s}$ and thus
\[
S\left(\mu_l, \gamma^{\otimes \ell}_{d_l, s}\right) = - \log\left(  \gamma^{\otimes \ell}_{d_l, s}(\Gamma_l)\right)
\]
by the definitions.  By Lemma \ref{lem:Wasserstein-inequalities}, we obtain that
\[
W_{2}\left(\widehat{\mu_l}, \widehat{\gamma^{\otimes \ell}_{d_l, s}}\right) \leq 4R\sqrt{n+m}\sqrt{-\frac{1}{d_l^2} \log\left(  \gamma^{\otimes \ell}_{d_l, s}(\Gamma_l)\right)}.
\]
Hence Proposition \ref{prop:Wass-lower-semicontinuous}, Lemma \ref{lem:weak-limit-1}, and Lemma \ref{lem:weak-limit-2} yield the result by taking $l$ to infinity.
\end{proof}

Now, to complete the proof of Theorem \ref{thm:orbital-bi-free-characterization}.  Indeed as $\chi_{\orb}( {\bf Z}_1, \ldots, {\bf Z}_\ell) = 0$, we obtain that
\[
	W_2(\varphi_{{\bf Z}_1, \ldots, {\bf Z}_\ell}, \varphi^{\text{bi-free}}_{{\bf Z}_1, \ldots, {\bf Z}_\ell}) = 0,
\]
thereby showing that ${\bf Z}_1$, ${\bf Z}_2$, $\ldots$, ${\bf Z}_\ell$ are bi-free by Proposition \ref{prop:Wasserstein-0-equal}.
\end{proof}

\section{Calculating Microstate Entropy}
\label{sec:Calc}

In this section, we will compute the microstate bi-free entropy of several collections.  We begin with the cases where there is a `linear dependence in distribution'.

\begin{lem}
	\label{lem:micro-neg-infinity-if-repeat-operator}
	Let $(\A, \varphi)$ be a C$^*$-non-commutative probability space and let $X, Y \in \A$ be self-adjoint such that $\varphi(X) = \varphi(Y) = 0$ and $\varphi(X^2) = \varphi(Y^2) = \varphi(XY) = 1$.  Then
	\[
		\chi(X \sqcup Y) = -\infty.
	\]
\end{lem}

\begin{proof}
	Fix $R > \max\set{\norm{X}, \norm{Y}}$.
	Notice that it suffices to show that $\chi_R(X\sqcup Y; 2, \epsilon) \to -\infty$ as $\epsilon \to 0$.
	Towards this end, notice that for any $1 > \epsilon > 0$ and $d \in \bN$,
	\[
		\Gamma_R(X\sqcup Y; 2, d, \epsilon)
		\subseteq \set{(A, B) \in (M_d^{\sa})^2 \,\mid\, \tau_d(A^2), \tau_d(B^2), \tau(AB) \in (1-\epsilon, 1+\epsilon)}.
	\]
	Recall, however, that the Lebesgue measure used is normalized based on the inner product given by the \emph{unnormalized} trace:
	\[
		\langle A, B\rangle_{\M_d^{\sa}} = d \tau_d(B^*A) = \Tr_d(B^*A) = \langle A, B \rangle_{\bR^{d^2}}.
	\]
	The three conditions on the set above then become $\norm{A}^2_{\bR^{d^2}}, \norm{B}^2_{\bR^{d^2}}, \ang{A, B}_{\bR^{d^2}} \in (d(1-\epsilon), d(1+\epsilon))$.
	These restrictions allow us to deduce a bound on the angle $\theta_{A,B}$ between any $A$ and $B$ in the set:
	\[\cos\theta_{A, B} = \frac{\ang{A, B}}{\norm{A}\norm{B}} \geq \frac{1-\epsilon}{1+\epsilon}
		\qquad\text{whence}\qquad
	\tan\theta_{A,B} = \frac{\sqrt{1-\cos^2\theta_{A,B}}}{\cos\theta_{A,B}} \leq \frac{\sqrt{1 - \paren{\frac{1-\epsilon}{1+\epsilon}}^2}}{\frac{1-\epsilon}{1+\epsilon}} = \frac{2\sqrt\epsilon}{1-\epsilon}.\]
	Consequently $B$ must lie in the cone from the origin in the direction of $A$ with height $\sqrt{d(1+\epsilon)}$ and radius at its base $\sqrt{d(1+\epsilon)}\frac{2\sqrt\epsilon}{1-\epsilon}.$
	Letting $C\paren{A,d,\epsilon}$ represent this cone, we have
	\[
		\Gamma_R(X\sqcup Y; 2,d,\epsilon) \subseteq \set{(A, B) \in \bR^{2d^2} \,\mid\, \norm{A}^2 \leq d(1+\epsilon), B \in C\paren{A, d, \epsilon}}.
	\]
	Since volume of the cone $C\paren{A, d, \epsilon}$ does not depend on $A$, the volume of the set on the right hand side is the product of that of the ball $\mathcal{B}\paren{d^2, \sqrt{d(1+\epsilon)}}$ of radius $\sqrt{d(1+\epsilon)}$ in dimension $d^2$, and that of any cone $C\paren{A, d, \epsilon}$.
	Fortunately, both volumes are known:
	\begin{align*}
		\lambda_{d^2}\paren{\mathcal{B}\paren{d^2, \sqrt{d(1+\epsilon)}}} &= \frac{\pi^{\frac{d^2}2}}{\Gamma\paren{\frac{d^2}{2}+1}}\paren{d(1+\epsilon)}^{\frac{d^2}{2}}, \qand \\
		\lambda_{d^2}\paren{C\paren{A, d, \epsilon}}
		&= \frac1{d^2}\paren{\sqrt{d(1+\epsilon)}}\lambda_{d^2-1}\paren{\mathcal{B}\paren{d^2-1, \sqrt{d(1+\epsilon)}\frac{2\sqrt\epsilon}{1-\epsilon}}} \\
		&= \frac1{d^2}\paren{\sqrt{d(1+\epsilon)}}\paren{\frac{\pi^{\frac{d^2-1}{2}}}{\Gamma\paren{\frac{d^2-1}2+1}}\paren{\sqrt{d(1+\epsilon)}\frac{2\sqrt\epsilon}{1-\epsilon}}^{d^2-1}} \\
		&= \frac1{d^2}\paren{d(1+\epsilon)}^{\frac{d^2}2}\paren{\frac{\pi^{\frac{d^2-1}{2}}}{\Gamma\paren{\frac{d^2-1}2+1}}\paren{\frac{2\sqrt\epsilon}{1-\epsilon}}^{d^2-1}}.
	\end{align*}

	We now recall that Stirling's formula allows us to make the estimate that for large $z>0$, $\frac1z\log\Gamma(z) = \log z + \mathcal{O}(1)$.
	This allows us to make the following estimate:
	\begin{align*}
		\frac1{d^2}\chi_R(X\sqcup Y; 2,d,\epsilon)
		&\leq \log(d) - \frac1{d^2}\log\Gamma\paren{\frac{d^2}2+1} - \frac1{d^2}\log\Gamma\paren{\frac{d^2-1}2+1} + \frac{d^2-1}{2d^2}\log\epsilon 
		+ \mathcal{O}_{d,\epsilon}\paren{1} \\
		&= -\log(d) + \frac{d^2-1}{2d^2}\log\epsilon + \mathcal{O}_{d,\epsilon}\paren{1}.
	\end{align*}
	Thus $\chi_R(X\sqcup Y; 2, \epsilon) \leq \frac12\log\epsilon + \mathcal{O}_\epsilon\paren{1}$ so sending $\epsilon \to 0$ yields $\chi(X\sqcup Y) = -\infty$.
\end{proof}

Using the above, we can prove the following which, when combined with Corollary \ref{cor:micro-linear-transformations}, completely determines the microstate bi-free entropy of a tracially bi-partite system with a linear dependence in distribution.

\begin{thm}
\label{thm:micro-LD}
Let $(\{X_i\}^n_{i=1}, \{Y_j\}^m_{j=1})$ be a tracially bi-partite system in  a C$^*$-non-commutative probability space $(\A, \varphi)$.  If there exists an $X \in \mathrm{span}\{X_1, \ldots, X_n\}$ and a $Y \in \mathrm{span}\{Y_1, \ldots, Y_n\}$ such that $1 = \varphi(X^2) = \varphi(XY) = \varphi(Y^2)$ (e.g. $X_1, \ldots, X_n$ linearly independent, $Y_1, \ldots, Y_m$ linearly independent, yet $X_1, \ldots, X_n, Y_1, \ldots, Y_m$ linearly dependent in distribution), then
\[
\chi( \smbfe ) = -\infty.
\]
\end{thm}

\begin{proof}
If $X_1, \ldots, X_n$ or $Y_1, \ldots, Y_n$ are linearly dependent, then the result follows from Corollary \ref{cor:micro-linear-transformations}.  Otherwise there exists an $i \in \{1,\ldots, n\}$ and a $j \in \{1,\ldots, m\}$ such that $\{X_1, \ldots, X_n\}$ and $\{X_1, \ldots, X_{i-1}, X, X_{i+1}, \ldots, X_n\}$  are bases for the same subspace of $\A$, and $\{Y_1, \ldots, Y_m\}$ and $\{Y_1, \ldots, Y_{j-1}, Y, Y_{j+1}, \ldots, Y_m\}$  are bases for the same subspace of $\A$.  By Corollary \ref{cor:micro-linear-transformations} there exists a $C \in \bR$ such that
\begin{align*}
\chi( \smbfe ) = C + \chi(X_1, \ldots, X_{i-1}, X, X_{i+1}, \ldots, X_n \sqcup Y_1, \ldots, Y_{j-1}, Y, Y_{j+1}, \ldots, Y_m).
\end{align*}
As
\begin{align*}
\chi(X_1,& \ldots, X_{i-1}, X, X_{i+1}, \ldots, X_n \sqcup Y_1, \ldots, Y_{j-1}, Y, Y_{j+1}, \ldots, Y_m) \\
& \leq \chi(X \sqcup Y) + \chi(X_1, \ldots, X_{i-1}, X_{i+1}, \ldots, X_n \sqcup Y_1, \ldots, Y_{j-1}, Y_{j+1}, \ldots, Y_m)
\end{align*}
by Proposition \ref{prop:micro-subadditive} and as $\chi( X_1, \ldots, X_{i-1}, X_{i+1}, \ldots, X_n \sqcup Y_1, \ldots, Y_{j-1}, Y_{j+1}, \ldots, Y_m) < \infty$ by Proposition \ref{prop:not-plus-infinity}, Lemma \ref{lem:micro-neg-infinity-if-repeat-operator} yields $\chi(X \sqcup Y) = -\infty$ and the result.
\end{proof}

Next we investigate the bi-free entropy of bi-free central limit distributions.
Since we are only able to apply transformations to the left variables and the right variables separately, we cannot directly remove correlations between left and right semicircular variables.
We therefore start with the case of two variables.

\begin{thm}
\label{thm:entropy-pair-of-semis}
Let $(\A, \varphi)$ be a C$^*$-non-commutative probability space and let $(S_\ell, S_r)$ be a centred, self-adjoint bi-free central limit distribution in $\A$ in which each variable is of variance one.
If $c= \varphi(S_\ell S_r) \in [-1, 1]$, then
\[
\chi(S_\ell \sqcup S_r) = \log(2 \pi e) + \frac{1}{2}\log(1-c^2).
\]
Furthermore, the $\limsup_{d \to \infty}$ when computing $\chi(S_\ell \sqcup S_r)$ is actually a $\lim_{d \to \infty}$.
\end{thm}
\begin{proof}
By Example \ref{exam:changing-semis-to-free}, we see that
\[
\chi(S_\ell \sqcup S_r) \geq \log(2 \pi e) + \frac{1}{2}\log(1-c^2)
\]
as the free entropy of a single semicircular operator of variance one is $\frac{1}{2}\log(2 \pi e)$.
Furthermore, we claim this inequality holds if we use the $\liminf_{d \to \infty}$ in place of $\limsup_{d \to \infty}$ for $\chi(S_\ell \sqcup S_r)$.  To see this, notice Example \ref{exam:changing-semis-to-free} holds for the $\liminf_{d \to \infty}$ version since both \cite{V1993}*{Proposition 3.5 and Proposition 5.4}  and Theorem \ref{thm:micro-converting-rights-to-lefts} do as well.  Therefore, since the free entropy of a single semicircular operator agrees with the $\liminf_{d \to \infty}$ variety, the claim is complete.

For the other direction, we will apply some volume arguments.  Note the case $c = \pm 1$ follows from Lemma \ref{lem:micro-neg-infinity-if-repeat-operator}, so we will assume $|c| < 1$.

For each $R > 2$ and $M \in \bN$ with $M \geq 2$, notice that $\Gamma_R(S_\ell \sqcup S_r ; M, d, \epsilon)$ is contained in
\[
\Psi := \left\{(A_1, A_2) \in (\M_d^{\sa})^2 \, \left| \,  1- \epsilon \leq \tau_d(A_k^2) \leq 1 + \epsilon, c- \epsilon \leq \tau_d(A_1A_2) \leq c + \epsilon \right. \right\}.
\]
We desire an estimate on the Lebesgue measure of $\Psi$. 

Recall we view $(\M_d^{\sa})^2 \cong (\bR^{d^2})^2$ as Hilbert spaces where for $A, B \in \M_d^{\sa}$ we have 
\[
\langle A, B\rangle_{\M_d^{\sa}} = d \tau_d(B^*A) = \Tr(B^*A) = \langle A, B \rangle_{\bR^{d^2}}.
\]
Hence
\[
\Psi \cong \left\{(A_1, A_2) \in (\bR^{d^2})^2 \, \left| \,  \sqrt{d(1-\epsilon)} \leq \left\|A_k\right\|_2 \leq \sqrt{d(1+\epsilon)}, d(c-\epsilon) \leq \langle A_1, A_2\rangle_{\bR^{d^2}} \leq d(c+\epsilon)\right. \right\}.
\]

Consider the map $\Theta : (\bR^{d^2})^2 \to (\bR^{d^2})^2$ defined by
\[
\Theta(A_1, A_2) = \left(A_1, -\frac{c}{\sqrt{1-c^2}} A_1 + \frac{1}{\sqrt{1-c^2}} A_2\right).
\]
Clearly $\Theta$ is a direct sum of $d^2$ copies of the matrix
\[
Q = \begin{bmatrix}
1 & 0 \\
-\frac{c}{\sqrt{1-c^2}} & \frac{1}{\sqrt{1-c^2}}
\end{bmatrix}
\]
via a specific choice of orthonormal basis of $\bR^{d^2}$.  
Hence the Jacobian of $\Theta$ is also a direct sum of $d^2$ copies of $Q$ and thus
\[
\mathrm{Vol}(\Psi) = \frac{1}{\det(\J(\Theta))} \mathrm{Vol}(\Theta(\Psi)) = (1-c^2)^{\frac{d^2}{2}} \mathrm{Vol}(\Theta(\Psi)).
\]

To obtain an upper bound for the volume of $\Theta(\Psi)$, we claim that
\[
\Theta(\Psi) \subseteq \left\{(B_1, B_2) \in (\bR^{d^2})^2 \, \left| \, \left\|B_k\right\|_2 \leq \sqrt{d \left(1 + \epsilon \frac{(1+|c|)^2}{1-c^2}\right)}\right. \right\}.
\]
To see this, fix $(A_1, A_2) \in \Psi$ and let $(B_1, B_2) = \Theta(A_1, A_2)$.  Then $B_1 = A_1$ so 
\[
\left\|B_1\right\|_2 \leq \sqrt{d(1 + \epsilon)} \leq \sqrt{d \left(1 + \epsilon \frac{(1+|c|)^2}{1-c^2}\right)}.
\]
Next notice that
\begin{align*}
\left\|B_2\right\|_2^2 &= \left\langle -\frac{c}{\sqrt{1-c^2}} A_1 + \frac{1}{\sqrt{1-c^2}} A_2, -\frac{c}{\sqrt{1-c^2}} A_1 + \frac{1}{\sqrt{1-c^2}} A_2\right\rangle_{\bR^{d^2}} \\
&= \frac{1}{1-c^2} \left(c^2 \langle A_1, A_1 \rangle - 2 c \langle A_1, A_2\rangle + \langle A_2, A_2\rangle    \right) \\
&\leq \frac{1}{1-c^2} \left( d (1+\epsilon) c^2 - 2d c^2 + 2d|c| \epsilon + d(1+\epsilon)\right) \\
&= \frac{d}{1-c^2} \left((1-c^2) + \epsilon(1 + 2|c| + c^2)\right) \\
&= d \left(1 + \epsilon \frac{(1+|c|)^2}{1-c^2}\right).
\end{align*}
Hence the claim is complete.

Using the above and the fact that $\Theta(\Psi)$ is contained in the product of two $d^2$-dimensional balls of radius $\sqrt{ d \left( 1 + \epsilon \frac{(1+|c|)^2}{1-c^2}\right)}$, we obtain that
\begin{align*}
&\lambda_{d, 2}(\Gamma_R(S_\ell \sqcup S_r ; M, d, \epsilon)) \\
& \leq  \mathrm{Vol}(\Psi)\\
& \leq (1-c^2)^{\frac{d^2}{2}} \mathrm{Vol}(\Theta(\Psi)) \\
& \leq ((1-c^2)^{\frac{d^2}{2}}   \left( \frac{\pi^{\frac{d^2}{2}}}{\Gamma\left(\frac{d^2}{2} + 1\right)} \left(d \left(1 + \epsilon \frac{(1+|c|)^2}{1-c^2}\right)\right)^{\frac{d^2}{2}}  \right)^2.
\end{align*}
Hence, via an application of Stirling's formula, we obtain that
\begin{align*}
& \chi_R(S_\ell \sqcup S_r; M, \epsilon) \\
& \leq \limsup_{d \to \infty} \frac{1}{2}\log(1 - c^2) + \log(\pi) - 2\frac{1}{d^2} \log\left( \Gamma\left(\frac{d^2}{2} + 1\right)\right) + \log(d) + \log\left(1 + \epsilon \frac{(1+|c|)^2}{1-c^2}  \right) + \frac{2}{2}\log(d) \\
& \leq \limsup_{d \to \infty} \frac{1}{2}\log(1 - c^2) + \log(\pi) + \log(2e) + \log\left(1 + \epsilon \frac{(1+|c|)^2}{1-c^2}   \right).
\end{align*}
Therefore
\[
\chi(S_\ell \sqcup S_r) \leq \log(2\pi e) + \frac{1}{2}\log(1 - c^2)
\]
completing the claim.
\end{proof}

Combining all of the results of this paper, we obtain the following.

\begin{thm}
\label{thm:micro-entropy-all-central-limits}
Let $(\{S_k\}^{n}_{k=1}, \{S_k\}^{n+m}_{k=n+1})$ be a centred self-adjoint bi-free central limit distribution with respect to $\varphi$ with $\varphi(S^2_k) = 1$ for all $k$.  Recall that the joint distribution is completely determined by the positive matrix
\[
A = [a_{i,j}] = [\varphi(S_iS_j)] \in \M_n(\bR).
\]
Then
\[
\chi( S_1, \ldots, S_n \sqcup S_{n+1}, \ldots, S_{n+m} ) = \frac{n+m}{2} \log(2 \pi e) + \frac{1}{2}\log\left(\det(A)\right).
\]
\end{thm}
\begin{proof}
Note that if $A$ is not invertible then either $\{S_k\}^{n}_{k=1}$ are linearly dependent (in distribution), $\{S_k\}^{n+m}_{k=n+1}$ are linearly dependent (in distribution), or the hypotheses of Theorem \ref{thm:micro-LD} are satisfied.
Hence, by Corollary \ref{cor:micro-linear-transformations}, the result holds if $A$ is not invertible.
Thus we will suppose that $A$ is invertible.

Recall that we can view $(\{S_k\}^{n}_{k=1}, \{S_k\}^{n+m}_{k=n+1})$ as left and right semicircular operators acting on a real Fock space.
In particular for $k \in \{1,\ldots, n\}$ we can write
\[
S_k = l(e_k) + l^*(e_k)
\]
and for $k \in \{n+1, \ldots, n+m\}$ we can write
\[
S_k = r(e_k) + r^*(e_k)
\]
where $\{e_k\}^{n+m}_{k=1} \in \H$ are unit vectors.
Note 
\[
A = [\langle e_i, e_j\rangle]
\]
so we obtain that $\{e_k\}^{n+m}_{k=1}$ is linearly independent.

We now discuss how modifications to $\{e_k\}^n_{k=1}$ and modifications to $\{e_k\}^{n+m}_{k=n+1}$ modify the bi-free entropy and the covariance matrix.
Suppose $Q = [q_{i,j}] \in \M_n(\bR)$ and $R = [r_{i,j}] \in \M_m(\bR)$ are invertible.
If for each $k \in \{1,\ldots, n\}$ we define
\[
e'_k = \sum^n_{i=1} q_{k,i} e_i
\]
and for each $k \in \{n+1, \ldots, n+m\}$ we define
\[
e'_k = \sum^{m}_{j=1} r_{k,j} e_{j+n}
\]
then $\{e'_k\}^{n+m}_{k=1}$ is linearly independent, 
\begin{align*}
& \chi(l(e'_1) + l^*(e'_1), \ldots, l(e'_n) + l^*(e'_n) \sqcup r(e'_{n+1}) + r^*(e'_{n+1}), \ldots, r(e_{n+m}) + r^*(e_{n+m})) \\
&= \chi\left( \sum^n_{i=1} q_{1,i} S_i, \ldots, \sum^n_{i=1} q_{n,i} S_i \sqcup  \sum^{m}_{j=1} r_{1,j} S_{j+n}, \ldots, \sum^{m}_{j=1} r_{m,j} S_{j+n}\right) \\
&= \chi( S_1, \ldots, S_n \sqcup S_{n+1}, \ldots, S_{n+m} ) + \log(|\det(Q)|) + \log(|\det(R)|)
\end{align*}
by Corollary \ref{cor:micro-linear-transformations}, and
\[
[\langle e'_i, e'_j\rangle] = (Q \oplus R)[\langle e_i, e_j\rangle](Q \oplus R)^*.
\]
Thus
\[
\frac{1}{2}\log(\abs{\det([\langle e'_i, e'_j\rangle])}) =\frac{1}{2}\log(\abs{\det([\langle e_i, e_j\rangle])}) + \log(\abs{\det(Q)}) + \log(\abs{\det(R)}).
\]
Therefore, as both sides of the claimed formula
\[
\chi( S_1, \ldots, S_n \sqcup S_{n+1}, \ldots, S_{n+m} ) = \frac{n+m}{2} \log(2 \pi e) + \frac{1}{2}\log\left(\det(A)\right)
\]
are preserved under such operations, we will apply such operations until we arrive at a case we can deduce from previous results.

First, as applying the Gram-Schmidt Orthogonalization Process to $\{e_k\}^{n}_{k=1}$ and to $\{e_k\}^{n+m}_{k=n+1}$ produces such matrices $Q$ and $R$ due to linear independence of $\{e_k\}^{n+m}_{k=1}$, we may assume that $\{e_k\}^{n}_{k=1}$ is orthonormal and $\{e_k\}^{n+m}_{k=n+1}$ is orthonormal.  In this case
\[
A = \begin{bmatrix}
I_n & B \\ B^* & I_m
\end{bmatrix}
\]
where $B$ is an $n \times m$ matrix with real entries.
Let us assume that $m > n$ (the other case being similar). Whence then there are $m - n$ columns of $B$ that are linear combinations of the other $n$ columns of $B$.
Let $\{j_1, \ldots, j_n\}$ denote the indices of these other $n$ columns of $B$.
Notice since $\{e_k\}^{n+m}_{k=n+1}$ is linearly independent set of $m$ vectors that we can replace $e_k$ where $k \geq n+1$ and $k \neq j_q$ for all $q$ with $e_k - \sum^n_{q=1} c_{k,q} e_{j_q}$ (where the $c_{k,q}$ are chosen based on how column $k$ of $B$ is a linear combination of columns $j_1, \ldots, j_n$) so that $\{e_k\}^{n+m}_{k=n+1}$ remains a linearly independent set and so that $\langle e_k, e_p\rangle = 0$ for all $k \geq n+1$ with $k \neq j_q$ for all $q$, and all $p \leq n$.
Subsequently, if we apply the Gram-Schmidt Orthogonalization Process first to the modified $e_k$ for $k \neq j_q$ for all $q$, and then the remainder of the $e_k$, and if we then permute the order of the resulting vectors, the resulting change of basis matrix can then, with the above arguments, be used so that we may assume
\[
A = \begin{bmatrix}
I_n & C & 0_{n, m-n} \\ C^* & I_n & 0_{n, m-n} \\ 0_{m-n, n} & 0_{m-n, n} & I_{m-n}
\end{bmatrix}
\]
where $C$ is an $n \times n$ matrix with real entries.  

Recall, by the Singular Value Decomposition, we can write $C = UDV$ where $U, V \in \M_n(\bR)$ are unitary matrices and $D = \diag(d_1,\ldots, d_n)$ is a diagonal matrix.  By using $Q= U$ and $R = V^* \oplus I_{m-n}$, we reduce to the case where
\[
A = \begin{bmatrix}
I_n & D & 0_{n, m-n} \\ D^* & I_n & 0_{n, m-n} \\ 0_{m-n, n} & 0_{m-n, n} & I_{m-n}
\end{bmatrix}.
\]
Notice in this case that the determinant of $A$ is $\prod^n_{k=1} (1-d_k^2)$.  Furthermore, in this case, we obtain that
\[
(S_1, S_{n+1}), (S_2, S_{n+2}), \ldots, (S_n, S_{2n}), (I, S_{2n+1}), \ldots, (I, S_{n+m})
\]
are bi-freely independent.  Therefore, as pairs of semicirculars have finite-dimensional approximants and as Theorem \ref{thm:entropy-pair-of-semis} implies the $\limsup_{d \to \infty}$ for pairs of semicirculars is actually a limit, Theorem \ref{thm:micro-bi-free-additive} implies that
\[
\chi( S_1, \ldots, S_n \sqcup S_{n+1}, \ldots, S_{n+m} ) = \sum^n_{k=1} \chi(S_k \sqcup S_{k+n}) + \sum^{n+m}_{j=2n+1} \chi(S_j).
\]
As Theorem \ref{thm:entropy-pair-of-semis} implies that
\[
 \chi(S_k \sqcup S_{k+n}) = \log(2\pi e) + \frac{1}{2}\log(1-d_k^2),
\]
and as we know
\[
\chi(S_j) = \frac{1}{2} \log(2\pi e),
\]
we obtain that 
\begin{align*}
\chi( S_1, \ldots, S_n \sqcup S_{n+1}, \ldots, S_{n+m} ) &= \frac{n+m}{2} \log(2\pi e) + \frac{1}{2}\sum^n_{k=1} \log(1-d_k^2) \\
&= \frac{n+m}{2} \log(2\pi e) + \frac{1}{2}\log(|\det(A)|). \qedhere
\end{align*}
\end{proof}

\begin{rem}
Note that Theorem \ref{thm:micro-entropy-all-central-limits} includes the free case (i.e. when $m = 0$).  However, the proof for the microstate free entropy of free central limit distributions is substantially easier as one may apply transformations to all of the variables.  The bi-free proof is more difficult as Section \ref{sec:Trans} did not demonstrate the ability to mix left and right variables.   Still it is not surprising that we get the same result as the free case seeing as, asymptotically, almost all matrices are microstates for semicircular operators so it is simply a matter of angles.  One would expect other random variables which may have more complicated microstate sets could lead to different behaviours for which the above angle arguments would not apply. 
\end{rem}

\section{Microstate Bi-Free Entropy Dimension}
\label{sec:Entropy-Dimension}

For the sake of completeness, we briefly study the microstate bi-free entropy dimension.  Unfortunately, we do not know the correct bi-free generalizations of the known von Neumann algebra implications of free entropy dimension.

\begin{defn}
\label{defn:entropy-dimension}
Let $(\A, \varphi)$ be a C$^*$-non-commutative probability space and let $X_1, \ldots, X_n, Y_1, \ldots, Y_m$ be self-adjoint operators in $\A$.  The \emph{$n$-left, $m$-right, microstate bi-free entropy dimension} is defined by
\[
\delta(X_1, \ldots, X_n \sqcup Y_1, \ldots, Y_m) = n+m + \limsup_{\epsilon \to 0^+} \frac{\chi(X_1 + \sqrt\epsilon S_1, \ldots, X_n+ \sqrt\epsilon S_n \sqcup Y_1+ \sqrt\epsilon T_1, \ldots, Y_m+ \sqrt\epsilon T_m)}{|\log(\sqrt\epsilon)|}
\]
where $\{(S_i, I)\}^{n}_{i=1} \cup \{(I, T_j)\}^m_{j=1}$ is a bi-free central limit distribution of semicircular operators with variances 1 and covariances 0 that is bi-free from $(\{X_i\}^n_{i=1}, \{Y_j\}^m_{j=1})$.
\end{defn}

It is elementary to see based on bi-freeness that the self-adjoint operators $(\{X_i + \sqrt\epsilon S_i\}^n_{i=1}, \{Y_j + \sqrt\epsilon T_i\}^m_{j=1})$ still form a tracially bi-partite collection and thus $\delta(X_1, \ldots, X_n \sqcup Y_1, \ldots, Y_m)$ is well-defined.  In addition, a few basis properties of free entropy dimension carry-forward to the bi-free setting.

\begin{prop}
If $0 \leq p \leq n$ and $0 \leq q \leq m$ then
\begin{align*}
\delta(\smbfe) &\leq  \delta(X_1, \ldots, X_p\sqcup Y_1,\ldots Y_q) + \delta(X_{p+1}, \ldots, X_n\sqcup Y_{q+1},\ldots Y_m).
\end{align*}
In particular,
\[
\delta(\smbfe) \leq \delta(X_1, \ldots, X_n) + \delta(Y_1, \ldots, Y_m).
\]
\end{prop}
\begin{proof}
This result immediately follows from Definition \ref{defn:entropy-dimension} and Proposition \ref{prop:micro-subadditive}.
\end{proof}

\begin{prop}
Let $(\{X_i\}^n_{i=1}, \{Y_j\}^m_{j=1})$ be a tracially bi-partite system and let $\{(S_i, I)\}^{n}_{i=1} \cup \{(I, T_j)\}^m_{j=1}$ is a bi-free central limit distribution of semicircular operators with variances 1 and covariances 0 that is bi-free from $(\{X_i\}^n_{i=1}, \{Y_j\}^m_{j=1})$.  Suppose that for some $0 \leq p \leq n$ and $0 \leq q \leq m$ that
\[
(\alg(X_1, \ldots, X_p), \alg(Y_1, \ldots, Y_q)) \qqand (\alg(X_{p+1}, \ldots, X_{n}), \alg(Y_{q+1}, \ldots, Y_m))
\]
are bi-free and that 
\begin{align*}
&\{X_1 + \sqrt\epsilon S_1, \ldots, X_p+ \sqrt\epsilon S_p\} \sqcup \{ Y_1+ \sqrt\epsilon T_1, \ldots, Y_q+ \sqrt\epsilon T_q \} \qqand \\
&\{X_{p+1}+ \sqrt\epsilon S_{p+1}, \ldots, X_{n}+ \sqrt\epsilon S_n\} \sqcup\{Y_{q+1}+ \sqrt\epsilon T_{q+1}, \ldots, Y_m+ \sqrt\epsilon T_m\}
\end{align*}
have finite-dimensional approximants for all $\epsilon$. Then
\[
\delta( \smbfe) = \delta(X_1, \ldots, X_p \sqcup Y_1, \ldots, Y_q ) + \delta(X_{p+1}, \ldots, X_{n}\sqcup Y_{q+1}, \ldots, Y_m).
\]
\end{prop}
\begin{proof}
This result immediately follows from Definition \ref{defn:entropy-dimension}, bi-freeness, and Theorem \ref{thm:micro-bi-free-additive}.
\end{proof}

What is most interesting about microstate bi-free entropy dimension is its value of bi-free central limit distributions.

\begin{thm}
Let $(\{S_k\}^{n}_{k=1}, \{S_k\}^{n+m}_{k=n+1})$ be a centred self-adjoint bi-free central limit distribution with respect to $\varphi$ with $\varphi(S^2_k) = 1$ for all $k$.  Recall that the joint distribution is completely determined by the positive matrix
\[
A = [a_{i,j}] = [\varphi(S_iS_j)] \in \M_n(\bR).
\]
Then
\[
\delta( S_1, \ldots, S_n \sqcup S_{n+1}, \ldots, S_{n+m} ) = \mathrm{rank}(A).
\]
\end{thm}
\begin{proof}
Let $(\{T_k\}^{n}_{k=1}, \{T_k\}^{n+m}_{k=n+1})$ be a centred self-adjoint bi-free central limit distribution with respect to $\varphi$ with 
\[
\varphi(T_iT_j) = \begin{cases}
1 & \text{if }i = j \\
0 & \text{if }i \neq j
\end{cases}.
\]
If we define $Z_{k,\epsilon} = S_k + \sqrt\epsilon T_k$ for all $1 \leq k \leq n+m$, then $(\{Z_k\}^{n}_{k=1}, \{Z_k\}^{n+m}_{k=n+1})$ is a centred self-adjoint bi-free central limit distribution with respect to $\varphi$ with 
\[
\varphi(Z_{i,\epsilon}Z_{j,\epsilon}) = \begin{cases}
1 + \epsilon & \text{if }i = j \\
\varphi(S_iS_j) & \text{if }i \neq j
\end{cases}
\]
and
\[
\delta( S_1, \ldots, S_n \sqcup S_{n+1}, \ldots, S_{n+m} ) = n+m + \limsup_{\epsilon \to 0+} \frac{\chi(Z_{1,\epsilon}, \ldots, Z_{n,\epsilon} \sqcup Z_{n+1,\epsilon}, \ldots, Z_{n+m,\epsilon})}{|\log(\sqrt\epsilon)|}.
\]
By applying Corollary \ref{cor:micro-linear-transformations} and Theorem \ref{thm:micro-entropy-all-central-limits}, we see that
\begin{align*}
&\chi(Z_{1,\epsilon}, \ldots, Z_{n,\epsilon} \sqcup Z_{n+1,\epsilon}, \ldots, Z_{n+m,\epsilon})\\
&= (n+m) \log(\sqrt{1 + \epsilon}) + \chi\left(\frac{1}{\sqrt{1+\epsilon}}Z_{1,\epsilon}, \ldots, \frac{1}{\sqrt{1+\epsilon}}Z_{n,\epsilon} \sqcup \frac{1}{\sqrt{1+\epsilon}}Z_{n+1,\epsilon}, \ldots, \frac{1}{\sqrt{1+\epsilon}}Z_{n+m,\epsilon}\right) \\
&=  \frac{n+m}{2} \log(1 + \epsilon) + \frac{n+m}{2} \log(2\pi e) + \frac{1}{2} \log\left(   \det\left( \left(1 - \frac{1}{1+\epsilon}\right) I_{n+m} +    \frac{1}{1 + \epsilon}A  \right)\right) \\
&=  \frac{n+m}{2} \log(2\pi e) + \frac{1}{2} \log\left(   \det\left( \epsilon I_{n+m} +     A  \right)\right).
\end{align*}
As $A$ is a positive matrix and thus diagonalizable, we know that
\[
\det\left( \epsilon I_{n+m} +     A  \right) = \epsilon^{\mathrm{nullity}(A)} p(\epsilon)
\]
where $p$ is a polynomial of degree $\mathrm{rank}(A)$ with real coefficients that does not vanish at 0.  Consequently, we obtain that
\begin{align*}
&\delta( S_1, \ldots, S_n \sqcup S_{n+1}, \ldots, S_{n+m} ) \\
&= n+m + \limsup_{\epsilon \to 0^+} \frac{\frac{n+m}{2} \log(2\pi e) + \frac{1}{2} \log(\epsilon^{\mathrm{nullity}(A)}p(\epsilon))}{|\log(\sqrt\epsilon)|} \\
&= n+m + \limsup_{\epsilon \to 0^+} \frac{\frac{n+m}{2} \log(2\pi e) + \frac12\mathrm{nullity}(A) \log(\epsilon) + \frac{1}{2}  \log(p(\epsilon))}{|\log(\sqrt\epsilon)|} \\
&= n+m - \mathrm{nullity}(A) = \mathrm{rank}(A)
\end{align*}
as desired.
\end{proof}

\begin{rem}
Let $(S, T)$ be a bi-free central limit distribution with variances 1 and covariance $c \in [-1,1]$.  Hence
\[
\delta(S \sqcup T) = \begin{cases}
2 & \text{if } c \neq\pm 1 \\
1 & \text{if } c =\pm 1 \\
\end{cases}.
\]
In particular, the support of the joint distribution of $(S, T)$ has dimension $\delta^*(S\sqcup T)$: indeed, if $c \neq \pm 1$ then $(S, T)$ has joint distribution with support $[-2, 2]^2 \subset \bR^2$ by \cite{HW2016}, while otherwise it is supported on the line $y = cx$.
This adds validation to the name ``bi-free microstate entropy dimension''.
\end{rem}

\section{Microstate Bi-Free Entropy for Non-Bi-Partite Systems}
\label{sec:Gen}
In the section, we will discuss our notion of microstate bi-free entropy to non-bi-partite systems where further complications arise.
To do this, we will find it useful to take an approach from operator-valued bi-free probability.
We refer the reader to \cite{CNS2015-1} rather than reintroduce the entire setting here.

Let $(\C, \varphi)$ be a non-commutative probability space and let $B$ be a unital algebra.
Then $\C\otimes B$ can be viewed as a $B$-$B$-bimodule where
\[
	b \cdot (a \otimes b') = a \otimes bb', \qqand (a \otimes b') \cdot b = a \otimes b'b
\]
for $b, b' \in B$ and $a \in \C$.
Let us denote by $L_b$ and $R_b$ the left and right actions of $b$ above.
If $p_B : \C\otimes B \to B$ is defined by
\[
	p_B(a \otimes b) = \varphi(a) b,
\]
then $\L(\C \otimes B)$ is a $B$-$B$-non-commutative probability space with left and right $B$-operators $L_b$ and $R_b$ respectively and expectation $E : \L(\C \otimes B) \to B$ defined by
\[
	E(Z) = p_B(Z(1_\C \otimes 1_B))
\]
for all $Z \in \L(\C \otimes B)$.
Let $\L(\C \otimes B)_\ell$ denote all elements of $\L(\C \otimes B)$ that commute with elements of $\{R_b \, \mid \, b \in B\}$ and let $\L(\C \otimes B)_r$ denote all elements of $\L(\C \otimes B)$ that commute with elements of $\{L_b \, \mid \, b \in B\}$.  Therefore, if $X, Y \in \C$ and $b \in B$, we can define $L(X \otimes b) \in \L(\C \otimes B)_\ell$ and $R(Y \otimes b) \in \L(\C \otimes B)_r$ via
\[
	L(X \otimes b)(a \otimes b') = Xa \otimes bb' \qqand R(Y \otimes b)(a \otimes b') = Ya \otimes b'b.
\]
for all $a \in \C$ and $b' \in B$.

Our current approach to matricial microstates has been to find matrices in $\M_d$ for which the moments of the appropriate left or right multiplication operators computed against $\tau_d(\cdot I_d)$ have been approximately correct.
Since $\L(\M_d) \cong \M_d \otimes \M_d^\op$ via $L(A)R(B) \mapsto A\otimes B^\op$, we may view this in the above setting with $\C = \bC$.
If we replace $\bC$ by some larger matrix algebra, we introduce non-commutativity between the left and the right approximates.

Indeed for fixed $d_1, d_2 \in \bN$, if we identify $\L(\M_{d_1}\otimes \M_{d_2}) \cong \L(\M_{d_1}) \otimes \L(\M_{d_2}) \cong \L(\M_{d_1}) \otimes \M_{d_2} \otimes \M_{d_2}^\op$, we find
\[
	\L(\M_{d_1}\otimes \M_{d_2})_\ell \cong \L(\M_{d_1}) \otimes \M_{d_2} \otimes \bC \qqand \L(\M_{d_1}\otimes \M_{d_2})_r \cong \L(\M_{d_1}) \otimes \bC \otimes \M_{d_2}^\op.
\]
In particular, the pair of faces $(L(\M_{d_1}\otimes \M_{d_2}), R(\M_{d_1}\otimes \M_{d_2}))$ in $\L(\M_{d_1}\otimes \M_{d_2})$ is isomorphic to the pair of faces
\[
	\paren{(\M_{d_1}\otimes \M_{d_2}\otimes \bC, \M_{d_1}\otimes \bC \otimes \M_{d_2}^\op)}
\]
in $\M_{d_1}\otimes \M_{d_2}\otimes \M_{d_2}^\op$.  Note that the state becomes $\tau_{d_1}\otimes(\tau_{d_2}\circ m)$, where $m(B_1\otimes B_2^\op) = B_1B_2$ is the multiplication map.
These faces are each as measure spaces isomorphic to $\M_{d_1}\otimes \M_{d_2}$.
We have for $A\otimes B \in \M_{d_1} \otimes \M_{d_2}$, $L(A\otimes B) = A\otimes B\otimes I_{d_2}$ while $R(A\otimes B) = A\otimes I_{d_2} \otimes B^\op$.

Using the above constructions, we postulate the following generalization of our microstate bi-free entropy to the non-tracially bi-partite setting.
Let $(\A, \varphi)$ be a C$^*$-non-commutative probability space and let $X_1,\ldots, X_n, Y_1, \ldots, Y_m$ be self-adjoint operators in $\A$, where we will consider $X_1, \ldots, X_n$ as left variables and $Y_1, \ldots, Y_m$ as right variables.
We desire to approximate $X_1, \ldots, X_n$ by $A_1, \ldots, A_n \in (\M_{d_1}\otimes\M_{d_2}\otimes \bC)^\sa$, and $Y_1, \ldots, Y_m$ by $B_1, \ldots, B_m \in (\M_{d_1}\otimes\bC\otimes \M_{d_2}^\op)^\sa$.
For $M,d \in \bN$ and $R, \epsilon> 0$, let $\Gamma_R(\smbfe; M, d_1, d_2, \epsilon)$ denote the set of all $(n+m)$-tuples $(A_1, \ldots, A_n, B_1, \ldots, B_m) \in ((M_{d_1} \otimes M_{d_2})^\sa)^{n+m}$ such that $\left\|A_i\right\|, \left\|B_j\right\| \leq R$ for all $1 \leq i \leq n$ and $1 \leq j \leq m$, and
\[
	\left|\varphi(Z_{k_1} \cdots Z_{k_p}) - \tau_{d_1}\otimes(\tau_{d_2}\circ m)(C_{k_1} \cdots C_{k_p})\right| < \epsilon
\]
for all $i_1, \ldots, i_p \in \{1,\ldots, n+m\}$ and $1 \leq p \leq M$ where 
\[
Z_{k} = \begin{cases}
X_k & \text{if } k \in\{1\ldots, n\} \\
Y_{k-n} & \text{if }k \in \{n+1,\ldots, n+m\}
\end{cases}
\qqand 
C_{k} = \begin{cases}
L(A_k) & \text{if } k \in\{1\ldots, n\} \\
R(B_{k-n}) & \text{if }k \in \{n+1,\ldots, n+m\}
\end{cases}.
\] 

\begin{defn}
	\label{defn:micro-bi-free-gen}
	Using the above notation, if $\lambda_{d_1d_2,n+m}$ denotes the Lebesgue measure on $(\M_{d_1d_2}^{\sa})^{n+m}$, define
	\begin{align*}
		\chi_R(\smbfe; M, d_1, d_2, \epsilon) &= \log\left( \lambda_{d_1d_2,n+m}\left(\Gamma_R(\smbfe; M, d_1, d_2, \epsilon)  \right)   \right) \\
		\chi_R(\smbfe; M, \epsilon) &= \limsup_{(d_1, d_2) \to \infty} \frac{1}{(d_1d_2)^2}\chi_R(\smbfe; M, d_1, d_2, \epsilon) \\
		& \qquad + \frac{1}{2} (n+m) \log(d_1d_2) \\
		\chi_R(\smbfe) &=  \inf\{\chi_R(\smbfe; M, \epsilon) \, \mid \, M \in \bN, \epsilon > 0\} \\
		\chi(\smbfe) &= \sup_{R > 0} \chi_R(\smbfe).
	\end{align*}
	The quantity $\chi(\smbfe) \in [-\infty, \infty)$ will be called the \emph{microstate bi-free entropy of $X_1, \ldots, X_n \sqcup Y_1, \ldots, Y_m$}.
	\end{defn}

\begin{rem}
Of course, one must specify what is meant by $\limsup_{(d_1, d_2) \to \infty}$.
There are many possible definitions (i.e. $d_1 + d_2 \geq K$ for sufficiently large $K$ or $\min\{d_1, d_2\} \geq K$ for sufficiently large $K$).
It may even be possible that if $d_2$ is sufficiently large, then there is no difference using $d_1 = 2$ or $d_1 > 2$.
Of course, the real question is, ``How do the sets $\Gamma_R(\smbfe; M, d_1, d_2, \epsilon)$ behave as $d_1$ and $d_2$ vary?''
\end{rem}

\begin{rem}
We have seen above that the $d_1 = 1$ case may only model tracially bi-partite systems.
However, adding the flexibility that $d_1 > 1$ appears to reduce these restrictions.
Specifically, the only obvious restriction is that $\tau_{d_1}\otimes(\tau_{d_2}\circ m)$ is self-adjoint so we may only approximate distributions of left and right operators with respect to self-adjoint states.
This is not a cumbersome restriction since we are already assuming that the operators and $\varphi$ are self-adjoint.
\end{rem}

\begin{rem}
	We note that many results in this paper may be simply extended to apply to Definition \ref{defn:micro-bi-free-gen}; specifically Proposition \ref{prop:micro-subadditive}, Proposition \ref{prop:not-plus-infinity}, Proposition \ref{prop:R-does-not-matter}, Proposition \ref{prop:transforms}, and computations like those in Theorem \ref{thm:entropy-pair-of-semis}.
	(Note, though, that the computations in Theorem~\ref{thm:entropy-pair-of-semis} provided an upper bound on the entropy of a semicircular system, while the lower bound came from Theorem~\ref{thm:micro-converting-rights-to-lefts} which does not have an analogue in this setting.)
	One point of interest is there is less of a connection between these bi-free microstates and known free microstates (e.g. the argument used in the proof of Theorem \ref{thm:micro-converting-rights-to-lefts} and Theorem \ref{thm:micro-bi-free-additive} are no longer clear).
	Of course knowledge that $\limsup$ can be replaced with $\liminf$ in the definition of microstate free entropy (Definition \ref{defn:micro-free}) immediately implies the quantities in Definition \ref{defn:micro-bi-free-gen} agree with those in Definition \ref{defn:micro-free} when $n = 0$ or $m = 0$.
	However, when $n,m > 1$, it is not clear how to obtain these generalized bi-free microstates from free microstates of the left variables and microstates of the right variables.
\end{rem}

\begin{rem}
	We have made the choice to find approximates in the algebra $\M_{d_1}\otimes \M_{d_2}\otimes \M_{d_2}^\op$.
	One may consider replacing $\M_{d_1}$ by some other algebra -- possibly of infinite dimension -- to allow more flexibility.
	While it then becomes easier to find approximates, it becomes less clear how to treat the measure of the set of approximates.
	Nonetheless, \cite{S2016-3} has argued that this is the correct constructs for the bi-free analogue of random matrices and thus the correct construct for bi-free microstates.
\end{rem}

\section{Open Questions}
\label{sec:Ques}

We conclude this paper with several important and intriguing questions raised in this paper in addition to the question of whether results in bi-free entropy may be applied to obtain results pertaining to von Neumann algebras.

To begin, as we are dealing with tracially bi-partite systems, one of the most natural questions is the following.
\begin{ques}
\label{ques:left-to-right-always-works}
Given a tracially bi-partite family of operators $\left(\{X_i\}^n_{i=1}, \{Y_j\}^m_{j=1}\right)$, is there always a single-sided version as in Theorem~\ref{thm:micro-converting-rights-to-lefts} for which the stated inequality is an equality?
If not, does taking a supremum over all systems which may stand on the left hand side lead to equality?
\end{ques}
One can produce examples by making a ``poor choice'' where the inequality is strict: for example, if $(X, Y)$ is a pair of classically independent semi-circular operators (of non-zero variance), letting $X'$ and $Y'$ in the parlance of that theorem merely be $X$ and $Y$ themselves, the hypotheses of the theorem are satisfied and
$$-\infty = \chi(X, Y) < \chi(X \sqcup Y).$$

The answer to Question \ref{ques:left-to-right-always-works} is affirmative for the bi-free central limit distributions and for independent distributions.
A general answer to Question \ref{ques:left-to-right-always-works} would be of interest as it directly relates the free and bi-free non-microstate entropies and could answer the following.

\begin{ques}
\label{ques:transform-lefts-and-rights}
Is there an analogue of Proposition \ref{prop:transforms} where the transformation can intermingle left and right variables simultaneously?
\end{ques}
Of course Question \ref{ques:transform-lefts-and-rights} would be of interest as it would provide a greater flexibility in handling this entropy theory.  However, there have been no instances in bi-free probability where right operators can intermingle with left operators and the resulting operators still behaves like left operators.

Question \ref{ques:left-to-right-always-works} also relates to the following question.
\begin{ques}
\label{ques:integration}
Let $(X,Y)$ be a bi-partite pair with joint distribution $\mu$.  Is there an integration formula involving $\mu$ to compute $\chi(X \sqcup Y)$?
\end{ques}
Question \ref{ques:integration} arises from the integration formula established in \cite{V1994}: if $X$ is a self-adjoint operator with distribution $\mu$, then the free entropy of $X$ is
\[
\chi(X) = \frac{1}{2} \log(2\pi) + \frac{3}{4} + \int_\bR \int_\bR \log|s-t| \, d\mu(s) \, d\mu(t).
\]
To determine $\chi(X \sqcup Y)$ for a tracially bi-partite pair $(X, Y)$, one must understand the microstates $(A, B) \in (\M_d(\bC)^{\sa})^2$ that are good approximates for $(X, Y)$.  If $(A', B') \in (\M_d(\bC)^{\sa})^2$ is another microstate that is a good approximate of $(X, Y)$, then \cite{V1993} implies that $\left\|A' - A\right\|_2$ and $\left\|B' - B\right\|_2$ are small.  Therefore, for any $n,m \in \bN$ and any unitary $U \in \M_d(\bC)$ we have that $\left\|A^n U^*B^m U - (A')^n U^* (B')^m U\right\|_2$ is small in norm (as the operator norm of microstates will be bounded by some $R$).  Therefore, an understanding of microstates of the pair $(X, Y)$ can be reduced to understanding the vector-valued random variable on the unitary group of $\M_d(\bC)$ defined by
\[
U \mapsto (A^{n_1} U^*B^{m_1} U, \ldots, A^{n_k} U^*B^{m_k} U)
\]
for every $k \in \bN$ and every distinct $(n_1, m_1), \ldots, (n_k, m_k) \in \bN^2$.  When $k = 1$, the characteristic function of this random variable may be computable using the Harish-Chandra-Itzykson-Zuber integral formula, but deriving the necessary information from the characteristic function to describe the microstate bi-free entropy appears difficult.

Of course, an affirmative answer to both Questions \ref{ques:left-to-right-always-works} and \ref{ques:integration} would enable the computation of the microstate free entropy of certain pairs of self-adjoint operators via an integration formula.  Thus we do not expect an affirmative answer to both Questions \ref{ques:left-to-right-always-works} and \ref{ques:integration}.

Other natural questions pertaining to this microstate bi-free entropy are
\begin{ques}
Is Theorem \ref{thm:micro-bi-free-additive} true without the $\limsup_{d \to \infty}$ being a limit condition?
\end{ques}
which clearly will follow from
\begin{ques}
Can $\limsup$ be replaced with $\liminf$ in Definition \ref{defn:micro-bi-free}?
\end{ques}
As these questions have been extremely difficult even in the free setting, we presume they will have equal if not greater difficulty in the bi-free setting.  Another natural question to extend to the bi-free setting is the following.
\begin{ques}
Under the assumption the operators under consideration have microstates, does the microstate bi-free entropy from \cite{CS2017} agree with the above non-microstate bi-free entropy for  tracially bi-partite collections? 
\end{ques}
In the free setting, \cite{BCG2003} first showed that the microstate free entropy is always less than the non-microstate free entropy.  Thus perhaps a good starting point would be a bi-free version of \cite{BCG2003}.  Of course much progress was made towards the converse in \cite{D2016}.

Although our proofs do not require the following, it would be nice to be able to answer the following question:
\begin{ques}
Is the bi-free analogue of the Wasserstein metric actually a metric?
\end{ques}

Finally, as most of this paper deals only with the tracially bi-partite setting, we ask the following.
\begin{ques}
\label{ques:non-bi-partite}
Are the quantities in Definition \ref{defn:micro-bi-free-gen} finite when $n,m > 0$?  Furthermore, does Definition \ref{defn:micro-bi-free-gen} agree with Definition \ref{defn:micro-bi-free} for tracially bi-partite systems?
\end{ques}
An answer to Question \ref{ques:non-bi-partite} would enable us to extend the notion of microstate bi-free entropy to non-bi-partite systems thereby allowing a richer theory and demonstrating the notions in this paper are the correct extensions of microstate free entropy to the bi-free setting.

\section*{Acknowledgements}

The authors would like to thank Dimitri Shlyakhtenko for many conversations and ideas that led to the computations in Section \ref{sec:Calc}.


\begin{thebibliography}{13}


\bib{BB2003}{article}{
    author={Belinschi, S.T.},
    author={Bercovivi, H.},
    title={A Property of Free Entropy},
	journal={Pac. J. Math},
	volume={211},
	number={1},
    year={2003},
    pages={35-40}
}



\bib{BBGS2017}{article}{
    author={Belinschi, S.},
    author={Bercovici, H.},
    author={Gu, Y.},
    author={Skoufranis, P.},
    title={Analytic subordination for bi-free convolution},
    journal={J. Funct. Anal.},
    year={2018},
    number={4},
    volume={275},
    pages={926-966}
}



\bib{BCG2003}{article}{
    author={Biane, P.},
    author={Capitaine, M.},
    author={Guionnet, A.},
    title={Large deviation bounds for matrix Brownian motion},
    journal={Invent. math.},
    year={2003},
    volume={152},
    number={2},
    pages={433-459}
}




\bib{BD2013}{article}{
    author={Biane, P.},
    author={Dabrowski, Y.},
    title={Concavification of free entropy},
    journal={Adv. Math.},
    year={2013},
    volume={234},
    pages={667-696}
}




\bib{BV2001}{article}{
    author={Biane, P.},
    author={Voiculescu, D.},
    title={A free probability analogue of the Wasserstein metric on the trace-state space},
    journal={GAFA},
    year={2001},
    volume={11},
    number={6},
    pages={1125-1138}
}



\bib{C2016}{article}{
    author={Charlesworth, I.},
    title={An alternating moment condition for bi-freeness},
    journal={Adv. Math.},
    year={2019},
    volume={346},
    number={13},
    pages={546-568}
}


\bib{CNS2015-1}{article}{
    author={Charlesworth, I.},
    author={Nelson, B.},
    author={Skoufranis, P.},
    title={Combinatorics of Bi-Free Probability with Amalgamation},
	journal={Comm. Math. Phys.},
	volume={338},
	number={2},
    year={2015},
    pages={801-847}
}




\bib{CNS2015-2}{article}{
    author={Charlesworth, I.},
    author={Nelson, B.},
    author={Skoufranis, P.},
    title={On two-faced families of non-commutative random variables},
    journal={Canad. J. Math.},
    volume={67},
    number={6},
    year={2015},
    pages={1290-1325}
}
 

\bib{CS2017}{article}{
    author={Charlesworth, I.},
    author={Skoufranis, P.},
    title={Analogues of Entropy in Bi-Free Probability Theory: Non-Microstate},
	eprint={preprint},
    year={2017},
    pages={46}
}
 

\bib{D2016}{article}{
    author={Dabrowski, Y.},
	title={A Laplace Principle for Hermitian Brownian Motion and Free Entropy I: the convex functional case},
    eprint={arXiv:1604.06420},
    year={2017},
    pages={83}
}



\bib{G1997}{article}{
    author={Ge, L.},
    title={Applications of free entropy to finite von Neumann algebras},
	journal={Am. J. Math},
    year={1997},
	pages={467-485},
}



\bib{G1998}{article}{
    author={Ge, L.},
    title={Applications of free entropy to finite von Neumann algebras, II},
	journal={Ann. Math},
    year={1998},
    volume={147},
    number={1},
    pages={143-157}
}

 

\bib{HP2006}{article}{
    author={Hiai, F.},
    author={Petz, D.},
    title={The semicircle law, free random variables and entropy},
    publisher={American Mathematical Society},
    volume={77},
    series={Mathematical Surveys and Monographs},
    year={2006}
}
 

 
\bib{HMU2009}{article}{
    author={Hiai, F.},
    author={Miyamoto, T.},
    author={Ueda, Y.},
    title={Orbital approach to microstate free entropy},
    journal={Int. J. Math},
    year={2009},
    number={2},
    volume={20},
    pages={227-273}
}



		\bib{HW2016}{article}{
			author={Huang, H.-W.},
			author={Wang, J.-C.},
			title={Analytic Aspects of Bi-Free Partial $R$-Transforms},
			journal={J. Funct. Anal.},
			volume={271},
			number={4},
			date={2016},
			pages={922-957}
		}

\bib{J2007}{article}{
    author={Jung, K.},
    title={Amenability, tubularity, and embeddings into $\mathcal{R}^\omega$},
    journal={Math. Ann.},
    year={2007},
    volume={228},
    pages={241-248}
}


\bib{OV2000}{article}{
			author={Otto, F.},
			author={Villani, C.},
			title={Generalization of an inequality by Talagrand and links with the logarithmic Sobolev inequality},
			journal={J. Funct. Anal.},
			volume={173},
			number={2},
			date={2000},
			pages={361-400}
}

 




\bib{R2017}{article}{
    author={Ramsey, C.},
    title={Faithfulness of bi-free product states},
	journal={Proc. Amer. Math. Soc.},
	volume={146},
    year={2018},
    pages={5279-5288}
}





 

\bib{S2016-2}{article}{
    author={Skoufranis, P.},
    title={Independences and Partial $R$-Transforms in Bi-Free Probability},
    journal={Ann. Inst. Henri Poincar\'{e} Probab. Stat.},
    volume={52},
    number={3},
    pages={1437-1473},
    year={2016}
}



\bib{S2016-3}{article}{
    author={Skoufranis, P.},
    title={On Operator-Valued Bi-Free Distributions},
    journal={Adv. Math.},
    volume={303},
    year={2016},
    pages={638-715}
}




\bib{S2016-4}{article}{
    author={Skoufranis, P.},
    title={Some Bi-Matrix Models for Bi-Free Limit Distributions},
    journal={Indiana Univ. Math. J.},
    year={2017},
    number={5},
    volume={66},
    pages={1755-1795}
}



\bib{U2014}{article}{
    author={Ueda, Y.},
    title={Orbital Free Entropy, Revisited},
    journal={Indiana Univ. Math. J.},
    year={2014},
    number={2},
    volume={63},
    pages={551-577}
}



\bib{U2017}{article}{
    author={Ueda, Y.},
    title={A remark on orbital free entropy},
    journal={Arch. Math.},
    year={2017},
    volume={108},
    pages={629-638}
}





\bib{V1991}{article}{
    author={Voiculescu, D. V.},
    title={Limit laws for random matrices and free products},
    journal={Invent. Math. },
    number={1},
    volume={104},
    year={1991},
    pages={201-220}
}



\bib{V1993}{article}{
    author={Voiculescu, D. V.},
    title={The analogues of entropy and of Fisher's information measure in free probability theory, I},
    journal={Comm. Math. Phys.},
    volume={155},
    number={1},
    year={1993},
    pages={71-92}
}


\bib{V1994}{article}{
    author={Voiculescu, D. V.},
    title={The analogues of entropy and of Fisher's information measure in free probability theory, II},
    journal={Invent. Math.},
    volume={118},
    number={1},
    year={1994},
    pages={411-440}
}



\bib{V1996}{article}{
    author={Voiculescu, D. V.},
    title={The analogues of entropy and of Fisher's information measure in free probability theory III: The absence of Cartan subalgebras},
    journal={Geom. Funct. Anal.},
    volume={6},
    number={1},
    year={1996},
    pages={172-199}
}

\bib{V1997}{article}{
    author={Voiculescu, D. V.},
    title={The analogues of entropy and of Fisher's information measure in free probability theory, IV: Maximum entropy and freeness},
    journal={Free Probability Theory},
    volume={12},
    year={1997},
    pages={293-302}
}



\bib{V1998-1}{article}{
    author={Voiculescu, D. V.},
    title={A strengthened asymptotic freeness result for random matrices with applications to free entropy},
    journal={Internat. Math. Res. Notices},
    volume={1998},
    number={1},
    year={1998},
    pages={41-63}
}




\bib{V1998-2}{article}{
    author={Voiculescu, D. V.},
    title={The analogues of entropy and of Fisher's information measure in free probability theory, V: Noncommutative Hilbert transforms},
    journal={Invent. Math.},
    volume={132},
    number={1},
    year={1998},
    pages={189-227}
}









\bib{V1999}{article}{
    author={Voiculescu, D. V.},
    title={The analogues of entropy and of Fisher's information measure in free probability theory: VI. Liberation and Mutual Free Information},
    journal={Adv. Math.},
    volume={146},
    number={2},
    year={1999},
    pages={101-166}
}



\bib{V2014}{article}{
    author={Voiculescu, D. V.},
    title={Free probability for pairs of faces I},
    journal={Comm. Math. Phys.},
    year={2014},
    volume={332},
    pages={955-980}
}








\end{thebibliography}
\end{document}